\documentclass[11pt]{amsart}
\usepackage{amsmath,amsthm,graphicx,amscd,multirow}
\usepackage{ amssymb} 
\usepackage{ mathrsfs} 
\allowdisplaybreaks[3]
\usepackage[dvipdfmx]{xcolor}
\usepackage{ascmac}
\usepackage[arrow,matrix]{xy}
\usepackage{enumerate}

\usepackage{enumitem}

\theoremstyle{theorem}
\newtheorem{thm}{Theorem}[section]
\newtheorem{prop}[thm]{Proposition}
\newtheorem{cor}[thm]{Corollary}
\newtheorem{lem}[thm]{Lemma}

\numberwithin{equation}{section}

\theoremstyle{definition}
\newtheorem{dfn}[thm]{Definition}

\newtheorem{conj}[thm]{Conjecture}

\theoremstyle{remark}
\newtheorem{rem}[thm]{Remark}

\makeatletter
\@namedef{subjclassname@2020}{%
\textup{2020} Mathematics Subject Classification}
\makeatother

\title{A motivic interpretation of Whittaker periods for $\mathrm{GL}_n$}
\author{Takashi Hara}
\author{Kenichi Namikawa}
\address{ Department of Mathematics, College of Liberal Arts,  Tsuda University, 2-1-1 Tsuda-machi, Kodaira City, Tokyo 187-8577, Japan }
\email{t-hara@tsuda.ac.jp}

\address{ Department of Mathematics, School of System Design and Technology, Tokyo Denki University, 5 Senju Asahi-cho, Adachi-ku, Tokyo 120-8551, Japan }
\email{namikawa@mail.dendai.ac.jp}

\subjclass[2020]{Primary 11F67,  Secondary 11F75, 11F70}

\begin{document}

\begin{abstract}
Admitting the existence of conjectural motives attached to cohomological irreducible cuspidal automorphic representations of ${\rm GL}_n$, 
we write down Raghuram and Shahidi's Whittaker periods in terms of Yoshida's fundamental periods when the base field is a totally real number field or a CM field.
\end{abstract}

\maketitle

\tableofcontents

\section{Introduction}

In recent years, the terminology {\em periods} is used to denote (mainly transcendental) complex numbers obtained as period integrals, that is, integrals of algebraic differential forms on certain algebraically defined domains. As Kontsevich and Zagier suggested, study of properties of periods themselves is of course important and attractive, but under number theoretical situations, periods are also subjects of great interest in that they are often closely linked to special values of appropriate  $L$-functions. Meanwhile the notion of periods became diversified up to the present, many people have introduced various kinds of periods even for the same arithmetic objects, adopting rather different techniques. Hence elucidation of the relations among various periods is a quite important issue in the study of periods. In the present article, admitting several suitable conjectures, we shall  derive an explicit formula relating two periods of different types defined for automorphic representations of the general linear group defined over a totally real field or a CM field: one is based on automorphic construction, whereas the other has a motivic background. For that purpose, we let critical values of the Rankin--Selberg $L$-functions intermediate between them.  

Let us explain details of the present article. For a general number field $\mathsf{F}$, let $\pi^{(n)}$ denote a cohomological irreducible cuspidal automorphic representation $\pi^{(n)}$ of $\mathrm{GL}_n(\mathsf{F}_{\mathbf{A}})$. In \cite[Definition/Proposition 3.3]{rs08}, Raghuram and Shahidi associate to the pair $\pi^{(n)}$ and an appropriate signature $\varepsilon^{(n)}\in \{\pm 1\}^{\Sigma_{\mathsf{F},\mathbf{R}}}$ a complex number $p(\pi^{(n)},\varepsilon^{(n)})$, which we call the {\em Whittaker period} for $\pi^{(n)}$ and $\varepsilon^{(n)}$. There they have made a choice of a generator $[\pi_\infty^{(n)}]^{\varepsilon^{(n)}}$ of the $1$-dimensional space $H^{\mathrm{b}_{n}}(\mathfrak{g}_{n,\infty}, K^{(n)}_\infty ; \pi^{(n)}_\infty \otimes \mathcal{V}(\boldsymbol{\mu}^{(n),\vee}))[\varepsilon^{(n)}]$ (refer to Section~\ref{sec:CohAuto} for undefined notation), on which the Whittaker period $p(\pi^{(n)}, \varepsilon^{(n)})$ implicitly depends.  For a fixed generator $[\pi_\infty^{(n)}]^{\varepsilon^{(n)}}$, the Whittaker period $p(\pi^{(n)},\varepsilon^{(n)})$ is determined uniquely up to an algebraic multiple.
   Later Raghuram shows in \cite[Theorem 1.1]{rag16} that
   any critical value $L(m+\frac{1}{2},\pi^{(n+1)}\times \pi^{(n)})$ of the Rankin--Selberg $L$-function is expressed, up to an algebraic multiple, as the product of $p(\pi^{(n+1)},\varepsilon^{(n+1)}_m)$, $p(\pi^{(n)},\varepsilon^{(n)}_m)$ and a certain constant $p_\infty(m,\pi^{(n+1)}_\infty, \pi^{(n)}_\infty)$ coming from an archimedean zeta integral concerning the fixed cohomology classes $[\pi_\infty^{(n+1)}]^{\varepsilon^{(n+1)}}$ and $[\pi_\infty^{(n)}]^{\varepsilon^{(n)}}$. 
   Note that (one of) the signatures $\varepsilon^{(n+1)}_m$, $\varepsilon^{(n)}_m$ and the constant $p_\infty(m,\pi^{(n+1)}_\infty, \pi^{(n)}_\infty)$ depend on the critical point $m$.
See \cite[Section 2.5.3.6]{rag16} for details on the constant $p_\infty(m,\pi^{(n+1)}_\infty, \pi^{(n)}_\infty)$, 
which is denoted as $p^{\epsilon, \eta}_\infty(\mu+m, \lambda)$ (when $n$ is odd) or $p^{\epsilon, \eta}_\infty(\mu, \lambda+m)$ (when $n$ is even) in \cite{rag16}.

When the base field $\mathsf{F}$ is the rational number field ${\mathbf Q}$, the authors have succeeded in suitably normalizing the generator $[\pi_\infty^{(n)}]^{\varepsilon^{(n)}}$ of the $(\mathfrak{gl}_n, K^{(n)}_\infty)$-cohomology group for $n=2$ and $3$ in \cite[Section~7.1]{hn}, by utilizing an {\em explicit Eichler--Shimura map} $\delta^{(n)}$ constructed as in \cite[Section~5.3]{hn}. Other than it, one of the key ingredients of our normalization procedure is precise and explicit study of the archimedean zeta integrals for the Rankin--Selberg convolution $\mathrm{GL}_3\times \mathrm{GL}_2$ developed by Hirano, Ishii and Miyazaki in \cite{him}. As a consequence, we can deduce that, for the normalized generators $[\pi_\infty^{(3)}]^{\varepsilon^{(3)}}$ and $[\pi_\infty^{(2)}]^{\varepsilon^{(2)}}$, the ``unknown constant''  
   $p_\infty(m, \pi^{(n+1)}_\infty, \pi^{(n)}_\infty)$ appearing in \cite{rag16} indeed coincides with the expected $\Gamma$-factor $L_\infty(m+\frac{1}{2},\pi^{(3)}\times \pi^{(2)})$ of the Rankin--Selberg $L$-function multiplied by a simple and explicit constant; see  \cite[Remark~7.3]{hn} for details. 
Recently Ishii and Miyazaki have made similar explicit calculation in \cite{im} on the archimedean zeta integrals for $\mathrm{GL}_{n+1}\times \mathrm{GL}_n$ with arbitrary $n$ when the base field is totally imaginary, and by using their results, one can also deduce a corresponding consequence on $p_\infty(m,\pi_\infty^{(n+1)},\pi_\infty^{(n)})$ in the case (this is an ongoing joint work with Tadashi Miyazaki).
From these results one observes that, admitting Clozel's conjecture \cite[Conjecture 4.5]{clo90} on the existence of the pure motive ${\mathcal M}[\pi^{(n)}]$ attached to $\pi^{(n)}$ and Deligne's conjecture \cite[Conjecture 1.8]{del79} on critical values of motivic $L$-functions,
the product $p(\pi^{(n+1)},\varepsilon^{(n+1)})p(\pi^{(n)},\varepsilon^{(n)})$ of the Whittaker periods coincides
with, up to an algebraic multiple, Deligne's period of the Rankin--Selberg motive  $c^+({\mathcal M}[\pi^{(n+1)}] \otimes_{\mathsf{F}} {\mathcal M}[\pi^{(n)}])$ multiplied by a suitable power of $2\pi\sqrt{-1}$ (when $\mathsf{F}$ is totally imaginary, we set $\varepsilon^{(n+1)}=\varepsilon^{(n)}=1$ as convention). 
In other words, these results provide a motivic interpretation to the {\em product} $p(\pi^{(n+1)},\varepsilon^{(n+1)}) p(\pi^{(n)},\varepsilon^{(n)})$ of Raghuram and Shahidi's Whittaker periods. 

The goal of the present article is to provide a motivic interpretation to the {\em single} Whittaker period $p(\pi^{(n)},\varepsilon^{(n)})$ when the base field $F$ is either a totally real field or a CM field, using Yoshida's fundamental periods introduced in \cite{yos01} and \cite{yos16}.   In the following, we continue to assume Clozel's and Deligne's conjectures, and let $\mathcal{M}[\pi^{(n)}]$ denote the (conjectural) pure motive attached to $\pi^{(n)}$. It is defined over $F$ with coefficients in a number field $E$ containing the field of rationality of $\pi^{(n)}$.
For each embedding  $\tau \colon F\hookrightarrow {\mathbf C}$ of $F$ and $\sigma\colon E\hookrightarrow \mathbf{C}$ of $E$, 
    write the Hodge decomposition of ${\mathcal M}[\pi^{(n)}]$  with respect to $\tau$ and $\sigma$ as 
    \begin{align*}
      H_{\rm B} ( {\mathcal M}[\pi^{(n)}]_\tau   )  \otimes_{E, \sigma} {\mathbf C} 
      = \bigoplus^n_{i=1}  H^{p^{\mathcal{M}_{\tau,\sigma}}_i, q^{\mathcal{M}_{\tau,\sigma}}_i}  (  {\mathcal M}[\pi^{(n)}]_\tau   )_\sigma.
    \end{align*} 
Hereafter the Grothendieck restriction of scalars of $\mathcal{M}[\pi^{(n)}]$ from the base field $F$ to its subfield $F'$ is denoted as $\mathrm{Res}_{F/F'} (\mathcal{M}[\pi^{(n)}])$.
Let $c^{\pm} ({\rm Res}_{F/{\mathbf Q}}  ( {\mathcal M}[\pi^{(n)}] )  )$ denote  Deligne's $\pm$-periods for ${\rm Res}_{F/{\mathbf Q}}  ( {\mathcal M}[\pi^{(n)}] )$, 
and put
\begin{align*}
 \lambda( {\mathcal M}[\pi^{(n)}]  )
 &= \sum_{\tau\colon  F\hookrightarrow {\mathbf C}} 
     \sum^n_{i=1} (n-i) p^{\mathcal{M}_{\tau,\sigma}}_i   , &
\widetilde{c}({\mathcal M}[\pi^{(n)}]) &=(\tilde{c}_\sigma(\mathcal{M}[\pi^{(n)}]))_{\sigma \colon E\hookrightarrow \mathbf{C}} \\  
 & & & =  \left( \prod^{ \lceil \frac{ n }{2}-1 \rceil }_{\beta=1} \prod_{\tau \colon F\hookrightarrow \mathbf{C}} c_{\beta,(\tau,\sigma)}    
       ( {\mathcal M}[\pi^{(n)}]    )\right)_{\sigma\colon E\hookrightarrow \mathbf{C}}.
\end{align*}

In fact $\lambda(\mathcal{M}[\pi^{(n)}])$ does not depend on the choice of $\sigma$ (see Proposition~\ref{prop:perprod}). For each $a, b \in \overline{\mathbf{Q}}\otimes_{\mathbf{Q}}{\mathbf C}$, we write $a\sim b$ if $a=bc$ holds for a nonzero algebraic number $c\in \overline{\mathbf Q}^\times$. The main result of the present article is as follows.

\begin{thm}\label{thm:mainintro} $(=$ Theorem~$\ref{thm:main})$
Suppose that all of the following three conjectures are fulfilled\,{\rm :}
\begin{itemize}
\item[--] Clozel's conjecture {\rm \cite[Conjecture 4.5]{clo90}} on the existence of the motive ${\mathcal M}[\pi^{(n)}]${\rm ;} 
\item[--] Deligne's conjecture on critical values \cite[Conjecture 1.8]{del79} for critical Tate twists of the tensored motive $\mathrm{Res}_{F/F^+}(\mathcal{M}[\pi^{(n+1)}])\otimes_{F^+} \mathrm{Res}_{F/F^+}({\mathcal M}[\pi^{(n)}])${\rm ;}
\item[--] Conjecture~$\ref{conj:mellin}$ which concerns archimedean zeta integrals $\widetilde{\mathcal I}_v(s, [\pi^{(n+1)}_v]^{\varepsilon^{(n+1)}},  [\pi^{(n)}_v]^{\varepsilon^{(n)}}  , \varphi_v)$ for each archimedean place $v$ of $F$. 
\end{itemize}
Further assume that $\pi^{(n)}$ satisfies the assumption $(\ref{eq:auto_assump})$, the purity weight $w(\pi^{(n)})$ is even,
 and the highest weight $\boldsymbol{\mu}^{(n)}_{\pi}$ attached to $\pi^{(n)}$ is in a good position $($see Definition~$\ref{dfn:good}$$)$. Then we have 
\begin{align*}
  \left(p({}^\sigma \pi^{(n)}, \varepsilon^{(n)})\right)_{\sigma\colon E \hookrightarrow \mathbf{C}} 
  \sim (2\pi \sqrt{-1})^{\lambda(\mathcal{M}[\pi^{(n)}])}\widetilde{c}(\mathcal{M}[\pi^{(n)}]) 
         \times 
         \begin{cases}  1 & \text{if $n$ is odd},  \\
                                c^{-\varepsilon^{(n)}}(\mathrm{Res}_{F/\mathbf{Q}}(\mathcal{M}[\pi^{(n)}]))  & \text{if $n$ is even}.
          \end{cases}
\end{align*}
\end{thm}

We here summarize a strategy to prove Theorem~\ref{thm:mainintro}.
The first step is to establish an {\em explicit factorization formula}, which describes Deligne's period of the tensor product 
$\mathrm{Res}_{F/F^+}({\mathcal M}) \otimes_{F^+} \mathrm{Res}_{F/F^+}({\mathcal N})$  
of pure motives $\mathrm{Res}_{F/F^+}(\mathcal{M})$ and $\mathrm{Res}_{F/F^+}(\mathcal{N})$ in terms of each individual fundamental periods of $\mathcal{M}$ and $\mathcal{N}$. 
Following Yoshida's work \cite[Proposition~12]{yos01}, Bhagwat has already studied such a factorization formula in \cite{bha15} when the base field is the rational number field ${\mathbf Q}$.
However we find that, if we impose the {\em $($strong$)$ interlace condition} on the Hodge types of $\mathcal{M}$ and $\mathcal{N}$, it produces a more explicit and precise factorization formula, which plays a crucial role in the verification of Theorem~\ref{thm:mainintro}. The (strong) interlace condition is formulated in Definition~\ref{def:interlace}, and the explicit factorization formula is proposed and proved in Proposition~\ref{prop:perdecomp}. We then study, as the second step of the proof, a certain local zeta integral $\widetilde{\mathcal I}_v(s, [\pi^{(n+1)}_v]^{\varepsilon^{(n+1)}},  [\pi^{(n)}_v]^{\varepsilon^{(n)}}  , \varphi_v)$ at each archimedean place $v$ of $F$, which appears as a result of the computation of the cohomological cup product (see Sections 6.2 and 6.3 for details) and concerns the fixed cohomology classes $[\pi_\infty^{(n+1)}]^{\varepsilon^{(n+1)}}$ and $[\pi_\infty^{(n)}]^{\varepsilon^{(n)}}$. On this local integral, we propose an important conjecture (Conjecture~\ref{conj:mellin}) which insists that $\widetilde{\mathcal{I}}_v(s,[\pi^{(n+1)}_v]^{\varepsilon^{(n+1)}},[\pi^{(n)}_v]^{\varepsilon^{(n)}},\varphi_v)$ essentially coincides with the $\Gamma$-factor $L_v(s,\pi^{(n+1)}\times \pi^{(n)}\times \varphi)$ of the  Rankin--Selberg $L$-function at $v$ if we can normalize $[\pi_\infty^{(n+1)}]^{\varepsilon^{(n+1)}}$ and $[\pi_\infty^{(n)}]^{\varepsilon^{(n)}}$ appropriately.
Combining these arguments,  we deduce the desired formula in Theorem~\ref{thm:mainintro} by induction on the degree $n$ of the general linear group. Here we use the result of Bhagwat and Raghuram \cite{br17} on the existence of an automorphic representation of prescribed infinity type.

\begin{rem}
The validity of Conjecture~\ref{conj:mellin} is closely related to explicit calculations of the archimedean zeta integrals for $\mathrm{GL}_{n+1}\times \mathrm{GL}_n$. When $n$ equals $2$, such calculations are thoroughly carried out by Hirano, Ishii and Miyazaki in \cite{him}. In our previous work \cite[Section 7.3.2]{hn}, limiting ourselves to the case where the base field $\mathsf{F}$ is the rational number field $\mathbf{Q}$, we normalize the cohomology classes $[\pi_\infty^{(3)}]^{\varepsilon^{(3)}}$ and $[\pi_\infty^{(2)}]^{\varepsilon^{(2)}}$ based on the results of \cite{him}, and succeed in extracting an explicit formula for $p(\pi^{(3)}, \varepsilon^{(3)})$ proposed in Theorem~\ref{thm:mainintro}.  
          Recently Ishii and Miyazaki have extended in \cite{im} calculations carried out in \cite{him} to the case where $n$ is an arbitrary natural number and the base field is totally imaginary. Combining the result of \cite{im} with explicit calculation using Gel'fand--Tsetlin type bases of irreducible algebraic representations of $\mathrm{GL}_n(\mathbf{C})$, we can normalize the cohomology classes $[\pi^{(n+1)}_\infty]^{\varepsilon^{(n+1)}}$ and $[\pi^{(n)}_\infty]^{\varepsilon^{(n)}}$ so that Conjecture~\ref{conj:mellin} holds also in this case. 
          See also Remark~\ref{rem:conj}.
\end{rem}

\begin{rem}
 When the base field $\mathsf{F}=F$ is a CM field, Harris and Lin propose another factorization formula on $c^+({\mathcal M}[\pi^{(n+1)}] \otimes_F {\mathcal M}[\pi^{(n)}])$ (not $c^+(\mathrm{Res}_{F/F^+}({\mathcal M}[\pi^{(n+1)}]) \otimes_{F^+} \mathrm{Res}_{F/F^+}({\mathcal M}[\pi^{(n)}]))$) in \cite[Proposition 3.13]{hl17} using  {\em motivic  periods} (also denoted as {\em quadratic periods} in some literature)  introduced by Harris in \cite{har97}. Note that Harris's motivic periods are defined only when the base field $F$ is a CM field (or the motive under consideration is polarized). 
Yoshida's fundamental periods and Harris's motivic periods are compared in \cite{yos16}; in particular the explicit formulas connecting them are obtained in \cite[(2.7), (2.8)]{yos16}. Our factorization formula (Proposition~\ref{prop:perdecomp}) then immediately follows from \cite[Proposition 3.13]{hl17} when $F$ is a CM field. 
 Furthermore Harris and Lin also obtain a result similar to Theorem~\ref{thm:mainintro} in \cite[(40)]{hl17}, but their method of proof is quite different from ours.  
Indeed, Harris and Lin take $\pi^{(n)}$ as the automorphic representation associated to Eisenstein series in \cite[Section 5]{hl17}, 
           in contrast to the present article where we always take $\pi^{(n)}$ as a cuspidal automorphic representation. 
           More precisely, Harris and Lin's proof relies on rationality of critical values of the standard $L$-functions for ${\rm GL}_n \times {\rm GL}_1$ due to Guerberoff (see \cite[Theorem 4.4]{hl17} for details).
           On the other hand, our proof relies on rationality of critical values of the Rankin--Selberg $L$-functions for ${\rm GL}_{n+1} \times {\rm GL}_n$ (due to Raghuram \cite{rag16}) and 
           Conjecture \ref{conj:mellin} (valid  when $F$ is totally imaginary due to Ishii and Miyazaki \cite{im}).
\end{rem}

The organization of this article is as follows. We prepare general notation in Section~\ref{sec:notation}, which will be used throughout the article.  
In Section~\ref{sec:mot}, we review Yoshida's theory on fundamental periods of motives according to \cite{yos01} and \cite{yos16};
here we also collect and prove several necessary statements. 
Section~\ref{sec:RSPer} is devoted to verify the explicit factorization formula for Deligne's period of the tensor product of motives satisfying the {\em $($strong$)$ interlace condition} (Proposition~\ref{prop:perdecomp}).
General settings and notation on cohomological cuspidal automorphic representations are prepared in  Section~\ref{sec:CohAuto}. In Section \ref{sec:RSMot}, we cohomologically interpret the critical values of the Rankin--Selberg $L$-functions as the cup products of certain cohomology classes, and define Raghuram and Shahidi's Whittaker periods. 
Finally we give the proof of Theorem~\ref{thm:mainintro} in Section~\ref{sec:MotInt}. 

\section{Notation} \label{sec:notation}

\subsection{Generality}

As usual, the symbols  $\mathbf{Q}$, $\mathbf{R}$ and $\mathbf{C}$ denote the rational number field, the real number field and the complex number field respectively. The complex conjugate on ${\mathbf C}$ is denoted as $\rho$. Throughout this article, we regard every number field (that is, every finite extension of $\mathbf{Q}$) as a subfield of the complex number field $\mathbf{C}$. We use the symbol $\iota_0$ for the canonical inclusion $E\hookrightarrow \mathbf{C}$ of an arbitrary subfield $E$ into $\mathbf{C}$, so as to distinguish it from many other embeddings which we shall also use. 

For a number field $\mathsf{F}$,  let $\Sigma_{\mathsf{F}}$ be the set of all places of $\mathsf{F}$
and $I_\mathsf{F}$ the set of all embeddings of $\mathsf{F}$ into ${\mathbf C}$.
We mainly consider that $\mathsf{F}$ is either a totally real number field $F^+$ or a CM field $F$, whose maximal totally subfield is denoted as $F^+$. Let $\Sigma_{\mathsf{F}, \infty}$ (resp.\ $\Sigma_{\mathsf{F}, {\rm fin}}$)  
        denote the set of archimedean (resp.\ finite) places of $\mathsf{F}$. 
We sometimes regard $\Sigma_{\mathsf{F},\infty}$ as a subset of $I_\mathsf{F}$, by choosing either of two embeddings associated to each complex place; see Section~\ref{sec:t-per} for details.
We use the symbols $\Sigma_{\mathsf{F}, {\mathbf R}}$ (resp.\  $\Sigma_{\mathsf{F}, {\mathbf C}}$) for the subsets of $\Sigma_{\mathsf{F},\infty}$ consisting of all real (resp.\ complex) places. 
The normal (or equivalently Galois) closure of $\mathsf{F}$ over $\mathbf{Q}$ is denoted as $\mathsf{F}^{\mathrm{gal}}$.

We let $\mathsf{F}_{\mathbf A}$ denote the ring of ad\`eles of a number field $\mathsf{F}$, and 
use the symbol $\mathsf{F}_{\mathbf A, {\rm fin}}$ for its subring of finite ad\`eles. The archimedean part of $\mathsf{F}_{\mathbf{A}}$ is denoted as $\mathsf{F}_{\mathbf A, \infty} = \mathsf{F}\otimes_{\mathbf{Q}} \mathbf{R}= \prod_{v\in \Sigma_{\mathsf{F}, \infty}} \mathsf{F}_v$. 
We use similar notation for every ad\`elic object.

For every $a, b\in {\mathbf C}$ and an algebraic extension $E$ of $\mathbf{Q}$ (regarded as an subfield of $\mathbf{C}$), we write $a\sim_E b$ if $a=bc$ holds for a nonzero element $c\in E^\times$. We abbreviate $a\sim_{\overline{\mathbf{Q}}}b$ as $a\sim b$ for simplicity. 


\subsection{Subgroups of $\mathrm{GL}_n$ at the archimedean part} \label{sec:grinfty}

We define the orthogonal group ${\rm O}(n)$ and the unitary group ${\rm U}(n)$ of degree $n$ as
\begin{align*}
{\rm O}(n)&:=\{g\in {\rm GL}_n(\mathbf{R}) \mid {}^{\rm t}gg=g{}^{\rm t}g=\mathbf{1}_n\}, \\
{\rm U}(n)&:=\{g\in {\rm GL}_n(\mathbf{C}) \mid {}^{\rm t}(g^\rho)g=g{}^{\rm t}(g^\rho)=\mathbf{1}_n\}. 
\end{align*}
Here $\mathbf{1}_n$ denote the identity matrix of degree $n$. The special orthogonal group ${\rm SO}(n)$ is defined as the subgroup of ${\rm O}(n)$ consisting of elements with determinant equal to $1$. 

Now consider the general linear group ${\mathrm{GL}_n}_{/\mathsf{F}}$ of degree $n$ defined over a number field $\mathsf{F}$. Let $C^{(n)}_\infty$ and $K^{(n)}_\infty$ denote the subgroups of $\mathrm{GL}_n(\mathsf{F}_{\mathbf{A},\infty})$ defined by
\begin{align*}
 C^{(n)}_\infty &=\prod_{v\in \Sigma_{\mathsf{F},\infty}}C^{(n)}_v= \prod_{v\in \Sigma_{\mathsf{F},\mathbf{R}}} {\rm SO}(n)\times \prod_{v\in \Sigma_{\mathsf{F},\mathbf{C}}} {\rm U}(n), \\
 K^{(n)}_\infty 
     &= \mathsf{F}^\times_{\mathbf{A}, \infty, +} C^{(n)}_{\infty}
        =   \prod_{v\in \Sigma_{\mathsf{F},\mathbf{R}}} \mathbf{R}^\times_+ {\rm SO}(n)  
             \times \prod_{v\in \Sigma_{\mathsf{F},\mathbf{C}}} \mathbf{C}^\times  {\rm U}(n).
\end{align*}
Note 
that $\mathsf{F}^\times_{\mathbf{A}, \infty, +}$ denotes the identity component of $\mathsf{F}^\times_{\mathbf{A}, \infty}$
and that $C^{(n)}_\infty$ is the identity component of the standard maximal compact subgroup of $\mathrm{GL}_n(\mathsf{F}_{\mathbf{A},\infty})$. Finally the connected components of $\mathrm{GL}_n(\mathsf{F}_{\mathbf{A},\infty})$ 
containing the identity matrix $\mathbf{1}_n$ is denoted as 
\begin{align*}
 \mathrm{GL}_n(\mathsf{F}_{\mathbf{A},\infty})^\circ &=\prod_{v\in \Sigma_{\mathsf{F},\mathbf{R}}} \mathrm{GL}_n(\mathbf{R})^+ \times \prod_{v\in \Sigma_{\mathsf{F},\mathbf{C}}} \mathrm{GL}_n(\mathbf{C}).
\end{align*}

\subsection{General linear Lie algebras}\label{sec:lie}

For each $v\in \Sigma_{\mathsf{F}, \infty}$,  
let $\mathfrak{g}_{n, v}$ denote the complexification of the Lie algebra of ${\rm GL}_n(\mathsf{F}_v)$, and put $\mathfrak{g}_{n,\infty} = \bigotimes_{v\in \Sigma_{\mathsf{F}, \infty} }  \mathfrak{g}_{n, v}$. The dual Lie coalgebras of $\mathfrak{g}_{n,v}$ and $\mathfrak{g}_{n,\infty}$ are denoted as $\mathfrak{g}_{n,v}^*$ and $\mathfrak{g}_{n,\infty}^*=\bigotimes_{v\in \Sigma_{\mathsf{F},\infty}} \mathfrak{g}_{n,v}^*$ respectively.

Let ${\rm M}_n(\mathbf{C})$ denote the ring of square matrices of degree $n$ with coefficients in $\mathbf{C}$. We define $E_{ij} \in {\rm M}_n(\mathbf{C})$ as the square matrix such that its $(i,j)$-entry equals $1$ and the other entries equal $0$. We regard $E_{ij}$ as an element of $\mathfrak{g}_{n,v}$ for a real place  $v$. We also define $E^{\pm}_{ij}$ as 
\begin{align*}
  E^\pm_{ij} = \frac{1}{2}  \left\{   E_{i j}  \otimes 1    \mp  (\sqrt{-1}  E_{ij} ) \otimes \sqrt{-1}  \right\} \quad (\text{double sign in the same order})   \quad   \in \mathrm{M}_n(\mathbf{C})\otimes_{\mathbf{R}} \mathbf{C},
\end{align*}
which we regard as elements of $\mathfrak{g}_{n,v}$ for a complex place $v$. We sometimes use the symbols $E_{ij, v}$ or $E^\pm_{ij, v}$ to indicate that they are regarded as elements of $\mathfrak{g}_{n,v}$.
Then $\{  E_{ij, v} \}_{1\leq i,\, j \leq n}$ (resp.\  $\{  E^\pm_{ij, v} \}_{1\leq i,\, j \leq n}$) gives rise to a basis of $\mathfrak{g}_{n, v}$ if $v$ is real (resp.\ complex).  Hereafter we always choose a basis $\omega_{n, v}$ of the $k_v$-fold exterior power $\bigwedge^{k_v} \mathfrak{g}_{n, v}$ of $\mathfrak{g}_{n,v}$ as the set of all the $k_v$-fold wedge products of $\{  E_{ij, v} \}_{1\leq i,\, j \leq n}$ or $\{  E^\pm_{ij, v} \}_{1\leq i,\, j \leq n}$ depending on whether $v$ is real or complex. For $k=(k_v)_{v\in \Sigma_{\mathsf{F},\infty}}\in \mathbf{Z}_{\geq 0}^{\Sigma_{\mathsf{F},\infty}}$, we put $\bigwedge^k \mathfrak{g}_{n,\infty}=\bigotimes_{v\in \Sigma_{\mathsf{F},\infty}} \bigwedge^{k_v} \mathfrak{g}_{n,v}$. Then the dual basis of $\bigwedge^{\mathrm{max}}\mathfrak{g}_{n,\infty}^*=\bigotimes_{v\in \Sigma_{\mathsf{F},\mathbf{R}}}\bigwedge^{n^2} \mathfrak{g}_{n,v}^* \otimes \bigotimes_{v\in \Sigma_{\mathsf{F},\mathbf{C}}}\bigwedge^{2n^2} \mathfrak{g}_{n,v}^*$ gives rise to a volume form on ${\rm GL}_n(\mathsf{F}_{\mathbf A, \infty} )$, which is denoted by $\omega_n$.

\section{Fundamentals of periods of pure motives over general number fields}\label{sec:mot}

In this section we recall the notion of {\em fundamental periods} of pure motives defined over number fields, refering to Yoshida's expositions \cite{yos94}, \cite{yos01} and \cite{yos16}. 
 
\subsection{$\tau$-periods}\label{sec:t-per}

According to \cite{yos94}, we here summarize notation on $\tau$-periods of pure motives over general number fields. Let $\mathsf{F}$ and $E$ be number fields, $\mathsf{w}$ an integer and $d$ a positive integer. Let us consider a pure motive $\mathcal{M}$ of weight $\mathsf{w}=\mathsf{w}(\mathcal{M})$ and rank $d=d(\mathcal{M})$ defined over $\mathsf{F}$ with coefficients in $E$. For each embedding $\tau\in I_\mathsf{F}$ of $\mathsf{F}$, let $\mathcal{M}_\tau$ denote the motive defined over $\mathbf{C}$ obtained as the base change of $\mathcal{M}$ with respect to $\tau \colon F\hookrightarrow \mathbf{C}$, whose Betti realization is denoted as $H_{\rm B}(\mathcal{M}_\tau)$. By definition $H_{\rm B}(\mathcal{M}_\tau)$ is an $E$-vector space of dimension $d$ equipped with the {\em Hodge decomposition}
\begin{align*}
   H_{\rm B}(\mathcal{M}_\tau) \otimes_{\mathbf Q} {\mathbf C}  
      &=  \bigoplus_{p_\tau +q_\tau=\mathsf{w}} H^{p_\tau, q_\tau}(\mathcal{M}_\tau).
\end{align*} 
Note that each Hodge component $H^{p_\tau,q_\tau}(\mathcal{M}_\tau)$ is an $E\otimes_{\mathbf{Q}} \mathbf{C}$-module which is not necessarily free. 

The complex conjugation $\rho$ on the base field $\mathbf{C}$ of $\mathcal{M}_\tau$ transports the complex structure, which induces an $E$-linear isomorphism $F_{\infty,\tau} \colon H_{\rm B}(\mathcal{M}_\tau) \xrightarrow{\sim} H_{\rm B}(\mathcal{M}_{\rho \tau})$. Note that $F_{\infty,\rho \tau}$ is the inverse of $F_{\infty,\tau}$, and $F_{\infty,\tau}\otimes \mathrm{id}$ induces an isomorphism between $H^{p_\tau,q_\tau}(\mathcal{M}_\tau)$ and $H^{q_\tau, p_\tau}(\mathcal{M}_{\rho\tau})$. If $\tau$ is a real embedding, it coincides with $\rho\tau$, and thus $F_{\infty,\tau}$ (resp.\ $F_{\infty,\tau}\otimes \mathrm{id}$) is indeed an involution on  $H_{\rm B}(\mathcal{M}_\tau)$ (resp.\ on $H^{\mathsf{w}/2,\mathsf{w}/2}(\mathcal{M}_\tau)$ when $\mathsf{w}$ is even).  Meanwhile, if $\tau$ is a complex embedding,  $F_{\infty,\tau}\oplus F_{\infty,\rho\tau}$ (resp.\ $\bigl(F_{\infty,\tau}\oplus F_{\infty,\rho\tau}\bigr)\otimes \mathrm{id}$) is an involution on $H_{\rm B}(\mathcal{M}_\tau)\oplus H_{\rm B}(\mathcal{M}_{\rho\tau}) $ (resp.\ on $H^{\mathsf{w}/2,\mathsf{w}/2}(\mathcal{M}_\tau)\oplus H^{\mathsf{w}/2,\mathsf{w}/2}(\mathcal{M}_{\rho\tau}) $ when $\mathsf{w}$ is even). 
Let $H^\pm_{\rm B}(\mathcal{M}_\tau)$ denote the $(\pm 1)$-eigenspaces of $F_{\infty, \tau}$ acting on 
$H_{\rm B}(\mathcal{M}_\tau)$   
(resp.\ $F_{\infty,\tau}\oplus F_{\infty,\rho\tau}$ acting on $H_{\rm B}(\mathcal{M}_\tau) \oplus H_{\rm B}(\mathcal{M}_{\rho\tau})$) when $\tau$ is real (resp.\ complex). Put $d^{\pm}_\tau=\dim_E H^{\pm}_{\rm B}(\mathcal{M}_\tau)$; then, by definition, $d^+_\tau+d^-_\tau$ equals $d$ if $\tau\in I_\mathsf{F}$ is real and $2d$ if $\tau \in I_\mathsf{F}$ is complex.
Hereafter we assume that, when the weight $\mathsf{w}$ of the motive $\mathcal{M}$ is even, 
\begin{itemize}
\item[--] the involution $\bigoplus_{\tau \in I_\mathsf{F}}\bigl(F_{\infty, \tau}\otimes \mathrm{id}\bigr)$ on $\bigoplus_{\tau \in I_\mathsf{F}} H^{\mathsf{w}/2,\mathsf{w}/2}(\mathcal{M}_\tau)$ acts as multiplication by a scalar,
\end{itemize}
or, in other words,
\begin{itemize}
 \item[--] the involution $F_\infty\otimes \mathrm{id}$ on $H^{\mathsf{w}/2, \mathsf{w}/2}({\rm Res}_{\mathsf{F}/{\mathbf Q}} (\mathcal{M}))$ acts as multiplication by a scalar.
\end{itemize}
This condition implies that the diagonal component $H^{\mathsf{w}/2,\mathsf{w}/2}(\mathcal{M}_\tau)$ vanishes for every complex $\tau$.

Next let $H_{\rm dR}(\mathcal{M})$ denote the de Rham realization of $\mathcal{M}$, which is a free $E\otimes_{\mathbf Q} \mathsf{F}$-module of rank $d$. 
For each $\tau \in I_\mathsf{F}$, we have a comparison isomorphism
\begin{align*}
  I_\tau(\mathcal{M}) \colon H_{\rm B}(\mathcal{M}_\tau)   \otimes_{\mathbf Q} {\mathbf C}  
              \stackrel{\sim}{\longrightarrow} 
             H_{\rm dR}(\mathcal{M})   \otimes_{\mathsf{F}, \tau} {\mathbf C}, 
\end{align*}
which is  an isomorphism of $E\otimes_{\mathbf{Q}}\mathbf{C}$-modules. The de Rham realization $H_{\rm dR}(\mathcal{M})$ is equipped with the Hodge filtration $\{ {\rm Fil}^p H_{\rm dR}(\mathcal{M}) \}_{p}$, which is a descending filtration of (not necessarily free) $E\otimes_{\mathbf{Q}}\mathsf{F}$-modules characterized by
\begin{align*}
{\rm Fil}^p H_{\rm dR}(\mathcal{M}) \otimes_{\mathsf{F}, \tau} {\mathbf C}
  =  I_\tau (\mathcal{M}) \left(  \bigoplus_{i \geq p}   H^{i, \mathsf{w}-i} (\mathcal{M}_\tau) \right)  
\end{align*}
for each $\tau$.   
We define $E\otimes_{\mathbf{Q}} \mathsf{F}$-submodules $F^\mp H_{\rm dR}(\mathcal{M})$ of $H_{\rm dR}(\mathcal{M})$ as follows:
\begin{itemize}
 \item[--] $F^\mp H_{\rm dR}(\mathcal{M}) = {\rm Fil}^{(\mathsf{w}+1)/2} H_{\rm dR}(\mathcal{M})$  if $\mathsf{w}$ is odd,  
 \item[--] $F^- H_{\rm dR}(\mathcal{M}) = {\rm Fil}^{\mathsf{w}/2+1} H_{\rm dR}(\mathcal{M})$ and $F^+ H_{\rm dR}(\mathcal{M}) = {\rm Fil}^{\mathsf{w}/2} H_{\rm dR}(\mathcal{M})$ if $\mathsf{w}$ is even and $F_\infty\otimes \mathrm{id}$ acts on $H^{\mathsf{w}/2, \mathsf{w}/2}({\rm Res}_{\mathsf{F}/{\mathbf Q}} (\mathcal{M})  )$ as multiplication by $+1$,  
 \item[--]  $F^- H_{\rm dR}(\mathcal{M}) = {\rm Fil}^{\mathsf{w}/2} H_{\rm dR}(\mathcal{M})$ and
       $F^+ H_{\rm dR}(\mathcal{M}) = {\rm Fil}^{\mathsf{w}/2+1} H_{\rm dR}(\mathcal{M})$ if $\mathsf{w}$ is even and $F_\infty\otimes \mathrm{id}$ acts on $H^{\mathsf{w}/2, \mathsf{w}/2}({\rm Res}_{\mathsf{F}/{\mathbf Q}} (\mathcal{M})  )$ as multiplication by $-1$,
 \item[--]  $F^{\mp} H_{\rm dR}(\mathcal{M}) = {\rm Fil}^{\mathsf{w}/2} H_{\rm dR}(\mathcal{M})$ if $\mathsf{w}$ is even and $H^{\mathsf{w}/2,\mathsf{w}/2}(\mathrm{Res}_{\mathsf{F}/\mathbf{Q}}(\mathcal{M}))$ equals $0$.
\end{itemize}
Then define $H^{\pm}_{\rm dR}(\mathcal{M})$ as the quotients
\begin{align}\label{eq:drpm}
H^\pm_{\rm dR}(\mathcal{M}) = H_{\rm dR}(\mathcal{M})  /  F^\mp H_{\rm dR}(\mathcal{M}) \quad (\text{double sign in the same order}).
\end{align}
Note that $F^{\mp}H_{\rm dR}(\mathcal{M})$ and $H^{\pm}_{\rm dR}(\mathcal{M})$ are not necessarily free as $E\otimes_{\mathbf{Q}} \mathsf{F}$-modules.
For each $\tau\in I_{\mathsf{F}}$, the comparison isomorphism $I_\tau(\mathcal{M})$ (or $I_\tau(\mathcal{M})\oplus I_{\rho\tau}(\mathcal{M})$ if $\tau$ is complex) induces isomorphisms $I_\tau^\pm(\mathcal{M})$ of free $E\otimes_{\mathbf{Q}} \mathbf{C}$-modules 
\begin{align}\label{eq:comparison_pm}
\left\{
\begin{aligned}
      I^\pm_\tau(\mathcal{M}) & \colon H^\pm_{\rm B}(\mathcal{M}_\tau) \otimes_{\mathbf Q}{\mathbf C}  \stackrel{\sim}{\longrightarrow}  H^{\pm}_{\rm dR}(\mathcal{M}) \otimes_{\mathsf{F}, \tau} {\mathbf C}   && \text{if $\tau$ is real},  \\
      I^\pm_\tau(\mathcal{M})  & \colon \left( H_{\rm B}(\mathcal{M}_\tau) \oplus H_{\rm B}(\mathcal{M}_{\rho \tau})  \right)^\pm  \otimes_{\mathbf Q}{\mathbf C}
                       \stackrel{\sim}{\longrightarrow}   
                              \left(H^{\pm}_{\rm dR}(\mathcal{M})\otimes_{\mathsf{F},\tau} \mathbf{C}\right) \oplus \left( H^\pm_{\rm dR}(\mathcal{M})\otimes_{\mathsf{F},\rho\tau} \mathbf{C}\right)    &&  \text{if $\tau$ is complex}.
\end{aligned}
 \right.
\end{align}   

When $\mathsf{F}$ has a real place, the quotient $H^{\pm}_{\rm dR}(\mathcal{M})$ is indeed a free $E\otimes_{\mathbf{Q}} \mathsf{F}$-module due to \cite[Lemma~2.1 (2)]{yos94}. When $\mathsf{F}$ is totally imaginary, let us define $\mathcal{M}^{[\tau]}$ and $\mathcal{M}^{[\rho \tau]}$ as the extension of scalars of $\mathcal{M}$ from $\mathsf{F}$ to $\tau(\mathsf{F})\rho\tau(\mathsf{F})$ according to $\tau \colon \mathsf{F}\hookrightarrow \tau(\mathsf{F})\rho\tau(\mathsf{F})$ and $\rho\tau \colon \mathsf{F}\hookrightarrow \tau(\mathsf{F})\rho\tau(\mathsf{F})$ respectively.\footnote{Here $H^{\pm}_{\rm dR}(\mathcal{M}^{[\rho\tau]})$ play the roles of $\rho H^{\pm}_{\rm dR}(\mathcal{M})$ introduced in \cite{yos94} and \cite{yos16}; the $\rho$-twisted modules $\rho H^{\pm}_{\rm dR}(\mathcal{M})$ seem to be defined only when the base field of $\mathcal{M}$ is a subfield of $\mathbf{C}$ stable under the complex conjugation $\rho$.}
 Then the direct sum $H^{\pm}_{\rm dR}(\mathcal{M}^{[\tau]})\oplus H^{\pm}_{\rm dR}(\mathcal{M}^{[\rho\tau]})$ is a free $E\otimes_{\mathbf{Q}}\tau(\mathsf{F})\rho\tau(\mathsf{F})$-module due to \cite[Lemma~2.1 (3)]{yos94}. 
Note that one may identify $I_\tau^\pm(\mathcal{M})$ with the comparison isomorphisms $I_{\iota_0}^{\pm}(\mathcal{M}^{[\tau]})$, the (latter) isomorphisms of (\ref{eq:comparison_pm}) constructed for $\mathcal{M}^{[\tau]}$ and the canonical inclusion $\iota_0\colon \tau(\mathsf{F})\rho\tau(\mathsf{F})\hookrightarrow \mathbf{C}$, simply because $\left(\mathcal{M}^{[\tau]}\right)_{\iota_0}=\mathcal{M}_{\tau}$, $\left(\mathcal{M}^{[\tau]}\right)_{\rho\iota_0}=\mathcal{M}_{\rho\tau}$ and 
\begin{align*}
 \bigl(H_{\rm dR}^\pm(\mathcal{M}^{[\tau]})\otimes_{\tau(\mathsf{F})\rho\tau(\mathsf{F}),\iota_0}\mathbf{C}\bigr) \oplus \bigl(H_{\rm dR}^\pm(\mathcal{M}^{[\tau]})\otimes_{\tau(\mathsf{F})\rho\tau(\mathsf{F}),\rho\iota_0}\mathbf{C}\bigr)&=\bigl(H_{\rm dR}^\pm(\mathcal{M})\otimes_{\mathsf{F},\tau}\mathbf{C}\bigr) \oplus \bigl(H_{\rm dR}^\pm(\mathcal{M})\otimes_{\mathsf{F},\rho\tau}\mathbf{C}\bigr) \\
&=\bigl(H_{\rm dR}^\pm(\mathcal{M}^{[\tau]}) \oplus H_{\rm dR}^\pm(\mathcal{M}^{[\rho \tau]})\bigr)\otimes_{\tau(\mathsf{F})\rho\tau(\mathsf{F}),\iota_0} \mathbf{C}
\end{align*}
hold by construction.

\begin{dfn}
For each $\tau \in I_\mathsf{F}$, let us take an $E$-rational basis of $H_{\rm B}(\mathcal{M}_\tau)$ and an $E\otimes_{\mathbf{Q}} \mathsf{F}$-rational basis of $H_{\rm dR}(\mathcal{M})$. Let us also take $E$-rational bases of $H^\pm_{\rm B}(\mathcal{M}_\tau)$ 
and $E\otimes_{\mathbf Q} \mathsf{F}$-rational bases 
 of $H^{\pm}_{\rm dR} (\mathcal{M})$ if $\mathsf{F}$ has a real place (resp.\ $E\otimes_{\mathbf{Q}} \tau(\mathsf{F})\rho\tau(\mathsf{F})$-rational bases of $H^{\pm}_{\rm dR} (\mathcal{M}^{[\tau]}) \oplus  H^{\pm}_{\rm dR} (\mathcal{M}^{[\rho\tau]})$ if $\mathsf{F}$ is totally imaginary).       
Then we define {\em $\tau$-periods} $\delta_\tau(\mathcal{M})$, $c^\pm_\tau(\mathcal{M}) \in (E\otimes_{\mathbf Q} {\mathbf C})^\times$ of $\mathcal{M}$
to be the determinants with respect to these rational bases:
\begin{align*}
  \delta_\tau(\mathcal{M}) &= \begin{cases}  \det(I_\tau(\mathcal{M})) & \text{if $\tau$ is real},  \\  \det(I_\tau(\mathcal{M})\oplus I_{\rho \tau}(\mathcal{M})) & \text{if $\tau$ is complex},    \end{cases} & 
  c^\pm_\tau(\mathcal{M})  &= \det (I^\pm_\tau(\mathcal{M})).
\end{align*}
Here we identify $I^\pm_\tau(\mathcal{M})$ with $I^\pm_{\iota_0}(\mathcal{M}^{[\tau]})$ when $\mathsf{F}$ is totally imaginary. By construction, these periods are determined up to multiplication by elements of $(E\otimes_{\mathbf{Q}}\mathsf{F})^\times$ (or by elements of $(E\otimes_{\mathbf{Q}} \tau(\mathsf{F})\rho\tau(\mathsf{F}))^\times$ for $c^{\pm}_{\tau}(\mathcal{M})$ when $\mathsf{F}$ is totally imaginary). 
\end{dfn}

For a complex embedding $\tau\in I_\mathsf{F}$, one has $\delta_\tau(\mathcal{M}) = \delta_{\rho\tau}(\mathcal{M})$ and 
$c^\pm_\tau(\mathcal{M}) = c^\pm_{\rho\tau}(\mathcal{M})$ by symmetry, and it thus suffices to consider one of two embeddings  associated to the same complex place when one concerns $\tau$-periods. In other words, $\tau$-periods are defined for every archimedean place of $\mathsf{F}$.

For later use, we end this subsection by introducing two lemmata without proofs; the first one is 
the {\em product formulae}  which express Deligne's periods of the Grothendieck restriction of scalars $\mathrm{Res}_{\mathsf{F}/\mathbf{Q}}(\mathcal{M})$ of $\mathcal{M}$ to $\mathbf{Q}$ as the product of corresponding $\tau$-periods of $\mathcal{M}$:

\begin{lem}\label{lem:perfac}
Let $\mathsf{F}^{\rm gal}$ denote the normal closure of $\mathsf{F}$ in $\mathbf{C}$.
Then we have 
\begin{align*}
 \delta ( {\rm Res}_{\mathsf{F}/ {\mathbf Q}}  (\mathcal{M})  )   
     &\sim_{E\otimes_{\mathbf{Q}} \mathsf{F}^{\rm gal}} 
            \prod_{\tau \in \Sigma_{\mathsf{F}, \infty}}  \delta_\tau(\mathcal{M}), &
  c^\pm ( {\rm Res}_{\mathsf{F}/ {\mathbf Q}}  (\mathcal{M})  )   
     &\sim_{E\otimes_{\mathbf{Q}}\mathsf{F}^{\rm gal}} 
             \prod_{\tau \in \Sigma_{\mathsf{F}, \infty}}  c^\pm_\tau(\mathcal{M}). 
\end{align*}
\end{lem}

\begin{proof}
See \cite[Proposition 2.2]{yos94} and \cite[Proposition 2.11]{hl17}.
\end{proof}

The second lemma concerns the $\pm$-periods for every complex embedding $\mathsf{F}\hookrightarrow \mathbf{C}$. 
\begin{lem}\label{lem:cpm}
We have $ c^+_\tau(\mathcal{M})  \sim_{E\otimes_{\mathbf{Q}} \tau(\mathsf{F})\rho\tau(\mathsf{F})}   c^-_\tau(\mathcal{M})$ if  $H^{\mathsf{w}/2,\mathsf{w}/2}(\mathcal{M}_\tau)$ vanishes for every complex $\tau$. 
\end{lem}
\begin{proof}
See \cite[Remark 2.2]{hl17}.
\end{proof}

Note that the assumption of Lemma~\ref{lem:cpm} on the diagonal component $H^{\mathsf{w}/2,\mathsf{w}/2}(\mathcal{M}_\tau)$ is fulfilled in our setting since we assume that $F_{\infty}$ acts on $\mathrm{Res}_{\mathsf{F}/\mathbf{Q}}(\mathcal{M})$ as multiplication by a scalar.

\subsection{$(\tau,\sigma)$-fundamental periods}\label{sec:YosPer}

We continue to use the same notation as in Section \ref{sec:t-per}.
In this subsection, we briefly recall the definition of $(\tau,\sigma)$-fundamental periods of pure motives over general number fields, according to \cite[Section 4]{yos16}. For this purpose, we first introduce the following definition on relative invariant polynomials from \cite[Section 1]{yos01} and \cite[Section 1]{yos16}.

\begin{dfn} 
Let $\mathsf{d}$ be a positive integer equipped with two partitions 
\begin{align*}
 \mathsf{d}=\mathsf{s}_1+\mathsf{s}_2+\cdots+\mathsf{s}_{\mathsf{t}}=\mathsf{d}^++\mathsf{d}^-
\end{align*}
by positive integers for a certain positive integer $\mathsf{t}$ $(\leq \mathsf{d})$. Assume that there exist positive integers $1\leq \mathsf{t}^\pm \leq \mathsf{t}$ satisfying $\mathsf{s}_1+\mathsf{s}_2+\cdots+\mathsf{s}_{\mathsf{t}^\pm}=\mathsf{d}^\pm$. Define $P_{\{\mathsf{s}_\mu\}_\mu}$ as the lower parabolic subgroup of $\mathrm{GL}_{\mathsf{d}}$ corresponding to the partition $\{\mathsf{s}_\mu\}_{1\leq \mu\leq \mathsf{t}}$ of $\mathsf{d}$, and let $a_1,\dotsc,a_\mathsf{t}$ and $k^+,k^-$ be non-negative integers. 
\begin{enumerate}
\item We say that $\{  (a_1, \ldots, a_\mathsf{t}) ; (k^+, k^-)  \}$ is an {\em admissible type} if all the following conditions are fulfilled:
\begin{itemize}
 \item[--] the sequence $\{a_\mu\}_{1\leq \mu\leq \mathsf{t}}$ is monotonically non-increasing;
 \item[--] the sum $a_i+a_{\mathsf{t}+1-i}$ equals $k^++k^-$ for each $1\leq i\leq \min\{\mathsf{t}^+,\mathsf{t}^-\}$;
 \item[--] both of the inequalities $a_{\mathsf{t}^+}\geq k^+$ and  $k^-\geq a_{\mathsf{t}^++1}$ hold; 
\item[--] both of the inequalities $a_{\min\{\mathsf{t}^+,\mathsf{t}^-\}}\geq \max\{k^+,k^-\}$ and $\min\{k^+,k^-\}\geq a_{\max\{\mathsf{t}^+,\mathsf{t}^-\}}$ hold;
 \item[--] if $\mathsf{t}^+$ is strictly greater than $\mathsf{t}^-$, the integers $a_{\mathsf{t}^-+1}, a_{\mathsf{t}^-+2}, \dotsc, a_{\mathsf{t}^+}$ equal $k^+$, whereas  if $\mathsf{t}^-$ is strictly greater than $\mathsf{t}^+$, the integers  $a_{\mathsf{t}^++1}, a_{\mathsf{t}^++2}, \dotsc, a_{\mathsf{t}^-}$ equal $k^-$.
\end{itemize}
\item Let $f(x)$ be a polynomial function with coefficients in $\mathbf{Q}$ defined on ${\rm M}_{\mathsf{d}}({\mathbf C})$.     
We say $f(x)$ to be {\em of type} $\{  (a_1, \ldots, a_\mathsf{t}) ; (k^+, k^-)  \}$ if $f(x)$ satisfies the following condition:
\begin{quotation}
   for each $p= \begin{pmatrix}  p_{11} & 0 & \cdots & 0  \\   \ast &  p_{22} &  \cdots & 0  \\   \ast & \ast &  \ddots & \vdots \\ \ast & \ast & \ast & p_{\mathsf{tt}} \end{pmatrix} \in P_{\{\mathsf{s}_\mu\}_\mu}(\mathbf{C})$, 
          $h_+ \in {\rm GL}_{\mathsf{d}^{+}}(\mathbf{C})$ and  $h_- \in {\rm GL}_{\mathsf{d}^{-}}(\mathbf{C})$, we have 
          \begin{align*}
             f\left( p  x \begin{pmatrix}
			   h_+ & \\ & h_-
			  \end{pmatrix}  \right)  &= \left(\prod^{\mathsf{t}}_{i=1} \det(p_{ii} )^{a_i}\right) \det (h_+)^{k^+}  \det (h_-)^{k^-}f(x).     
          \end{align*}
\end{quotation}
\end{enumerate}
\end{dfn}

Yoshida has verified that there exists a polynomial function of type $\{(a_1,\dotsc,a_\mathsf{t});(k^+,k^-)\}$ (which is unique up to constant multiples) if and only if $\{(a_1,\dotsc,a_\mathsf{t});(k^+,k^-)\}$ is an admissible type; see \cite[Theorem~1]{yos01}. For example, the determinant function $\det(x)$ is of type $\{(1,1,\dotsc,1); (1,1)\}$, and the polynomial function $f^+(x)$ (resp.\ $f^-(x)$) defined as the determinant of the upper left square submatrix of degree $\mathsf{d}^+$ (resp.\ the upper right square submatrix of degree $\mathsf{d}^-$) of $x\in {\rm M}_{\mathsf{d}}(\mathbf{C})$ is of type $\{  (\overbrace{1, \ldots, 1}^{\mathsf{t}^+}, 0, \ldots, 0 ) ; (1, 0)  \}$ (resp.\ of type $\{  (\overbrace{1, \ldots, 1}^{\mathsf{t}^-}, 0, \ldots, 0 ) ; (0, 1)  \}$). For every positive integer $1\leq \beta\leq \min\{\mathsf{t}^+,\mathsf{t}^-\}$, there exists a polynomial function $f_\beta(x)$ of admissible type $\{(\overbrace{2,\dotsc,2}^{\beta},\overbrace{1,\dotsc,1}^{\mathsf{t}-2\beta},\overbrace{0,\dotsc,0}^\beta); (1,1)\}$ due to \cite[Theorem~1]{yos01}. The next theorem gives algebraic foundation on the theory of fundamental periods:

\begin{thm}[See {\protect\cite[Theorem~3 and p.~1181]{yos01}}] \label{thm:invpoly} Every polynomial function $f(x)$ on ${\rm M}_{\mathsf{d}}(\mathbf{C})$ of admissible type can be expressed uniquely $($up to constant multiples$)$ as a monomial of $\det(x)$, $f^{\pm}(x)$ and $f_\beta(x)$ for $1\leq \beta\leq \min\{\mathsf{t}^+,\mathsf{t}^-\}$. 
\end{thm}

The $(\tau,\sigma)$-fundamental periods of the pure motive $\mathcal{M}$ are defined as evaluation of the invariant polynomials appearing in the statement of Theorem~\ref{thm:invpoly} at the {\em $(\tau,\sigma)$-period matrix} $X^{\mathcal{M}}_{\tau,\sigma}$ associated to $\mathcal{M}$ defined as follows.
For each $\sigma \in I_E$ and $\tau \in I_\mathsf{F}$, we regard ${\mathbf C}$ as an $E\otimes_{\mathbf Q} \mathsf{F}$-module 
    via the ${\mathbf Q}$-algebra homomorphism
\begin{align*}
  \sigma\otimes\tau \colon  E\otimes_{\mathbf Q} \mathsf{F} \longrightarrow {\mathbf C} \, ; x\otimes y \longmapsto \sigma(x)\tau(y).      
\end{align*}    
Since $E\otimes_{\mathbf Q} {\mathbf C}$ is isomorphic to $\prod_{\sigma \in I_E} {\mathbf C}$,    
we readily see that  $H_{\rm B}(\mathcal{M}_\tau)\otimes_{\mathbf{Q}}\mathbf{C}$  and  $H_{\rm dR}(\mathcal{M})\otimes_{\mathsf{F},\tau} \mathbf{C}$ are decomposed as
\begin{align*}
  H_{\rm B}(\mathcal{M}_\tau) \otimes_{\mathbf Q} {\mathbf C}   &\cong \bigoplus_{\sigma \in I_E}  H_{\rm B} (\mathcal{M}_\tau) \otimes_{ E, \sigma  } {\mathbf C},   &
  H_{\rm dR}(\mathcal{M}) \otimes_{\mathsf{F}, \tau} {\mathbf C}   &\cong \bigoplus_{\sigma \in I_E}  H_{\rm dR} (\mathcal{M}) \otimes_{ E\otimes_{\mathbf Q} \mathsf{F}, \sigma \otimes \tau  } {\mathbf C}
\end{align*}
for each $\tau \in I_\mathsf{F}$. The comparison isomorphism $I_{\tau}(\mathcal{M})$ thus induces  an isomorphism of $\mathbf{C}$-vector spaces
\begin{align*}
   I_{\tau,\sigma}(\mathcal{M}) \colon H_{\rm B}(\mathcal{M}_\tau)  \otimes_{E, \sigma } {\mathbf C}  
                       \stackrel{\sim}{\longrightarrow}  
                   H_{\rm dR}(\mathcal{M})  \otimes_{E\otimes_{\mathbf Q} \mathsf{F}, \sigma \otimes \tau } {\mathbf C} 
\end{align*}
componentwisely for each $\sigma\in I_E$. In the following, for each $\sigma \in I_E$ and $\tau\in I_{\mathsf{F}}$, let $E^\sigma$, $\mathsf{F}^{\tau}$ and $\mathsf{F}^{(\tau,\rho)}$ denote $\sigma(E)$, $\tau(\mathsf{F})$ and $\tau(\mathsf{F})\rho\tau(\mathsf{F})$ respectively. Note that $E^\sigma \mathsf{F}^\tau$ and $E^\sigma \mathsf{F}^{(\tau,\rho)}$ are regarded as $E\otimes_{\mathbf{Q}}\mathsf{F}$-submodules of $\mathbf{C}$. 

First we choose rational bases of Betti cohomology groups. For each $\sigma \in I_E$ and $\tau \in I_\mathsf{F}$, let $\mathcal{B}_\sigma(\mathcal{M}_\tau)$ be an $E^\sigma$-vector space defined as
	  \begin{align} \label{eq:Bsigma}
          {\mathcal B}_\sigma(\mathcal{M}_\tau)
         :=  \begin{cases}
             H_{\rm B}(\mathcal{M}_\tau) \otimes_{E, \sigma} E^\sigma  &  \text{if $\tau \in I_\mathsf{F}$ is real},  \\
               \bigl(H_{\rm B}(\mathcal{M}_\tau) \otimes_{E, \sigma} E^\sigma\bigr) \oplus \bigl( H_{\rm B}(\mathcal{M}_{\rho \tau}) \otimes_{E, \sigma} E^\sigma\bigr)   &  \text{if $\tau\in I_\mathsf{F}$ is complex},    
          \end{cases}
         \end{align}         
         and ${\mathcal B}^\pm_\sigma(\mathcal{M}_\tau)$ the $\pm$-eigenspaces of the involution 
\begin{align} \label{eq:involution}
 F^{\mathcal{M}}_{\tau,\sigma}:=
\begin{cases}
 F_{\infty,\tau}\otimes \mathrm{id} &  \text{if $\tau$ is real}, \\
(F_{\infty,\tau}\otimes \mathrm{id})\oplus (F_{\infty,\rho\tau}\otimes \mathrm{id}) & \text{if $\tau$ is complex}
\end{cases}
\end{align}
acting on $\mathcal{B}_\sigma(\mathcal{M}_\tau)$. By construction, the dimensions $d_*$ and $d_*^{\pm}$ of $\mathcal{B}_\sigma(\mathcal{M}_\tau)$ and $\mathcal{B}_\sigma^\pm(\mathcal{M}_\tau)$ are respectively described as 
\begin{align*}
 d_*&=d_*(\mathcal{M})=
\begin{cases}
 d & \text{if $\tau\in I_{\mathsf{F}}$ is real}, \\
2d & \text{if $\tau\in I_{\mathsf{F}}$ is complex}, 
\end{cases} & 
 d_*^\pm&=d_*^{\pm}(\mathcal{M}_\tau)=
\begin{cases}
 d^\pm_\tau & \text{if $\tau\in I_{\mathsf{F}}$ is real},  \\
d & \text{if $\tau\in I_{\mathsf{F}}$ is complex}. 
\end{cases}
\end{align*}
Here recall that $d$ denotes the rank of $\mathcal{M}$. Obviously these numbers are independent of $\sigma\in I_E$, and $\{d^+_*,d^-_*\}$ gives a partition of $d_*$. 
  Let us choose bases $\{ u^{\pm}_1, u^{\pm}_2,\dotsc, u^\pm_{d^\pm_\ast}\}$ of the $E^\sigma$-vector spaces ${\mathcal B}^\pm_\sigma(\mathcal{M}_\tau)$.

Next we choose rational bases of de Rham  cohomology groups. For each $\sigma \in I_E$ and $\tau \in I_{\mathsf{F}}$, define $\mathcal{T}_{\tau,\sigma}(\mathcal{M})$ as
\begin{align} \label{eq:dRst}
   {\mathcal T}_{ \tau,\sigma} (\mathcal{M}) 
      :=  \begin{cases} 
             H_{\rm dR} (\mathcal{M}) \otimes_{E\otimes_{\mathbf Q} \mathsf{F}, \sigma \otimes \tau } E^\sigma \mathsf{F}^\tau   &  \text{if $\tau\in I_\mathsf{F}$ is real},  \\
              \bigl(H_{\rm dR} (\mathcal{M}) \otimes_{E\otimes_{\mathbf{Q}} \mathsf{F},\sigma\otimes \tau} E^\sigma \mathsf{F}^{(\tau,\rho)} \bigr)  \oplus \bigl( H_{\rm dR} (\mathcal{M})  \otimes_{E\otimes_{\mathbf Q} \mathsf{F}, \sigma \otimes \rho\tau } E^\sigma \mathsf{F}^{(\tau,\rho)}\bigr)    &  \text{if $\tau \in I_\mathsf{F}$ is complex}.
           \end{cases}         
\end{align}
Note that $\mathcal{T}_{\tau,\sigma}(\mathcal{M})$ is equipped with the Hodge filtration $\{\mathcal{T}^p_{\tau,\sigma}(\mathcal{M})\}_p$ described as 
\begin{multline}\label{eq:tfil}
    {\mathcal T}^p_{ \tau,\sigma} (\mathcal{M}) \\
      :=  \begin{cases} 
            \mathrm{Fil}^p H_{\rm dR} (\mathcal{M}) \otimes_{E\otimes_{\mathbf Q} \mathsf{F}, \sigma \otimes \tau } E^\sigma \mathsf{F}^\tau   &  \text{if $\tau\in I_\mathsf{F}$ is real},  \\
              \bigl(\mathrm{Fil}^p H_{\rm dR} (\mathcal{M}) \otimes_{E\otimes_{\mathbf{Q}} \mathsf{F},\sigma\otimes \tau} E^\sigma \mathsf{F}^{(\tau,\rho)} \bigr)  \oplus \bigl(  \mathrm{Fil}^p H_{\rm dR} (\mathcal{M})  \otimes_{E\otimes_{\mathbf Q} \mathsf{F}, \sigma \otimes \rho\tau } E^\sigma \mathsf{F}^{(\tau,\rho)}\bigr)    &  \text{if $\tau \in I_\mathsf{F}$ is complex}.
           \end{cases}         
\end{multline}
We now define integers $i_\mu=i^{\mathcal{M}_{\tau,\sigma}}_{\mu}$ for $\mu=1,2,\dotsc,\mathsf{t}^{\mathcal{M}_{\tau,\sigma}}+1$ inductively as follows: 
\begin{align} \label{eq:filindices}
\begin{aligned}
 i_1 &:= \max  \{ p\in \mathbf{Z} \mid \mathcal{T}^p_{\tau,\sigma}(\mathcal{M})=\mathcal{T}_{\tau,\sigma}(\mathcal{M})\}, && \\
 i_\mu &:=\max \{i_{\mu-1}+p \mid p\in \mathbf{N},\, \mathcal{T}^{i_{\mu-1}+p}_{\tau,\sigma}(\mathcal{M})=\mathcal{T}^{i_{\mu-1}+1}_{\tau,\sigma}(\mathcal{M})\} && \text{if $\mu\geq 2$ and $\mathcal{T}^{i_{\mu-1}+1}_{\tau,\sigma}(\mathcal{M})\neq \{0\}$}, \\ 
 i_{\mathsf{t}^{\mathcal{M}_{\tau,\sigma}}} &:=i_\mu \text{ and }  i_{\mathsf{t}^{\mathcal{M}_{\tau,\sigma}}+1} :=i_\mu+1 && \text{if $\mathcal{T}^{i_\mu}_{\tau,\sigma}(\mathcal{M})\neq 0$ and $\mathcal{T}^{i_\mu+1}_{\tau,\sigma}(\mathcal{M})=0$}.
\end{aligned}
\end{align}
Namely the sequence $\{i^{\mathcal{M}_{\tau,\sigma}}_{\mu}\}_{1\leq \mu\leq \mathsf{t}^{\mathcal{M}_{\tau,\sigma}}+1}$ denotes the {\em jumps} of the Hodge filtration $\{\mathcal{T}^p_{\tau,\sigma}(\mathcal{M})\}_p$. 
We also define integers $s_\mu=s^{\mathcal{M}_{\tau,\sigma}}_{\mu}$ for $\mu=1,2,\dotsc,\mathsf{t}^{\mathcal{M}_{\tau,\sigma}}$  as\footnote{Note that $E^\sigma \mathsf{F}^{(\tau, \rho)}=E^\sigma \mathsf{F}^\tau$ when $\tau\colon \mathsf{F}\hookrightarrow \mathbf{C}$ is a real embedding.}
\begin{align*}
   s_\mu  
    &=   {\rm dim}_{E^\sigma \mathsf{F}^{(\tau,\rho)}}    {\mathcal T}^{i_\mu}_{ \tau,\sigma} (\mathcal{M})   -     {\rm dim}_{E^\sigma \mathsf{F}^{(\tau,\rho)}}    {\mathcal T}^{i_{\mu+1}}_{ \tau,\sigma} (\mathcal{M}).   
\end{align*}
By construction, the set of integers $\{s_\mu\}_{1\leq \mu\leq \mathsf{t}^{\mathcal{M}_{\tau,\sigma}}}$ gives a symmetric partition of 
       $d_\ast = {\rm dim}_{E^\sigma \mathsf{F}^{(\tau,\rho)} }  {\mathcal T}_{ \tau,\sigma} (\mathcal{M})$;
namely we have 
\begin{align*}
   d_\ast  &= s_1 + s_2+ \cdots +  s_{\mathsf{t}^{\mathcal{M}_{\tau,\sigma}}}  &\text{with} \quad  s_i&=s_{\mathsf{t}^{\mathcal{M}_{\tau,\sigma}}+1-i} \qquad \text{for $1\leq i\leq \lfloor \mathsf{t}^{\mathcal{M}_{\tau,\sigma}}/2\rfloor$}.
\end{align*}
Note that, due to the assumption that $F_\infty$ acts on $H^{\mathsf{w}/2,\mathsf{w}/2}(\mathrm{Res}_{\mathsf{F}/\mathbf{Q}}(\mathcal{M}))$ as multiplication by a scalar, there exist positive integers $\mathsf{t}^{\mathcal{M}_{\tau,\sigma},\pm}$ satisfying $s_1+s_2+\cdots+s_{\mathsf{t}^{\mathcal{M}_{\tau,\sigma},\pm}}=d^{\pm}_\ast(\mathcal{M}_\tau)$. Let us choose a basis $\{\omega_1, \ldots, \omega_{d_\ast}\}$ 
                of the $E^\sigma \mathsf{F}^{(\tau,\rho)}$-vector space ${\mathcal T}_{ \tau,\sigma} (\mathcal{M})$  
         so that $\{\omega_{s_1 +s_2 + \cdots +s_{\mu-1} +1},  \ldots, \omega_{d_\ast} \}$ forms a basis of ${\mathcal T}^{i_\mu}_{\tau,\sigma} (\mathcal{M})$  for each $1\leq \mu \leq \mathsf{t}^{\mathcal{M}_{\tau,\sigma}}$ (here we set $s_0=0$ as convention).

 Now consider the comparison isomorphism
         \begin{align*}
                 I^0_{\tau,\sigma}(\mathcal{M}) \colon  {\mathcal B}_\sigma(\mathcal{M}_\tau) \otimes_{E, \sigma} {\mathbf C}  \longrightarrow {\mathcal T}_{ \tau,\sigma} (\mathcal{M})\otimes_{E\otimes_{\mathbf Q} \mathsf{F}, \sigma \otimes \tau } {\mathbf C}
         \end{align*}
         defined as $I^0_{\tau,\sigma}(\mathcal{M}) = I_{\tau,\sigma}(\mathcal{M})$ if $\tau$ is real, 
         and $I^0_{\tau,\sigma}(\mathcal{M}) = I_{\tau,\sigma}(\mathcal{M}) \oplus I_{\rho \tau,\sigma}(\mathcal{M})$ if $\tau$ is complex.
         In the following, we often regard $\{I^0_{\tau,\sigma}(\mathcal{M})\}_{\tau,\sigma}$ as indexed by $\tau \in \Sigma_{\mathsf{F}, \infty}$ and $\sigma\in I_E$.
Then define the {\em $(\tau,\sigma)$-period matrix} $X^{\mathcal{M}}_{\tau,\sigma}=\begin{pmatrix}
		X^{\mathcal{M},+}_{\tau,\sigma} &X^{\mathcal{M},-}_{\tau,\sigma}								   \end{pmatrix}$ as the matrix presentation of the comparison isomorphism $I^0_{\tau,\sigma}=I^0_{\tau,\sigma}(\mathcal{M})$ with respect to the bases $\{u_j^{\pm}\}_j$ and $\{\omega_j\}_j$:
         \begin{multline*}
            \begin{pmatrix}
	     I^0_{\tau,\sigma}(u^+_1) & \dotsc & I^0_{\tau,\sigma}(u^+_{d^+_\ast}) &   I^0_{\tau,\sigma}(u^-_1) & \dotsc & I^0_{\tau,\sigma}(u^-_{d^-_\ast}) 
	    \end{pmatrix} = 
\begin{pmatrix}
 \omega_1 & \omega_2 & \dotsc & \omega_{d_\ast}
\end{pmatrix} \begin{pmatrix}
		X^{\mathcal{M},+}_{\tau,\sigma} &X^{\mathcal{M},-}_{\tau,\sigma}								   \end{pmatrix},   \\   
            X^{\mathcal{M},\pm}_{\tau,\sigma} \in {\rm M}_{d_\ast, d^\pm_\ast} ({\mathbf C}).
         \end{multline*}          

\begin{dfn}
 Let $f_{(\tau,\sigma)}^{\pm}(x)$ and $f_{\beta,(\tau,\sigma)}(x)$ for $1\leq \beta\leq \min\{ \mathsf{t}^{\mathcal{M}_{\tau,\sigma},+},\mathsf{t}^{\mathcal{M}_{\tau,\sigma},-}\}$ be invariant polynomial functions on ${\rm M}_{d_*}(\mathbf{C})$ with respect to the partitions $\{s^{\mathcal{M}_{\tau,\sigma}}_{\mu}\}_\mu$ and $\{d_*^+(\mathcal{M}_\tau),d_*^-(\mathcal{M}_\tau)\}$ of $d_*$, defined as in the paragraphs just above Theorem~\ref{thm:invpoly}. Then we define the {\em $(\tau,\sigma)$-fundamental periods} $\delta_{(\tau,\sigma)}(\mathcal{M})$, $c^{\pm}_{(\tau,\sigma)}(\mathcal{M})$ and $c_{\beta,(\tau,\sigma)}(\mathcal{M})$ as follows:
\begin{align*}
 \delta_{(\tau,\sigma)} (\mathcal{M}) &=   \det(X^{\mathcal{M}}_{\tau,\sigma}), &
 c^\pm_{(\tau,\sigma)} (\mathcal{M}) &= f^\pm_{(\tau,\sigma)}(X^{\mathcal{M}}_{\tau,\sigma}),  &
  c_{\beta, (\tau,\sigma)} (\mathcal{M}) = f_{\beta,(\tau,\sigma)}(X^{\mathcal{M}}_{\tau,\sigma}).
\end{align*}
These periods are determined up to multiples of $E^{\sigma}\mathsf{F}^{(\tau,\rho)}$ depending on the choices of the rational bases $\{u_j^{\pm}\}_j$ and $\{\omega_j\}_j$, due to 
the admissibility of the polynomial functions $\det(x)$, $f^{\pm}_{(\tau,\sigma)}(x)$ and $f_{\beta,(\tau,\sigma)}(x)$.
\end{dfn}

       Note that, under the identification $ E \otimes_{\mathbf Q} {\mathbf C} = \prod_{\sigma \in I_E} {\mathbf C}$, 
       we have
       \begin{align*}
         c^\pm_\tau(\mathcal{M})   
             \sim_{ E \otimes_\mathbf{Q} \mathsf{F}^{(\tau,\rho)}}  
                \left(   c^{\pm}_{(\tau,\sigma)}(\mathcal{M})  \right)_{ \sigma \in I_E }  \in (E\otimes_\mathbf{Q}\mathbf{C})^\times.
       \end{align*}
Thus, for example, Lemma~\ref{lem:cpm} implies that $c^+_{(\tau,\sigma)}(\mathcal{M})\sim_{E^\sigma \mathsf{F}^{(\tau,\rho)}}c^-_{(\tau,\sigma)}(\mathcal{M})$ holds for every $\sigma\in I_E$ when $\tau \in \Sigma_{\mathsf{F},\infty}$ is a complex place.




\section{Periods of Rankin--Selberg motives}\label{sec:RSPer}


Let us consider two pure motives $\mathcal{M}$ and $\mathcal{N}$ both defined over $\mathsf{F}$ with coefficients in $E$; the weight and rank of the motive $\mathcal{M}$ (resp.\ the motive $\mathcal{N}$) are denoted as $\mathsf{w}(\mathcal{M})$ and $d(\mathcal{M})$ (resp.\ $\mathsf{w}(\mathcal{N})$ and $d(\mathcal{N})$). Let $\tau\in \Sigma_{\mathsf{F},\infty}$ be an archimedean place of $\mathsf{F}$ and $\sigma\in I_E$ an embedding of $E$ into $\mathbf{C}$.
The main subject of this section is to give explicit factorization formulae of (the product of) the $(\tau,\sigma)$-periods 
\begin{align} \label{eq:cmn}
 c^{\pm}_{(\tau,\sigma)}(\mathcal{M},\mathcal{N}) :=
\begin{cases}
 c^\pm_{(\tau,\sigma)}( \mathcal{M}\otimes_{\mathsf{F}}\mathcal{N}) & \text{if $\tau$ is real}, \\
c^\pm_{(\tau,\sigma)}( \mathcal{M}\otimes_{\mathsf{F}}\mathcal{N})c^\pm_{(\iota_0, \sigma)}( \mathcal{M}^{[\tau]}\otimes_{\mathsf{F}^{(\tau,\rho)}}\mathcal{N}^{[\rho\tau]}) & \text{if $\tau$ is complex}.
\end{cases}
\end{align}
In Section~\ref{sec:pertenprod}, we observe that $c^{\pm}_{(\tau,\sigma)}(\mathcal{M},\mathcal{N})$ coincides (up to nonzero algebraic multiples) with the product of two polynomial functions of admissible type, one evaluated at the period matrix of $\mathcal{M}$ and the other evaluated at that of $\mathcal{N}$ (see Proposition~\ref{prop:tensor}). This result is given by Yoshida for pure motives defined over $\mathbf{Q}$ \cite[Proposition~12]{yos01}. Then in Section~\ref{sec:intl}, we introduce the (strong) {\em interlace condition} and deduce explicit factorization formulae under this condition (see Proposition~\ref{prop:perdecomp}).  

We continue to use the same notation as in the previous sections.

\subsection{Periods of the tensor product of motives}\label{sec:pertenprod}


In the following, let $\bullet$ denote either  $\mathcal{M}$ or $\mathcal{N}$. 
We define ${\mathcal B}^\pm_\sigma(\bullet_\tau)$ in the same manner as (\ref{eq:Bsigma}), and define ${\mathcal T}_{\tau,\sigma}(\bullet)$ and $\{ {\mathcal T}^{i^{\bullet_{\tau,\sigma}}_{\mu}}_{\tau,\sigma}(\bullet)\}_{1\leq \mu\leq \mathsf{t}^{\bullet_{\tau,\sigma}}+1}$ as in (\ref{eq:dRst}), (\ref{eq:tfil}) and (\ref{eq:filindices}). Set
\begin{align*}
d_\ast (\bullet) &=\dim_{E^{\sigma}\mathsf{F}^{(\tau,\rho)}} \mathcal{B}_\sigma(\bullet_\tau)=
\begin{cases}
d(\bullet) & \text{if $\tau\in \Sigma_{\mathsf{F},\infty}$ is real}, \\
2d(\bullet) & \text{if $\tau \in \Sigma_{\mathsf{F},\infty}$ is complex},
\end{cases} \\
d^{\pm}_\ast (\bullet_\tau) &=\dim_{E^{\sigma}\mathsf{F}^{(\tau,\rho)}} \mathcal{B}^{\pm}_\sigma(\bullet_\tau)=
\begin{cases}
\dim_{E} H^{\pm}_{\rm B} (\bullet_\tau) & \text{if $\tau\in \Sigma_{\mathsf{F},\infty}$ is real}, \\
d(\bullet) & \text{if $\tau \in \Sigma_{\mathsf{F},\infty}$ is complex},
\end{cases}              \\          
s^{\bullet_{\tau,\sigma}}_{\mu} &=
   \dim_{E^\sigma \mathsf{F}^{(\tau,\rho)}}  {\mathcal T}^{i^{\bullet_{\tau,\sigma}}_{\mu}}_{\tau,\sigma} (\bullet )  
                                            / {\mathcal T}^{i^{\bullet_{\tau,\sigma}}_{\mu+1}}_{\tau,\sigma} (\bullet)  
\qquad \text{for $1\leq \mu \leq \mathsf{t}^{\bullet_{\tau,\sigma}}$}.
\end{align*}
Now let $X_{\tau,\sigma} = X_{\tau,\sigma}^{\mathcal{M}}= 
\begin{pmatrix}
 X^{{\mathcal M},+}_{\tau,\sigma}& X^{{\mathcal M},-}_{\tau,\sigma}
\end{pmatrix}$ 
     (resp.\ $Y_{\tau,\sigma} =Y_{\tau,\sigma}^{\mathcal{N}}= 
\begin{pmatrix}
 Y^{{\mathcal N},+}_{\tau,\sigma} & Y^{{\mathcal N},-}_{\tau,\sigma}
\end{pmatrix}$) denote the period matrix of $\mathcal{M}$ (resp.\ $\mathcal{N}$) with respect to rational bases of $\mathcal{B}^{\pm}_{\sigma}(\mathcal{M}_\tau)$ and $\mathcal{T}_{\tau,\sigma}(\mathcal{M})$ (resp.\ $\mathcal{B}^{\pm}_{\sigma}(\mathcal{N}_\tau)$ and $\mathcal{T}_{\tau,\sigma}(\mathcal{N})$) chosen as in Section~\ref{sec:YosPer}.   
We partition the period matrices $X^{{\mathcal M},\pm}_{\tau,\sigma}$ and $Y^{{\mathcal N},\pm}_{\tau,\sigma}$ into row vectors as follows:
           \begin{align*}
               X^{{\mathcal M}, \pm}_{\tau,\sigma} = \begin{pmatrix}  {\mathfrak X}^\pm_1  \\   \vdots  \\  {\mathfrak X}^\pm_{d_\ast(\mathcal{M})} \end{pmatrix},  \quad   
               Y^{{\mathcal N},\pm}_{\tau,\sigma} = \begin{pmatrix}  {\mathfrak Y}^\pm_1  \\   \vdots  \\  {\mathfrak Y}^\pm_{d_\ast(\mathcal{N})} \end{pmatrix}. 
           \end{align*}  
For a natural number $1\leq i\leq d_*(\mathcal{M})$, we  define the weight $w({\mathfrak X}^\pm_i)$ of $\mathfrak{X}_i^{\pm}$ as $w(\mathfrak{X}_i^{\pm})=i^{\mathcal{M}_{\tau, \sigma}}_\mu$ if the inequality $s^{\mathcal{M}_{\tau,\sigma}}_{1} + \cdots + s^{\mathcal{M}_{\tau,\sigma}}_{\mu-1}+1 \leq i  \leq s^{\mathcal{M}_{\tau,\sigma}}_{1}+ \cdots + s^{\mathcal{M}_{\tau,\sigma}}_{\mu}$ holds. We also define $w(\mathfrak{Y}^{\pm}_i)$ for $1\leq i\leq d_*(\mathcal{N})$ in a similar way.

Set $\mathcal{B}_{\tau,\sigma}^{\mathcal{M},\,\mathcal{N}}:=\mathcal{B}_\sigma(\mathcal{M}_\tau)\otimes_{E^\sigma} \mathcal{B}_\sigma(\mathcal{N}_\tau)$ and $\mathcal{T}_{\tau,\sigma}^{\mathcal{M},\,\mathcal{N}}:=\mathcal{T}_{\tau,\sigma}(\mathcal{M})\otimes_{E^\sigma \mathsf{F}^{(\tau,\rho)}} \mathcal{T}_{\tau,\sigma}(\mathcal{N})$. By construction  $\mathcal{B}_{\tau,\sigma}^{\mathcal{M},\,\mathcal{N}}$ admits an involution 
\begin{align} \label{eq:involution_tensor}
 F^{\mathcal{M},\,\mathcal{N}}_{\tau.\sigma}:= F^{\mathcal{M}}_{\tau,\sigma}\otimes F^{\mathcal{N}}_{\tau,\sigma}
\end{align}
where $F^{\mathcal{M}}_{\tau,\sigma}$ and $F^{\mathcal{N}}_{\tau,\sigma}$ are the involutions defined as in (\ref{eq:involution}). The $(\pm 1)$-eigenspaces according to this involution is denoted as $\bigl(\mathcal{B}_{\tau,\sigma}^{\mathcal{M},\,\mathcal{N}}\bigr)^{\pm}$. Meanwhile $\mathcal{T}_{\tau,\sigma}^{\mathcal{M},\,\mathcal{N}}$ is equipped with the Hodge filtration $\{\mathrm{Fil}^p (\mathcal{T}_{\tau,\sigma}^{\mathcal{M},\,\mathcal{N}}) \}_p$ defined as 
\begin{align*}
 \mathrm{Fil}^p (\mathcal{T}_{\tau,\sigma}^{\mathcal{M},\,\mathcal{N}}) =\sum_{\mu+\nu=p} \mathcal{T}_{\tau,\sigma}^\mu(\mathcal{M}) \otimes_{E^\sigma \mathsf{F}^{(\tau,\rho)}} \mathcal{T}_{\tau,\sigma}^{\nu}(\mathcal{N}).
\end{align*}
The tensor product of the comparison isomorphisms $I_{\tau,\sigma}^0(\mathcal{M}, \mathcal{N}):=I_{\tau,\sigma}^0(\mathcal{M})\otimes I_{\tau,\sigma}^0(\mathcal{N})$ induces 
\begin{align*}
 I_{\tau,\sigma}^{0}(\mathcal{M},\mathcal{N}) \colon \mathcal{B}_{\tau,\sigma}^{\mathcal{M},\,\mathcal{N}} \otimes_{E^\sigma,\iota_0} \mathbf{C}\xrightarrow{\,\sim\,} \mathcal{T}_{\tau,\sigma}^{\mathcal{M},\,\mathcal{N}}\otimes_{E^\sigma \mathsf{F}^{(\tau,\rho)},\iota_0} \mathbf{C}.
\end{align*}
Let us consider the following condition $(\mathrm{Fil}^\mp)_{\tau,\sigma}$ for $\mathcal{M}$, $\mathcal{N}$, $\tau\in I_\mathsf{F}$ and $\sigma\in I_E$;
\begin{itemize}[leftmargin=5em]
 \item[(Fil$^\mp$)$_{\tau,\sigma}$] There exist Hodge filtrations $\mathrm{Fil}^{\mp}(\mathcal{T}_{\tau,\sigma}^{\mathcal{M},\,\mathcal{N}})\in \{\mathrm{Fil}^p(\mathcal{T}_{\tau,\sigma}^{\mathcal{M},\,\mathcal{N}})\}_p$ such that $I_{\tau,\sigma}^0(\mathcal{M}, \mathcal{N})$ induces 
\begin{align*}
 I_{\tau,\sigma}^{0}(\mathcal{M},\mathcal{N})^{\pm} \colon \bigl(\mathcal{B}_{\tau,\sigma}^{\mathcal{M},\,\mathcal{N}}\bigr)^{\pm} \otimes_{E^\sigma,\iota_0} \mathbf{C}\xrightarrow{\,\sim\,} \bigl(\mathcal{T}_{\tau,\sigma}^{\mathcal{M},\,\mathcal{N}}/\mathrm{Fil}^{\mp}(\mathcal{T}_{\tau,\sigma}^{\mathcal{M},\,\mathcal{N}})\bigr)\otimes_{E^\sigma \mathsf{F}^{(\tau,\rho)},\iota_0} \mathbf{C}.
\end{align*}
\end{itemize}
It is easy to observe that the condition (Fil$^\mp$)$_{\tau,\sigma}$ is fullfiled if and only if the action of the involution $F^{\mathcal{M},\, \mathcal{N}}_{\tau,\sigma}$ defined as in (\ref{eq:involution_tensor}) on the diagonal component of the Hodge decomposition of $\mathcal{B}^{\mathcal{M},\,\mathcal{N}}_{\tau,\sigma}$ has only one eigenvalue (either of $\pm 1$), unless the diagonal component is trivial.
If the condition (Fil$^\mp$)$_{\tau,\sigma}$ is fulfilled, we define $q_{\tau,\sigma}^{\mp}\in \mathbf{Z}$ as the maximal integers satisfying 
\begin{align} \label{eq:qmp}
\mathrm{Fil}^{\mp}(\mathcal{T}_{\tau,\sigma}^{\mathcal{M},\,\mathcal{N}})=\mathrm{Fil}^{q^{\mp}_{\tau,\sigma}}(\mathcal{T}_{\tau,\sigma}^{\mathcal{M},\,\mathcal{N}}).
\end{align}

We are now ready to introduce the main result of this subsection, which states how the $(\tau,\sigma)$-period $c^{\pm}_{(\tau,\sigma)}(\mathcal{M},  \mathcal{N})$ (recall the definition (\ref{eq:cmn})) can be  described in terms of $(\tau,\sigma)$-fundamental periods of $\mathcal{M}$ and $\mathcal{N}$ (in an implicit form). 

\begin{prop}[See also {\cite[Proposition 12]{yos01}}] \label{prop:tensor}
Assume that the condition {\rm (Fil$^{\mp}$)$_{\tau,\sigma}$} is fulfilled for $\mathcal{M}$, $\mathcal{N}$, $\tau\in \Sigma_{\mathsf{F},\infty}$ and  $\sigma\in I_E$. Then there exist $\mathbf{Q}$-rational polynomial functions of admissible type,  $\phi^{\pm}_{\tau,\sigma}(x)$ defined on ${\rm M}_{d^{\pm}_*(\mathcal{M})}(\mathbf{C})$ and $\psi^{\pm}_{\tau,\sigma}(y)$ defined on ${\rm M}_{d^{\pm}_*(\mathcal{N})}(\mathbf{C})$, satisfying
\begin{align*}
c^{\pm}_{(\tau,\sigma)}(\mathcal{M}, \mathcal{N}) \sim_{E^{\sigma} \mathsf{F}^{(\tau,\rho)}}   \phi^\pm_{\tau,\sigma}(X_{\tau,\sigma}) \psi^\pm_{\tau,\sigma}(Y_{\tau,\sigma}). 
\end{align*}
Here $X_{\tau,\sigma}=X^{\mathcal{M}}_{\tau,\sigma}$ and  $Y_{\tau,\sigma}=Y^{\mathcal{N}}_{\tau,\sigma}$ respectively denote the $(\tau,\sigma)$-period matrices of $\mathcal{M}$ and $\mathcal{N}$. The admissible types $\{(a_1^{(\pm)}, \dotsc,a^{(\pm)}_{\mathsf{t}^{\mathcal{M}_{\tau, \sigma}}}) ; (k^{(\pm),+},k^{(\pm),-})\}$ of $\phi^{\pm}_{\tau,\sigma}(x)$  
         and $\{(b_1^{(\pm)}, \dotsc,b^{(\pm)}_{\mathsf{t}^{\mathcal{N}_{\tau, \sigma}}}) ; (l^{(\pm),+},l^{(\pm),-})\}$ of $\psi^{\pm}_{\tau,\sigma}(x)$  are given as follows, which characterize these functions up to constant multiples $($double sign in the same order$)$$:$ 
\begin{align*}
  a^{(\pm)}_\mu  &= \# \{  \lambda    \mid 1\leq \lambda \leq d_\ast(\mathcal{N}),\,  i^{\mathcal{M}_{\tau,\sigma}}_{\mu} + w({\mathfrak Y}^+_\lambda)   < q^{\mp}_{\tau,\sigma}  \} 
                         \qquad \text{for $1\leq \mu \leq \mathsf{t}^{\mathcal{M}_{\tau,\sigma}}$},   \\ 
  b^{(\pm)}_\nu  &= \# \{    \lambda  \mid 1\leq \lambda \leq d_\ast(\mathcal{M}), \, i^{\mathcal{N}_{\tau,\sigma}}_{\nu} + w({\mathfrak X}^+_\lambda)   < q^{\mp}_{\tau,\sigma}  \}   
                           \qquad \text{for $1\leq \nu \leq \mathsf{t}^{\mathcal{N}_{\tau,\sigma}}$}, \\ 
 k^{(\pm),+} &= d^\pm_\ast(\mathcal{N}_\tau ), \quad k^{(\pm),-} = d^\mp_\ast(\mathcal{N}_\tau), \quad 
      l^{(\pm),+}= d^\pm_\ast(\mathcal{M}_\tau ), \quad l^{(\pm),-} = d^\mp_\ast(\mathcal{M}_\tau ).  
\end{align*}
\end{prop}

\begin{proof}
As is discussed in the proof of \cite[Proposition~12]{yos01}, the determinants of $I_{\tau,\sigma}^0(\mathcal{M},\mathcal{N})^{\pm}$ are regarded as the evaluation of polynomial functions $h^{\pm}(x,y)$ at $(x,y)=(X_{\tau,\sigma},Y_{\tau,\sigma})$. One then observes that  $h^{\pm}(x,y)$ are of admissible types with respect to each of $x$ and $y$, and thus, by \cite[Theorem~1]{yos01}, they can be expressed as the products of polynomial functions $\phi^\pm_{\tau,\sigma}(x)$ and $\psi^{\pm}_{\tau,\sigma}(y)$ up to constant multiples (double sign in the same order).  One may determine the admissible types of $\phi^{\pm}_{\tau,\sigma}(x)$ and $\psi^{\pm}_{\tau,\sigma}(y)$ by the same computation as in \cite[Proposition~12]{yos01}, and thus it suffices to verify that $\det I_{\tau,\sigma}^0(\mathcal{M},\mathcal{N})^{\pm}$ coincide with $c^{\pm}_{\tau,\sigma}(\mathcal{M}, \mathcal{N})$ up to multiplication by nonzero elements of $E^\sigma \mathsf{F}^{(\tau,\rho)}$.

This is straightforward for real $\tau \in \Sigma_{\mathsf{F},\infty}$, and thus let us suppose that $\tau$ is complex.  In the case, by taking the formulae
\begin{align*}
 I_{\tau,\sigma}(\mathcal{M})\otimes I_{\tau,\sigma}(\mathcal{N}) &=I_{\tau,\sigma}(\mathcal{M}\otimes_{\mathsf{F}}\mathcal{N}), \\
 I_{\tau,\sigma}(\mathcal{M})\otimes I_{\rho\tau, \sigma}(\mathcal{N}) &=I_{\iota_0, \sigma}(\mathcal{M}^{[\tau]})\otimes I_{\iota_0, \sigma}(\mathcal{N}^{[\rho\tau]}) =I_{\iota_0, \sigma}(\mathcal{M}^{[\tau]}\otimes_{\mathsf{F}^{(\tau,\rho)}}\mathcal{N}^{[\rho\tau]}) 
\end{align*}
into accounts, we can calculate as 
\begin{align*}
 I_{\tau,\sigma}^0(\mathcal{M},\mathcal{N})&=I^0_{\tau,\sigma}(\mathcal{M})\otimes I_{\tau,\sigma}^0(\mathcal{N}) 
       =\left(I_{\tau,\sigma}(\mathcal{M})\oplus I_{\rho\tau, \sigma}(\mathcal{M})\right) \otimes \left(I_{\tau,\sigma}(\mathcal{N})\oplus I_{\rho\tau, \sigma}(\mathcal{N})\right) \\
&=\left(I_{\tau,\sigma}(\mathcal{M} \otimes_{\mathsf{F}} \mathcal{N})\oplus I_{\rho\tau, \sigma}(\mathcal{M} \otimes_{\mathsf{F}}\mathcal{N})\right) \\
&\qquad \quad  \oplus \left(I_{\iota_0, \sigma}(\mathcal{M}^{[\tau]}\otimes_{\mathsf{F}^{(\tau,\rho)}} \mathcal{N}^{[\rho\tau]}) \oplus I_{\rho \iota_0, \sigma} (\mathcal{M}^{[\tau]}\otimes_{\mathsf{F}^{(\tau,\rho)}} \mathcal{N}^{[\rho\tau]})\right) \\
&=I^0_{\tau,\sigma}(\mathcal{M}\otimes_{\mathsf{F}}\mathcal{N})\oplus I^0_{\iota_0, \sigma}(\mathcal{M}^{[\tau]}\otimes_{\mathsf{F}^{(\tau,\rho)}}\mathcal{N}^{[\rho\tau]}).
\end{align*}
Therefore we obtain
\begin{align*}
 \det \left(I_{\tau,\sigma}^0(\mathcal{M},\mathcal{N})\right)^{\pm} 
     &=\det I_{\tau,\sigma}^0(\mathcal{M}\otimes_{\mathsf{F}}\mathcal{N})^{\pm} \det I_{\iota_0, \sigma}^0(\mathcal{M}^{[\tau]}\otimes_{\mathsf{F}^{(\tau,\rho)}} \mathcal{N}^{[\rho\tau]})^{\pm} \\
&=c_{(\tau,\sigma)}^{\pm}(\mathcal{M}\otimes_{\mathsf{F}}\mathcal{N})c_{(\iota_0, \sigma)}^{\pm}(\mathcal{M}^{[\tau]}\otimes_{\mathsf{F}^{(\tau,\rho)}} \mathcal{N}^{[\rho\tau]}) = c^{\pm}_{(\tau,\sigma)}(\mathcal{M},\mathcal{N})
\end{align*}
as required.
\end{proof}

Since the admissible types of $\phi^{\pm}_{(\tau,\sigma)}(x)$ and $\psi^{\pm}_{(\tau,\sigma)}(y)$ are determined in Proposition~\ref{prop:tensor}, it is not difficult to express $c^{\pm}_{(\tau,\sigma)}(\mathcal{M},\mathcal{N})$ as the products of $(\tau,\sigma)$-fundamental periods of $\mathcal{M}$ and $\mathcal{N}$, due to \cite[Theorem~3]{yos01}. Indeed Bhagwat works on this task in \cite{bha15} when the base field is the rational number field $\mathbf{Q}$. In the next subsection, we impose certain constraints on the Hodge types of $\mathcal{M}$ and $\mathcal{N}$, which we call the (strong) {\em interlace condition}, and then deduce explicit factorization formulae of $c^{\pm}_{(\tau,\sigma)}(\mathcal{M},\mathcal{N})$.

\subsection{The interlace condition for motives}\label{sec:intl}

Let $\mathcal{M}$ and $\mathcal{N}$ be as in the previous subsection. In the following, we use the symbol $\bullet$ to express either $\mathcal{M}$ or $\mathcal{N}$. For each $\tau\in I_{\mathsf{F}}$ and $\sigma\in I_E$, write the Hodge decomposition of $H_{\rm B}(\bullet_\tau)\otimes_{E,\sigma}\mathbf{C}$ as 
\begin{align*}
    H_{\rm B}(\bullet_\tau) \otimes_{E,\sigma} {\mathbf C} 
     &=  \bigoplus^{t^{\bullet_{\tau,\sigma}}}_{ i = 1 } H^{p^{\bullet_{\tau,\sigma}}_{i},\, q^{\bullet_{\tau,\sigma}}_{i}}(\bullet_\tau)_\sigma, \\
 & H^{p^{\bullet_{\tau,\sigma}}_{i},\, q^{\bullet_{\tau,\sigma}}_{i}}(\bullet_\tau)_\sigma \neq 0 \quad  \text{for } i=1,2,\dotsc,t^{\bullet_{\tau,\sigma}}. 
\end{align*}
Without loss of generality, we may assume that the sequence $\{p^{\bullet_{\tau,\sigma}}_i \}_{1\leq i\leq t^{\bullet_{\tau,\sigma}}}$ is strictly increasing.  
Put $\boldsymbol{p}^{\bullet_{\tau,\sigma}} = (p^{\bullet_{\tau,\sigma}}_{1}, p^{\bullet_{\tau,\sigma}}_{2}, \dotsc , p^{\bullet_{\tau,\sigma}}_{t^{\bullet_{\tau,\sigma}}})$ and $\boldsymbol{q}^{\bullet_{\tau,\sigma}} = (q^{\bullet_{\tau,\sigma}}_{1}, q^{\bullet_{\tau,\sigma}}_{2}, \dotsc , q^{\bullet_{\tau,\sigma}}_{t^{\bullet_{\tau,\sigma}}})$. 

\begin{rem} \label{rem:symmetry}
 Since $F^{\bullet}_{\infty,\tau}\otimes \mathrm{id}$ induces an isomorphism between $H^{p^{\bullet_{\tau,\sigma}}_i,q^{\bullet_{\tau,\sigma}}_i}(\bullet_\tau)_\sigma$ and $H^{q^{\bullet_{\tau,\sigma}}_i,p^{\bullet_{\tau,\sigma}}_i}(\bullet_{\rho\tau})_\sigma$, we have $t^{\bullet_{\tau,\sigma}}=t^{\bullet_{\rho\tau,\sigma}}$, $p^{\bullet_{\tau,\sigma}}_i=q^{\bullet_{\rho\tau,\sigma}}_{t^{\bullet_{\tau,\sigma}}+1-i}$ and $q^{\bullet_{\tau,\sigma}}_i=p^{\bullet_{\rho\tau,\sigma}}_{t^{\bullet_{\tau,\sigma}}+1-i}$.
\end{rem}

The Hodge type $\boldsymbol{p}^{\bullet_{\tau,\sigma}}$ (and $\boldsymbol{q}^{\bullet_{\tau,\sigma}}$) behaves compatibly with respect to the action of $\mathrm{Aut}(\mathbf{C})$;

\begin{prop} \label{prop:auto_C}
Let $\alpha$ be an automorphicm of $\mathbf{C}$. Then, for every $\tau\in I_{\mathsf{F}}$ and $\sigma \in I_E$, we have
\begin{align*}
t^{\bullet_{\alpha\tau,\alpha\sigma}}&=t^{\bullet_{\tau,\sigma}}, & \boldsymbol{p}^{\bullet_{\alpha\tau,\alpha\sigma}}&=\boldsymbol{p}^{\bullet_{\tau,\sigma}}, & \boldsymbol{q}^{\bullet_{\alpha\tau,\alpha\sigma}}&=\boldsymbol{q}^{\bullet_{\tau,\sigma}}.
\end{align*}
\end{prop}

\begin{proof}
 See section 3.1 (especially the paragraph after the equation (3.4)) and Remark~A.2 of \cite{hl17}.
\end{proof}

In this section, we always assume the following conditions: for each $\tau \in I_{\mathsf{F}}$ and $\sigma\in I_E$, the equality 
\begin{align} \label{eq:cond_t}
 t^{\mathcal{M}_{\tau,\sigma}} = t^{\mathcal{N}_{\tau,\sigma}} + 1
\end{align}
holds.   
Hereafter we abbreviate $t^{\mathcal{N}_{\tau,\sigma}}$ as $t_{\tau,\sigma}$ for simplicity.

\begin{dfn} \label{def:interlace}
Let $\mathcal{M}$ and $\mathcal{N}$ be pure motives as above, and take $\tau\in \Sigma_{\mathsf{F},\infty}$ and $\sigma\in I_E$. Suppose that the condition (\ref{eq:cond_t}) is fulfilled.
\begin{enumerate}[label=(\arabic*)]
 \item We say that $\mathcal{M}$ and $\mathcal{N}$ satisfy the {\em interlace condition at $\tau$ and $\sigma$}  if there exists an integer $Q\in \mathbf{Z}$ such that the following inequalities are fulfilled: 
\begin{align}\label{eq:intlpq}
 \begin{cases}
  -p^{\mathcal{M}_{\tau,\sigma}}_{t_{\tau,\sigma}+1} < p^{\mathcal{N}_{\tau,\sigma}}_{1}+Q \leq 
 \cdots  \leq -p^{\mathcal{M}_{\tau,\sigma}}_{t_{\tau, \sigma}+2-i}   < p^{\mathcal{N}_{\tau,\sigma}}_{i}+Q  \leq    -p^{\mathcal{M}_{\tau,\sigma}}_{t_{\tau,\sigma}+1-i}<
                        \cdots 
< p^{\mathcal{N}_{\tau,\sigma}}_{t_{\tau,\sigma}}+Q \leq -p^{\mathcal{M}_{\tau,\sigma}}_{1}, \\
  -q^{\mathcal{M}_{\tau,\sigma}}_{1} < q^{\mathcal{N}_{\tau,\sigma}}_{t_{\tau,\sigma}}+Q \leq 
 \cdots  \leq -q^{\mathcal{M}_{\tau,\sigma}}_{i}   < q^{\mathcal{N}_{\tau,\sigma}}_{t_{\tau,\sigma}+1-i} +Q \leq  -  q^{\mathcal{M}_{\tau,\sigma}}_{i+1}< 
                        \cdots 
< q^{\mathcal{N}_{\tau,\sigma}}_{1}+Q < -q^{\mathcal{M}_{\tau,\sigma}}_{t_{\tau,\sigma}+1}.
 \end{cases}         
\end{align}  
The condition on the Hodge types $(\boldsymbol{p}^{\mathcal{M}_{\tau,\sigma}},\boldsymbol{q}^{\mathcal{M}_{\tau,\sigma}})$ and $(\boldsymbol{p}^{\mathcal{N}_{\tau,\sigma}},\boldsymbol{q}^{\mathcal{N}_{\tau,\sigma}})$ captured by (\ref{eq:intlpq}) is denoted as $\boldsymbol{p}^{\mathcal{M}_{\tau,\sigma}} \succ \boldsymbol{p}^{\mathcal{N}_{\tau,\sigma}}$ and $\boldsymbol{q}^{\mathcal{M}_{\tau,\sigma}} \succ \boldsymbol{q}^{\mathcal{N}_{\tau,\sigma}}$.
\item We say that $\mathcal{M}$ and $\mathcal{N}$ satisfy the {\em strong interlace condition at $\tau$ and $\sigma$} if there exists an integer $Q\in \mathbf{Z}$ such that the following inequalities are fulfilled: 
\begin{align}\label{eq:strintlpq}
\begin{cases}
 -\min\{ p_{t_{\tau,\sigma}+2-i}^{\mathcal{M}_{\tau,\sigma}}, q_{i}^{\mathcal{M}_{\tau,\sigma}}\} < \min\{ p_i^{\mathcal{N}_{\tau,\sigma}}, q_{t_{\tau,\sigma}+1-i}^{\mathcal{N}_{\tau,\sigma}} \} +Q \\
  \max\{ p_i^{\mathcal{N}_{\tau,\sigma}}, q_{t_{\tau,\sigma}+1-i}^{\mathcal{N}_{\tau,\sigma}} \} +Q \leq - \max\{ p_{t_{\tau,\sigma}+1-i}^{\mathcal{M}_{\tau,\sigma}}, q_{i+1}^{\mathcal{M}_{\tau,\sigma}}\}
\end{cases}
 \qquad \text{for }i=1,2,\dotsc,t_{\tau,\sigma}.
\end{align}  
The condition on the Hodge types $(\boldsymbol{p}^{\mathcal{M}_{\tau,\sigma}},\boldsymbol{q}^{\mathcal{M}_{\tau,\sigma}})$ and $(\boldsymbol{p}^{\mathcal{N}_{\tau,\sigma}},\boldsymbol{q}^{\mathcal{N}_{\tau,\sigma}})$ captured by (\ref{eq:strintlpq}) is denoted as $\boldsymbol{p}^{\mathcal{M}_{\tau,\sigma}} \succsim \boldsymbol{p}^{\mathcal{N}_{\tau,\sigma}}$ and $\boldsymbol{q}^{\mathcal{M}_{\tau,\sigma}} \succsim \boldsymbol{q}^{\mathcal{N}_{\tau,\sigma}}$.
\end{enumerate}

\end{dfn}

\begin{rem}
\begin{enumerate}[label=(\arabic*)]
 \item For real $\tau\in \Sigma_{\mathsf{F},\infty}$, the latter inequalities in (\ref{eq:intlpq}) follow from the former ones due to the symmetry of the Hodge types (see Remark~\ref{rem:symmetry}). 
 \item Using Remark~\ref{rem:symmetry}, one readily observes that the strong interlace condition at $\tau$ and $\sigma$ is fulfilled if and only if the condition $\boldsymbol{p}^{\mathcal{M}_{\tilde{\tau},\sigma}} \succ \boldsymbol{p}^{\mathcal{N}_{\tilde{\tau}',\sigma}}$ holds for every $\tilde{\tau},\tilde{\tau}'\in \{\tau, \rho\tau\}$; in particular,  the strong interlace condition implies the interlace condition at a complex place, and these two conditions are equivalent at a real place.
 \item If the condition (\ref{eq:cond_t}) and the (strong) interlace condition are fulfilled for every archimedean place $\tau\in \Sigma_{\mathsf{F},\infty}$ and {\em one} embedding $\sigma_0 \in I_E$ of $E$ into $\mathbf{C}$, those conditions are satisfied for {\em every} $\tau \in \Sigma_{\mathsf{F},\infty}$ and $\sigma\in I_E$; in fact, we have $t^{\bullet_{\tau,\sigma}}=t^{\bullet_{\alpha^{-1}\tau,\sigma_0}}$,  $p^{\bullet_{\tau,\sigma}}_i=p^{\bullet_{\alpha^{-1}\tau,\sigma_0}}_i$ and $q^{\bullet_{\tau,\sigma}}_i=q^{\bullet_{\alpha^{-1}\tau,\sigma_0}}_i$ due to Proposition~\ref{prop:auto_C} if $\alpha$ denotes an automorphism of $\mathbf{C}$ satisfying $\sigma=\alpha\sigma_0$.
\end{enumerate}
\end{rem}

\begin{lem}\label{lem:qmp}
Let $\mathcal{M}$ and $\mathcal{N}$ be pure motives defined over $\mathsf{F}$ with coefficients in $E$ as before, and suppose that the condition $(\ref{eq:cond_t})$ and the strong interlace condition $(\ref{eq:strintlpq})$ at $\tau\in \Sigma_{\mathsf{F},\infty}$ and $\sigma\in I_E$. Let $\mathcal{H}^{\mathcal{M},\,\mathcal{N}}_{\tau,\sigma}$ be the multiset consisting of the Hodge types of $\mathcal{B}^{\mathcal{M},\,\mathcal{N}}_{\tau,\sigma}$$:$
\begin{align*}
\mathcal{H}^{\mathcal{M},\,\mathcal{N}}_{\tau,\sigma}:=
\begin{cases}
\left\{ \left. \left(p_i^{\mathcal{M}_{\tau,\sigma}}+p_j^{\mathcal{N}_{\tau,\sigma}}, q_i^{\mathcal{M}_{\tau,\sigma}}+q_j^{\mathcal{N}_{\tau,\sigma}}\right) \, \right| \, 1\leq i\leq t_{\tau,\sigma}+1,\, 1\leq j\leq t_{\tau,\sigma} \right\} & \text{if $\tau$ is real}, \\[.7em]
\left\{ \left. \left(p_i^{\mathcal{M}_{\tilde{\tau},\sigma}}+p_j^{\mathcal{N}_{\tilde{\tau}',\sigma}}, q_i^{\mathcal{M}_{\tilde{\tau},\sigma}}+q_j^{\mathcal{N}_{\tilde{\tau}',\sigma}}\right) \, \right| \, \parbox{12em}{$1\leq i\leq t_{\tau,\sigma}+1,\, 1\leq j\leq 
t_{\tau,\sigma}$ \\  $\tilde{\tau},\tilde{\tau}'\in \{\tau, \rho\tau\}$} \right\}	& \text{if $\tau$ is complex}.						
\end{cases}
\end{align*}
Define $q_{\tau,\sigma}$ as 
\begin{align} \label{eq:qmp_tau_sigma}
  q_{\tau,\sigma}
      =   \min \left\{ \left.   p^{\mathcal{M}_{\tilde{\tau},\sigma}}_{t_{\tau,\sigma}+2- i} +  p^{\mathcal{N}_{\tilde{\tau}',\sigma}}_{i}  \, \right| \, 1 \leq i \leq t_{\tau,\sigma}, \; \tilde{\tau}, \tilde{\tau}'\in \{\tau, \rho\tau \}      \right\}.  
\end{align}
Then the multisubsets
\begin{align*}
 \mathcal{H}^{\mathcal{M},\,\mathcal{N}}_{\tau,\sigma,<q_{\tau,\sigma}} &= \{(p,q) \in \mathcal{H}^{\mathcal{M},\,\mathcal{N}}_{\tau,\sigma} \mid p<q_{\tau, \sigma}\}, & \mathcal{H}^{\mathcal{M},\,\mathcal{N}}_{\tau,\sigma,\geq q_{\tau,\sigma}} &= \{(p,q) \in \mathcal{H}^{\mathcal{M},\,\mathcal{N}}_{\tau,\sigma} \mid p\geq q_{\tau, \sigma}\}
\end{align*}
give an equal-cardinality partition of $\mathcal{H}^{\mathcal{M},\,\mathcal{N}}_{\tau,\sigma}$. 

\end{lem}

The following corollary immediate follows from Lemma~\ref{lem:qmp}, which plays an important role in the calculation of the explicit factorization formula (Proposition~\ref{prop:perdecomp}).

\begin{cor} \label{cor:qmp}
 If the pure motives $\mathcal{M}$ and $\mathcal{N}$ satisfy the condition $(\ref{eq:cond_t})$ and the strong interlace condition $(\ref{eq:strintlpq})$ at $\tau\in \Sigma_{\mathsf{F},\infty}$ and $\sigma\in I_E$, the condition $({\rm Fil}^\mp)_{\tau,\sigma}$ is fulfilled for $\mathcal{M}$ and $\mathcal{N}$, and both of $q^\mp_{\tau,\sigma}$ appearing in $(\ref{eq:qmp})$ equal $q_{\tau,\sigma}$ defined as $(\ref{eq:qmp_tau_sigma})$.
\end{cor}

\begin{proof}
 Lemma~\ref{lem:qmp} implies that the diagonal component of the Hodge decomposition of $\mathcal{B}^{\mathcal{M},\, \mathcal{N}}_{\tau,\sigma}$ is trivial, and thus the condition (Fil$^\mp$)$_{\tau,\sigma}$ is fulfilled. The rest of the statement is a direct consequence of Lemma~\ref{lem:qmp}.
\end{proof}

\begin{proof}[Proof of Lemma $\ref{lem:qmp}$]
We only consider the case where $\tau\in \Sigma_{\mathsf{F},\infty}$ is a complex place; the case where $\tau$ is real is verified by similar and easier arguments. First note that the multiset $\mathcal{H}^{\mathcal{M},\, \mathcal{N}}_{\tau,\sigma}$, consisting of $4t_{\tau,\sigma}(t_{\tau,\sigma}+1)$ elements (with multiplicity), is {\em symmetric}; that is, if $(p_1,q_1), (p_2,q_2),\dotsc, (p_{\# \mathcal{H}^{\mathcal{M},\, \mathcal{N}}_{\tau,\sigma}},q_{\# \mathcal{H}^{\mathcal{M},\, \mathcal{N}}_{\tau,\sigma}})$ denote the elements of $\mathcal{H}^{\mathcal{M},\,\mathcal{N}}_{\tau,\sigma}$ arranged so that their first components $p_1,p_2,\dotsc,p_{\# \mathcal{H}^{\mathcal{M},\, \mathcal{N}}_{\tau,\sigma}}$ are in ascending order, the equality  $p_{\# \mathcal{H}^{\mathcal{M},\, \mathcal{N}}_{\tau,\sigma}+1-i}=p_i$ holds for $i=1,2,\dotsc, \# \mathcal{H}^{\mathcal{M},\, \mathcal{N}}_{\tau,\sigma}=4t_{\tau,\sigma}(t_{\tau,\sigma}+1)$. It suffices to show that the cardinality of $\mathcal{H}^{\mathcal{M},\,\mathcal{N}}_{\tau,\sigma, <q_{\tau,\sigma}}$ (with multiplicity) equals $2t_{\tau,\sigma}(t_{\tau,\sigma}+1)$. Taking 
\begin{align*}
    \# \left\{   (a, k) \ \middle| \ 2 \leq a \leq t_{\tau,\sigma}+1, 1 \leq k \leq a-1   \right\}
  = \frac{1}{2} t_{\tau,\sigma} (t_{\tau,\sigma}+1)
\end{align*}
into accounts, the proof is reduced to check the following two claims:
\begin{enumerate}
 \item For each $2\leq a\leq t_{\tau,\sigma}+1$, $1\leq k\leq a-1$, $1\leq i \leq t_{\tau,\sigma}$ and $\tilde{\tau}, \tilde{\tau}', \tilde{\tau}'', \tilde{\tau}''' \in \{\tau,\rho\tau\}$, we have
            \begin{align*}
                p^{\mathcal{M}_{\tilde{\tau}'',\sigma}}_{a-k} + p^{\mathcal{N}_{\tilde{\tau}''',\sigma}}_{k}  
                <  p^{\mathcal{M}_{\tilde{\tau},\sigma}}_{ t_{\tau,\sigma}+2-i} + p^{\mathcal{N}_{\tilde{\tau}',\sigma}}_{i}.     
            \end{align*}       
 \item For each $t_{\tau,\sigma}+3 \leq a \leq 2t_{\tau,\sigma}+1$,  $a-t_{\tau,\sigma} \leq k \leq t_{\tau,\sigma}+1$ and $\tilde{\tau}'', \tilde{\tau}''' \in \{\tau, \rho\tau\}$,  
            there exist $1\leq i \leq t_{\tau,\sigma}$ and $\tilde{\tau}, \tilde{\tau}'\in \{\tau,\rho\tau\}$ 
satisfying the inequality
            \begin{align*}
                  p^{\mathcal{M}_{\tilde{\tau},\sigma}}_{t_{\tau,\sigma}+2-i} + p^{\mathcal{N}_{\tilde{\tau}',\sigma}}_{ i}  
                  \leq  p^{\mathcal{M}_{\tilde{\tau}'',\sigma}}_{a- k}  + p^{\mathcal{N}_{\tilde{\tau}''',\sigma}}_{k}.   
            \end{align*}          
\end{enumerate}

For the statement (i), we have
\begin{align*}   
   p^{\mathcal{M}_{\tilde{\tau},\sigma}}_{t_{\tau,\sigma}+2-  i} + p^{\mathcal{N}_{\tilde{\tau}',\sigma}}_{i}  
       - \left(  p^{\mathcal{M}_{\tilde{\tau}'',\sigma}}_{a- k} + p^{\mathcal{N}_{\tilde{\tau}''',\sigma}}_{ k}   \right)    
    &=   p^{\mathcal{M}_{\tilde{\tau},\sigma}}_{ t_{\tau,\sigma}+2- i} + p^{\mathcal{N}_{\tilde{\tau}',\sigma}}_{i}  
       - \left( p^{\mathcal{M}_{\tilde{\tau}'',\sigma}}_{t_{\tau,\sigma}+1-k-(t_{\tau,\sigma}+1-a)} + p^{\mathcal{N}_{\tilde{\tau}''',\sigma}}_{ k}   \right)  \\ 
&\geq  p^{\mathcal{M}_{\tilde{\tau},\sigma}}_{ t_{\tau,\sigma}+2- i} + p^{\mathcal{N}_{\tilde{\tau}',\sigma}}_{i}  
       - \left( p^{\mathcal{M}_{\tilde{\tau}'',\sigma}}_{t_{\tau,\sigma}+1-k} + p^{\mathcal{N}_{\tilde{\tau}''',\sigma}}_{ k}   \right) \\
&>(-Q)-(-Q)=0
\end{align*}
as required. The first inequality holds since $\{p^{\mathcal{M}_{\tau,\sigma}}_\nu\}_\nu$ is non-decreasing (note that $t_{\tau,\sigma}+1-a\geq 0$ holds), and the second inequality holds due to the strong interlace condition (\ref{eq:strintlpq}) for $Q\in {\mathbf Z}$.
Next, the second claim (ii) is fulfilled for $i=k$, $\tilde{\tau}=\tilde{\tau}''$ and $\tilde{\tau}'=\tilde{\tau}'''$. 
In fact, we have
\begin{align*}
   \left( p^{\mathcal{M}_{\tilde{\tau}'',\sigma}}_{a-k} + p^{\mathcal{N}_{\tilde{\tau}''',\sigma}}_{ k}  \right)
     -   \left(    p^{\mathcal{M}_{\tilde{\tau}'',\sigma}}_{t_{\tau,\sigma}+2- k}   + p^{\mathcal{N}_{\tilde{\tau}''',\sigma}}_{ k}   \right)
   = p^{\mathcal{M}_{\tilde{\tau}'',\sigma}}_{ a-k}  - p^{\mathcal{M}_{\tilde{\tau}'',\sigma}}_{ t_{\tau,\sigma}+2-k},
\end{align*}
which is non-negative since $\{p^{\mathcal{M}_{\tilde{\tau}'',\sigma}}_\nu\}_\nu$ is non-decreasing (note that $(a-k) - (t_{\tau,\sigma}+2-k) = a-(t_{\tau,\sigma}+2) \geq 1$ holds).
\end{proof}

\begin{dfn} \label{dfn:regular}
 Let $\mathcal{M}$ be a pure motive defined over $\mathsf{F}$ with coefficients in $E$. For an archimedean place $\tau\in \Sigma_{\mathsf{F},\infty}$ of $\mathsf{F}$ and an embedding $\sigma\in I_E$ of $E$ into $\mathbf{C}$, we say $\mathcal{M}$ to be {\em regular at $\tau$ and $\sigma$} if every component of the Hodge decomposition of $\mathcal{B}_\sigma(\mathcal{M}_\tau)$ is of dimension at most $1$. If $\mathcal{M}$ is regular at every $\tau \in \Sigma_{\mathsf{F},\infty}$ and $\sigma\in I_E$, we say $\mathcal{M}$ to be {\em regular}.
\end{dfn}

Note that, if $\mathcal{M}$ is regular at $\tau$ and $\sigma$, the number of jumps $t^{\mathcal{M}_{\tau,\sigma}}$ of the Hodge filtration equals the rank $d(\mathcal{M})$ of $\mathcal{M}$.
Now let us introduce the main result of this section. Set $c_{0,(\tau,\sigma)}(\mathcal{N})=\delta_{(\tau,\sigma)}(\mathcal{N})$ as convention. 

\begin{prop}\label{prop:perdecomp}
Let $\mathcal{M}$ and $\mathcal{N}$ be pure motives as above, which are regular at $\tau\in \Sigma_{\mathsf{F},\infty}$ ant $\sigma\in I_E$.
\begin{enumerate}[label={\rm (\arabic*)}]
  \item Suppose that $\tau$ is real. Further assume that $\mathcal{M}$ and $\mathcal{N}$ satisfy the condition $(\ref{eq:cond_t})$ and the interlace condition $(\ref{eq:intlpq})$ at $\tau$ and $\sigma$. Then, when $d(\mathcal{N})$ is even,  we have   
           \begin{align*}
             c^\pm_{(\tau,\sigma)} (\mathcal{M} ,\mathcal{N})  
               &= \prod^{\frac{d(\mathcal{N})}{2}}_{\beta=1}  c_{\beta,(\tau,\sigma)}(\mathcal{M}) 
                   \prod^{\frac{d(\mathcal{N})}{2}-1}_{\beta=0}  c_{\beta, (\tau,\sigma)}(\mathcal{N})  
                  \cdot \begin{cases} c^\pm_{(\tau,\sigma)}(\mathcal{N})   &  \text{if $d^+_\ast(\mathcal{M}_\tau) > d^-_\ast(\mathcal{M}_\tau)$ holds},   \\
                                                   c^\mp_{(\tau,\sigma)}(\mathcal{N})   &  \text{if $d^+_\ast(\mathcal{M}_\tau) < d^-_\ast(\mathcal{M}_\tau)$ holds}.       \end{cases}
         \end{align*}
When $d(\mathcal{N})$ is odd, we have 
           \begin{align*}
            c^\pm_{(\tau,\sigma)} (\mathcal{M}, \mathcal{N})  
              &= \prod^{\frac{d(\mathcal{N})-1}{2}}_{\beta=1}  c_{\beta,(\tau,\sigma)}(\mathcal{M}) 
                   \prod^{\frac{d(\mathcal{N})-1}{2}}_{\beta=0}  c_{\beta, (\tau,\sigma)}(\mathcal{N})  
                  \cdot \begin{cases} c^\pm_{(\tau,\sigma)}(\mathcal{M})   &  \text{if $d^+_\ast(\mathcal{N}_\tau) > d^-_\ast(\mathcal{N}_\tau)$ holds},   \\
                                                   c^\mp_{(\tau,\sigma)}(\mathcal{M})   &  \text{if $d^+_\ast(\mathcal{N}_\tau) < d^-_\ast(\mathcal{N}_\tau)$ holds}.       \end{cases}
         \end{align*}
   \item Suppose that $\tau$ is complex.  Further assume that $\mathcal{M}$ and $\mathcal{N}$ satisfy the condition $(\ref{eq:cond_t})$ and the strong interlace condition $(\ref{eq:strintlpq})$ at $\tau$ and $\sigma$.
          Then, when $d(\mathcal{N})$ is even, we have   
           \begin{align*}
             c^\pm_{(\tau,\sigma)} (\mathcal{M} ,\mathcal{N})
               &= \prod^{\frac{d(\mathcal{N})}{2}}_{\beta=1}  c_{\beta,(\tau,\sigma)}(\mathcal{M})^2 
                   \prod^{\frac{d(\mathcal{N})}{2}-1}_{\beta=0}  c_{\beta, (\tau,\sigma)}(\mathcal{N})^2  
                  \cdot  c^+_{(\tau,\sigma)}(\mathcal{N})c^-_{(\tau,\sigma)}(\mathcal{N}). 
         \end{align*}
        When $d(\mathcal{N})$ is odd, we have   
           \begin{align*}
              c^\pm_{(\tau,\sigma)} (\mathcal{M} ,\mathcal{N})
                 &= \prod^{\frac{d(\mathcal{N})-1}{2}}_{\beta=1}  c_{\beta,(\tau,\sigma)}(\mathcal{M})^2 
                   \prod^{\frac{d(\mathcal{N})-1}{2}}_{\beta=0}  c_{\beta, (\tau,\sigma)}(\mathcal{N})^2  
                  \cdot  c^+_{(\tau,\sigma)}(\mathcal{M})c^-_{(\tau,\sigma)}(\mathcal{M}).
         \end{align*}
\end{enumerate}
\end{prop}

\begin{proof}
Write $c^\pm_{(\tau,\sigma)}( \mathcal{M},\mathcal{N})= \phi^\pm_{\tau,\sigma}(X_{\tau,\sigma}) \psi^\pm_{\tau,\sigma}(Y_{\tau,\sigma})$ as in Proposition~\ref{prop:tensor}. We determine the admissible types of $\phi_{\tau,\sigma}^{\pm}(x)$ and $\psi_{\tau,\sigma}^{\pm}(y)$ in each case.

\begin{enumerate}[label=(\arabic*),leftmargin=2em]
 \item Suppose that $\tau$ is real. 
It suffices to show that the admissible types of $\phi_{\tau,\sigma}^{\pm}(x)$ and $\psi_{\tau,\sigma}^{\pm}(y)$ are given as follows.

\begin{center}
 \begin{tabular}{@{\vrule width 1pt }@{\,}c@{\,}|@{\,}c@{\,}|@{\,}c@{\,}@{\vrule width 1pt}}
\hline
 & $d(\mathcal{N})$: even & $d(\mathcal{N})$: odd \\ \hline 
$\phi^\pm_{\tau,\sigma}(x)$ & $\{ (d(\mathcal{N}), d(\mathcal{N})-1, \dotsc, 1, 0) ;  ( \frac{d(\mathcal{N})}{2} , \frac{d(\mathcal{N})}{2})  \}$ &  
$\{ (d(\mathcal{N}), d(\mathcal{N})-1, \dotsc, 1,0) ;  ( d^{\pm}(\mathcal{N}_\tau), d^{\mp}(\mathcal{N}_\tau)) \}$\\ \hline
$\psi^\pm_{\tau,\sigma}(y)$ & 
$\{ (d(\mathcal{N}), d(\mathcal{N})-1, \dotsc, 1) ;  ( d^\pm(\mathcal{M}_\tau), d^\mp(\mathcal{M}_\tau)) \}$
& $\{(d(\mathcal{N}),d(\mathcal{N})-1,\dotsc,1) ; (\frac{d(\mathcal{N})+1}{2}, \frac{d(\mathcal{N})+1}{2})\}$
\\ \hline
\end{tabular}
\end{center}

Note that, if $d(\bullet)$ is odd for $\bullet=\mathcal{M}$ or $\mathcal{N}$, we have $d^+(\bullet_\tau)=\frac{d(\bullet)+1}{2}$ (resp.\ $\frac{d(\bullet)-1}{2}$) and $d^-(\bullet_\tau)=\frac{d(\bullet)-1}{2}$ (resp.\ $\frac{d(\bullet)+1}{2}$) when the inequality $d^+(\bullet_\tau)>d^-(\bullet_\tau)$ (resp.\ $d^+(\bullet_\tau)<d^-(\bullet_\tau)$) holds. 

In fact, if we have determined the admissible types of $\phi^\pm_{\tau,\sigma}(x)$ and $\psi^\pm_{\tau,\sigma}(y)$ as above, it is easy to deduce the desired formulae. For example, 
\begin{align*}
 &\{ (d(\mathcal{N}), d(\mathcal{N})-1, \ldots, 1, 0 ) ; (d^{\pm}(\mathcal{N}_\tau), d^{\mp}(\mathcal{N}_\tau)) \} \\
&   = \sum^{\frac{d(\mathcal{N})-1}{2}}_{\beta=1}  \{ ( \overbrace{2, \dotsc, 2}^{\beta},  \overbrace{1, \ldots, 1}^{d(\mathcal{N})-2\beta},   \overbrace{0, \ldots, 0}^{\beta}  ) ;  (1,1) \} \\
&\qquad \qquad + 
\begin{cases}
 \{ (\overbrace{1,\dotsc,1}^{\frac{d(\mathcal{N})+1}{2}},0,\dotsc,0): ((1\pm 1)/2,(1\mp 1)/2) \} & \text{if $d^+(\mathcal{N}_\tau)>d^-(\mathcal{N}_\tau)$ holds}, \\
\{ (\overbrace{1,\dotsc,1}^{\frac{d(\mathcal{N})+1}{2}},0,\dotsc,0): ((1\mp 1)/2,(1\pm 1)/2) \} & \text{if $d^+(\mathcal{N}_\tau)<d^-(\mathcal{N}_\tau)$ holds}
\end{cases}
\end{align*}
implies 
\begin{align*}
 \phi^\pm_{\tau,\sigma}(X_{\tau,\sigma})=\prod_{\beta=1}^{\frac{d(\mathcal{N})-1}{2}} c_{\beta,(\tau,\sigma)}(\mathcal{M}) \times 
\begin{cases}
 c^{\pm}_{(\tau,\sigma)}(\mathcal{N}) & \text{if $d^+(\mathcal{N}_\tau)>d^-(\mathcal{N}_\tau)$ holds}, \\
 c^{\mp}_{(\tau,\sigma)}(\mathcal{N}) & \text{if $d^+(\mathcal{N}_\tau)<d^-(\mathcal{N}_\tau)$ holds}
\end{cases}
\end{align*}
when $d(\mathcal{N})$ is odd. By similar calculations, we can express $\phi^\pm_{\tau,\sigma}(X_{\tau,\sigma})$ and $\psi^\pm_{\tau,\sigma}(Y_{\tau,\sigma})$ as the products of Yoshida's fundamental periods of $\mathcal{M}$ and $\mathcal{N}$ in any other cases. Then the desired formulae are obtained due to Proposition~\ref{prop:tensor}. 

Here we only determine the admissible type of $\phi^\pm_{\tau,\sigma}(x)$ for odd $d(\mathcal{N})$; arguments are similar in other cases. 
For each $1 \leq \mu \leq d(\mathcal{M})=d(\mathcal{N})+1$, the definition of $i_\mu^{\mathcal{M}_{\tau,\sigma}}$ implies  $ i_\mu^{\mathcal{M}_{\tau,\sigma}}= p^{\mathcal{M}_{\tau,\sigma}}_\mu$. Meanwhile, for each $1\leq \lambda\leq d(\mathcal{N})$, the equality $w(\mathfrak{Y}_\lambda)=p^{\mathcal{N}_{\tau,\sigma}}_\lambda$ holds by definition. 
The identity (\ref{eq:qmp_tau_sigma}) in Lemma \ref{lem:qmp} implies the following inequalities:
\begin{align*}
         p^{\mathcal{M}_{\tau,\sigma}}_\mu  + p^{\mathcal{N}_{\tau,\sigma}}_1 
      < p^{\mathcal{M}_{\tau,\sigma}}_\mu  + p^{\mathcal{N}_{\tau,\sigma}}_2   
      < \cdots 
      < p^{\mathcal{M}_{\tau,\sigma}}_\mu  + p^{\mathcal{N}_{\tau,\sigma}}_{d(\mathcal{N})+1-\mu}  
      < q_{\tau,\sigma}^{\mp}
      \leq  p^{\mathcal{M}_{\tau,\sigma}}_\mu  + p^{\mathcal{N}_{\tau,\sigma}}_{d(\mathcal{N})+2-\mu}.
\end{align*}   
Thus we have $a_\mu^{(\pm)} = d(\mathcal{N})+1-\mu$. The equality $(k^{(\pm),+},k^{(\pm),-})=(d^\pm (\mathcal{N}_\tau), d^{\mp}(\mathcal{N}_\tau))$ follows from Proposition~\ref{prop:tensor}.

 \item Suppose that $\tau$ is complex. Similarly to (1), it suffices to show that the admissible types of $\phi_{\tau,\sigma}^\pm(x)$ and $\psi_{\tau,\sigma}^\pm(y)$ are given as follows.

\begin{center}
 \begin{tabular}{@{\vrule width 1pt\ }c|c@{\ \vrule width 1pt}}
\hline
$\phi^\pm_{\tau,\sigma}(x)$ & $\{(2d(\mathcal{N}), 2(d(\mathcal{N})-1),\dotsc,2,0) ; (d(\mathcal{N}), d(\mathcal{N}))\}$\\ \hline
$\psi^\pm_{\tau,\sigma}(y)$ & $\{(2d(\mathcal{N}), 2(d(\mathcal{N})-1),\dotsc,4,2) ; (d(\mathcal{N})+1, d(\mathcal{N})+1)\}$ \\\hline
\end{tabular}
\end{center}

The required calculations for determination of the admissible types are very close to those in (1), and thus we omit the details.
\qedhere
\end{enumerate}
\end{proof}

\section{Cohomological automorphic representations}\label{sec:CohAuto}

In this section, we prepare some basic facts about cohomological  irreducible cuspidal automorphic representations.
Clozel's exposition \cite{clo90} is a basic reference of this section. 
Brief summaries can also be found in \cite[Section 3.1.1]{mah05}, \cite[Section 2.3]{gr14} and \cite[Section 2.4]{rag16}.

\subsection{Langlands parameters}\label{sec:LangPara}

We summarize the notion of Langlands parameters of cohomological irreducible cuspidal automorphic representations.

We realize the Weil group $W_{\mathbf R}$ as the set ${\mathbf C}^\times \sqcup ({\mathbf C}^\times j)$ with $j^{2}=-1, jzj = -z^\rho$ for each $z \in {\mathbf C}^\times$.
For each $\nu \in {\mathbf C}, \delta\in \{0, 1\}$ and $l\in {\mathbf Z}_{\geq 0}$, 
define 1-dimensional representation $\phi^\delta_\nu\colon  W_{\mathbf R} \to {\mathbf C}^\times$ 
and 2-dimensional representation $\phi_{\nu, l} \colon W_{\mathbf R} \to {\rm GL}_2({\mathbf C})$ of $W_{\mathbf{R}}$ as follows:
\begin{align*}
     \phi^\delta_\nu (z) &= (zz^\rho)^{\nu} \ (z\in {\mathbf C}^\times),
                                       \quad \phi^\delta_\nu(j) = (-1)^\delta,   \\   
     \phi_{\nu, l} (z)  &=  (zz^\rho)^\nu \begin{pmatrix}  (z^\rho/z)^{\frac{l}{2}}  &  0 \\  0 &  (z/z^\rho)^{\frac{l}{2}}  \end{pmatrix} \ (z \in {\mathbf C}^\times), 
                                        \quad \phi_{\nu, l} (j) =  \begin{pmatrix}  0 & (-1)^l  \\ 1 & 0 \end{pmatrix}. 
\end{align*} 
Define $\chi^\delta_\nu \colon {\mathbf R}^\times \to {\mathbf C}^\times$ to be $\chi^\delta_\nu(t) = {\rm sgn}(t)^\delta |t|^\nu_\infty$. 
We also define an irreducible representation $D_{\nu, l}$ of $\mathrm{GL}_2(\mathbf{R})$ so that 
\begin{align*}
  D_{\nu, l}(t1_2) &= t^{2\nu} \ (t \in {\mathbf R}_{>0}), &
  D_{\nu, l}|_{ {\rm SL}_2({\mathbf R}) } & \cong D^+_l \oplus D^-_l   
\end{align*}
where $D^+_l$ (resp.\ $D^-_l$) is the holomorphic (resp.\ anti-holomorphic) discrete series representation of ${\rm SL}_2({\mathbf R})$  of lowest weight $l+1$ (resp.\ of highest weight $-l-1$). 
Then $\phi^\delta_\nu$ (resp.\ $\phi_{\nu, l}$) corresponds to $\chi^\delta_\nu$ (resp.\ $D_{\nu, l}$) via the local Langlands correspondence. Refer to \cite{kna94} for details on the archimedean local Langlands correspondence.

Now let $\mathsf{F}$ be a number field. In the following,  we consider a {\em cohomological} (or {\em regular algebraic}) irreducible cuspidal automorphic representation $\pi^{(n)}$ of $\mathrm{GL}_n(\mathsf{F}_{\mathbf{A}})$. Then, for any archimedean place $v\in \Sigma_{\mathsf{F},\infty}$, the $v$-component $\pi^{(n)}_v$ of $\pi^{(n)}$ is an irreducible admissible representation of $\mathrm{GL}_n(\mathsf{F}_v)$, and so there exists a corresponding semisimple complex representation $\phi(\pi^{(n)}_v)$ of $W_{\mathbf{R}}$ via the local Langlands correspondence. We call $\phi(\pi^{(n)}_v)$ the {\em Langlands parameter} of $\pi^{(n)}_v$, which is described as follows.

\begin{thm}[{\cite[Theorem 2, Theorem 5]{kna94} and \cite[Section~3.5]{clo90}}]  \label{thm:LLC}
Let $\mathsf{F}$ be a number field and $\pi^{(n)}$ a cohomological irreducible cuspidal automorphic representation of $\mathrm{GL}_n(\mathsf{F}_{\mathbf{A}})$. For each archimedean place $v\in \Sigma_{\mathsf{F}, \infty}$ of $\mathsf{F}$, let $\phi(\pi^{(n)}_v) \colon W_{\mathsf{F}_v} \to {\rm GL}_n({\mathbf C})$ denote the Langlands parameter of the $v$-component $\pi^{(n)}_v$ of $\pi^{(n)}$. 
\begin{itemize}
\item
{\bf ($v$: real, $n$: even)}
Suppose that $v$ is real and $n=2m$ is even.    
             Then  $\phi(\pi^{(n)}_v)$ is given by 
             \begin{align*}
             \phi(\pi^{(n)}_v)
             =  \bigoplus^m_{i=1} \phi_{\nu^{(n)}_{v, i},l^{(n)}_{v, i}} \quad \text{with $l^{(n)}_{v, 1}> l^{(n)}_{v, 2}>\cdots > l^{(n)}_{v, m}\geq 1$},                          \end{align*}
where each $\nu^{(n)}_{v,i}$ is a rational number such that $\nu^{(n)}_{v,i}\pm \dfrac{l^{(n)}_{v,i}}{2}-\dfrac{n-1}{2}$ is an integer.
\item
{\bf ($v$: real, $n$: odd)}
 Suppose that $v$ is real and $n=2m+1$ is odd.    
             Then  $\phi(\pi^{(n)}_v)$ is given by 
             \begin{align*}
             \phi(\pi^{(n)}_v)
             =   \phi^{\delta^{(n)}_v}_{\nu^{(n)}_{v,0}}  \oplus \bigoplus^m_{i=1} \phi_{\nu^{(n)}_{v,i}, l^{(n)}_{v,i}} \quad \text{with $l^{(n)}_{v, 1}> l^{(n)}_{v, 2}>\cdots > l^{(n)}_{v,m}\geq 1$},
             \end{align*} 
where $\delta^{(n)}_v\in \{ 0,1\}$ and each $\nu^{(n)}_{v,i}$ is a rational number such that $\nu^{(n)}_{v,i}\pm \dfrac{l^{(n)}_{v,i}}{2}-\dfrac{n-1}{2}$ is an integer $($with $l^{(n)}_{v,0}:=0)$.

\item
{\bf ($v$: complex)}
             Suppose that $v$ is complex. Then  $\phi(\pi^{(n)}_v)$ is given by 
\begin{align*}
     \phi(\pi^{(n)}_v)(z)
     =   {\rm diag}( z^{a^{(n)}_{v, 1}} (z^\rho)^{b^{(n)}_{v, 1}}  ,    \dotsc ,   z^{a^{(n)}_{v, n}} (z^\rho)^{b^{(n)}_{v, n}}     ), 
 \end{align*}
 where each $a^{(n)}_{v,i}$ $($resp.\ $b^{(n)}_{v,i})$ is a rational number such that $a^{(n)}_{v,i}-\dfrac{n-1}{2}$ $($resp.\ $b^{(n)}_{v,i}-\dfrac{n-1}{2})$ is an integer. 
Furthermore the sequence $\{a^{(n)}_{v,i}\}_{1\leq i\leq n}$ satisfies $a^{(n)}_{v,1}<a^{(n)}_{v,2}<\cdots <a^{(n)}_{v,n}$.

\end{itemize}

\end{thm}

For $v\in \Sigma_{\mathsf{F}, \infty}$, let $\phi(\pi^{(n)}_v) \colon W_{\mathsf{F}_v} \to {\rm GL}_n({\mathbf C})$ be the Langlands parameter of $\pi^{(n)}_v$.  
Then the evaluation of $\phi(\pi^{(n)}_v)$ at  $z \in {\mathbf C}^\times \subset W_{\mathsf{F}_v}$ is described as
\begin{align}\label{eq:LangPara}
    \phi(\pi_v)    (z) 
  &= {\rm diag}( z^{a^{(n)}_{v, 1}} (z^\rho)^{b^{(n)}_{v, 1}}  ,    \ldots ,   z^{a^{(n)}_{v, n}} (z^\rho)^{b^{(n)}_{v, n}}     ). 
%
\end{align}

We now attach the expected Hodge type 
$(p^{(n)}_{\tau, i}, q^{(n)}_{\tau,i})  \in {\mathbf Z}^{\oplus 2}$  for  $i=1, \ldots, n$ and $\tau \in I_{\mathsf{F}}$ to $(a^{(n)}_{v,i},b^{(n)}_{v,i})$ in (\ref{eq:LangPara}), according to \cite[Sections 3.3 and 4.4.3]{clo90}.
Let $\tau = \tau_v$ denote the embedding of $\mathsf{F}$ into $\mathbf{C}$ corresponding to $v\in \Sigma_{\mathsf{F}, \infty}$ via the fixed identification $\Sigma_{\mathsf{F}, \infty} \subset I_\mathsf{F}$. 
Then define $p^{(n)}_{\tau, i}$ and $q^{(n)}_{\tau, i}$ so that 
\begin{align*}
p^{(n)}_{\tau, i} = a^{(n)}_{v, i} - \frac{n-1}{2},   \quad 
q^{(n)}_{\tau, i} = b^{(n)}_{v, i} - \frac{n-1}{2}. 
\end{align*}
Due to regular algebraicity of $\pi^{(n)}$, each $p^{(n)}_{\tau,i}$ is an integer and the inequalities
\begin{align*}
   p^{(n)}_{\tau, 1}  < p^{(n)}_{\tau, 2} < \cdots < p^{(n)}_{\tau, n} 
\end{align*}
holds (see Theorem~\ref{thm:LLC}). Furthermore the purity of cohomological representations (see \cite[Lemme de puret\'{e} 4.9]{clo90} for details) implies that   
there exists an integer $w(\pi^{(n)})$ (called the {\em purity weight} of $\pi^{(n)}$), which is independent of $v\in \Sigma_{\mathsf{F},\infty}$, satisfying 
\begin{align*}
   w(\pi^{(n)}) - (n  -  1) =  a^{(n)}_{v, i} + b^{(n)}_{v, i} - (n - 1) = p^{(n)}_{\tau, i} + q^{(n)}_{\tau, i}.   
\end{align*}
Now define the {\em motivic weight} of $\pi^{(n)}$ as $\mathsf{w}^{(n)}:=w(\pi^{(n)})-(n-1)$.
If $\tau\in I_\mathsf{F}$ is real, we find that 
$ q^{(n)}_{\tau, i} = p^{(n)}_{\tau, n + 1 - i}$ holds for each $i = 1, \ldots, n$.   
If $\tau\in I_\mathsf{F}$ is complex, we define 
$p^{(n)}_{\rho\tau, i}$ and $q^{(n)}_{\rho\tau, i}$ to be $p^{(n)}_{\rho\tau, i} = q^{(n)}_{\tau,  n+1-i}$ and $q^{(n)}_{\rho\tau, i} = p^{(n)}_{\tau,  n+1-i}$ respectively, where $\rho$ is the complex conjugate.

\begin{rem}\label{rem:Hodgepq}
We briefly recall the relation between the expected Hodge numbers defined as above and the Langlands parameters, according to \cite[Section 4.3.3]{clo90}.   
Write the Langlands parameter $\phi(\pi_v)$ of $\pi_v$ as in Theorem \ref{thm:LLC}.   
Let $\tau \colon \mathsf{F}\to \mathbf{C}$  be the embedding corresponding to $v\in \Sigma_{\mathsf{F}, \infty} \subset I_\mathsf{F}$.  
\begin{itemize}
\item
{\bf ($v$: real, $n$: even)}
Suppose that $v$ is real and $n=2m$ is even.    
For each $m+1\leq i \leq n=2m$, define  $l^{(n)}_{v, i}$ so that the equation  $l^{(n)}_{v, i} + l^{(n)}_{v, n+1-i}  = 0$ is fulfilled for $i=1,2,\dotsc,m$.                             
Then we have
\begin{align*}
    \nu^{(n)}_{v, i}   &=    \frac{w(\pi^{(n)})}{2} \quad \text{for }i=1,2,\dotsc,m, \\
   p^{(n)}_{\tau, i} &= \dfrac{w(\pi^{(n)})}{2} - \frac{l^{(n)}_{v, i}}{2} - \frac{n-1}{2} =  \frac{w(\pi^{(n)}) - l^{(n)}_{v, i} - n+1}{2} \quad \text{for }i=1,2,\dotsc, n,    \\ 
   q^{(n)}_{\tau, i}  &=  \dfrac{w(\pi^{(n)})}{2} + \frac{l^{(n)}_{v, i}}{2} - \frac{n-1}{2}  =  \frac{w(\pi^{(n)}) + l^{(n)}_{v, i} - n+1}{2} \quad \text{for }i=1,2,\dotsc,n.
\end{align*}

\item
{\bf ($v$: real, $n$: odd)}
 Suppose that $v$ is real and $n=2m+1$ is odd.    
For each $m+2\leq i \leq n=2m+1$, define $l^{(n)}_{v, i}$ so that the equation $l^{(n)}_{v,i} + l^{(n)}_{v, n+1-i} = 0$ is fulfilled for $i=1,2,\dotsc,m$.
We also set $l^{(n)}_{v, m+1} = 0$.                               
Then we have
\begin{align*}
  \nu^{(n)}_v &=      \frac{w(\pi^{(n)})}{2} \quad \text{for }i=0,1,\dotsc,m, \\ 
   p^{(n)}_{\tau, i} &=\dfrac{w(\pi^{(n)})}{2} - \frac{l^{(n)}_{v, i}}{2} - \frac{n-1}{2} =  \frac{w(\pi^{(n)}) - l^{(n)}_{v, i} - n+1}{2} \quad \text{for }i=1,2,\dotsc,n,    \\ 
   q^{(n)}_{\tau, i}  &= \dfrac{w(\pi^{(n)})}{2} + \frac{l^{(n)}_{v, i}}{2} - \frac{n-1}{2}  =  \frac{w(\pi^{(n)}) + l^{(n)}_{v, i} - n+1}{2} \quad \text{for }i=1,2,\dotsc,n.
\end{align*}

\item
{\bf ($v$: complex)}
Suppose that $v$ is complex. Then we have
\begin{align*}
   p^{(n)}_{\tau, i} &= a^{(n)}_{v, i}  - \frac{n-1}{2} =q^{(n)}_{\rho\tau,n+1-i},   
   &q^{(n)}_{\tau, i}  &=  b^{(n)}_{v, i}  - \frac{n-1}{2}=p^{(n)}_{\rho\tau,n+1-i}.
\end{align*}
\end{itemize}
\end{rem}

Finally we impose the following constraints on the Langlands parameters of the cohomological irreducible cuspidal automorphic representations under consideration:
\begin{align} \label{eq:auto_assump}
 \begin{minipage}{\textwidth}
  \begin{itemize}
 \item[--] If $n$ is odd, the signatures $\delta^{(n)}_v$ for real places $v\in \Sigma_{\mathsf{F},\mathbf{R}}$ coincide with one another.
 \item[--] If the motivic weight $\mathsf{w}^{(n)}=w(\pi^{(n)})-(n-1)$ is even, the equality $a^{(n)}_{v,i}=b^{(n)}_{v,i}$ does not hold for any complex place $v\in \Sigma_{\mathsf{F},\mathbf{C}}$ and any $i$.
\end{itemize}
\end{minipage}
\end{align}

\subsection{Finite dimensional representations of $\mathrm{GL}_n(\mathsf{F}_v)$}\label{sec:FinDim}

Let $n\in {\mathbf Z}$ be a positive integer, $\mathsf{F}$ a number field 
and $\boldsymbol{\mu}=(\boldsymbol{\mu}_\tau)_{\tau\in I_\mathsf{F}}$ an element of $({\mathbf Z}^n)^{I_\mathsf{F}}$ such that, for each $\boldsymbol{\mu}_\tau = (\mu_{\tau, 1}, \ldots, \mu_{\tau, n})$, the {\em dominancy condition}  $\mu_{\tau, 1} \geq \cdots \geq \mu_{\tau, n}$ hold. 
For $\tau\in I_\mathsf{F}$ and its corresponding place $v_\tau\in \Sigma_{\mathsf{F},\infty}$, let $V(\boldsymbol{\mu}_\tau)$ be the finite dimensional complex representation of ${\rm GL}_n(\mathsf{F}_{v_\tau})$ of highest weight $\boldsymbol{\mu}_\tau$.     
Write $V(\boldsymbol{\mu}_\tau)^\vee$ the contragredient representation of $V(\boldsymbol{\mu}_\tau)$, whose highest weight $\boldsymbol{\mu}^\vee_\tau$ is given by $(-\mu_{\tau, n}, \ldots, -\mu_{\tau, 1})$.
Suppose that $\boldsymbol{\mu}$ satisfies the {\em purity condition},
that is, there exists an integer $w(\boldsymbol{\mu})\in {\mathbf Z}$ such that, if $\mathbf{1}$ denote the vector $(1,1,\dotsc,1)\in \mathbf{Z}^n$, 
the equality $\boldsymbol{\mu}_{\tau} = w(\boldsymbol{\mu})\mathbf{1} + \boldsymbol{\mu}^\vee_\tau$ holds for real $\tau\in I_\mathsf{F}$, 
and the equality $\boldsymbol{\mu}_{\rho \tau} = w(\boldsymbol{\mu}) \mathbf{1} + \boldsymbol{\mu}^\vee_\tau$ holds for complex $\tau\in I_\mathsf{F}$.  
We call $w(\boldsymbol{\mu})$ the {\em purity weight} of $\boldsymbol{\mu}$.
 
For each complex representation $(\varrho, V)$ of  ${\rm GL}_n({\mathbf C})$, 
we define a complex representaion $(\varrho^{\rm conj}, V^{\rm conj})$ of ${\rm GL}_n({\mathbf C})$ so that the base space $V^{\rm conj}$ equals $V$ and $g\in \mathrm{GL}_n(\mathbf{C})$ acts on it via $\varrho^{\rm}(g):=\varrho(g^\rho)$.
For each real embedding $\tau\in I_\mathsf{F}$, we define ${\mathcal V}(\boldsymbol{\mu}_\tau)$ to be $V(\boldsymbol{\mu}_\tau)$. 
For each complex embedding $\tau\in I_\mathsf{F}$ corresponds to a complex place $v\in \Sigma_{\mathsf{F},\infty}$, 
   we define ${\mathcal V}(\boldsymbol{\mu}_v)$ to be $V(\boldsymbol{\mu}_\tau) \otimes_{\mathbf C} V(\boldsymbol{\mu}_{\rho\tau})^{\rm conj}$. 
Then $ {\mathcal V} ( \boldsymbol{\mu} ) = \bigotimes_{v \in \Sigma_{\mathsf{F}, \infty}} {\mathcal V} ( \boldsymbol{\mu}_v )$ gives a complex representation  of ${\rm GL}_n(\mathsf{F}_{\mathbf A, \infty}) := \prod_{v \in \Sigma_{\mathsf{F}, \infty}}{\rm GL}_n(\mathsf{F}_v)$.
Let $\varrho_{\boldsymbol{\mu}}$ denote the action of ${\rm GL}_n(\mathsf{F}_{\mathbf A, \infty})$ on ${\mathcal V}(\boldsymbol{\mu})$.

\begin{dfn} \label{dfn:intl_weights}
Let $\boldsymbol{\mu}^{(n)}=(\boldsymbol{\mu}^{(n)}_\tau)_{\tau\in I_\mathsf{F}}$ and $\boldsymbol{\mu}^{(n+1)}=(\boldsymbol{\mu}^{(n+1)}_\tau)_{\tau\in I_\mathsf{F}}$ be elements of $({\mathbf Z}^n)^{I_\mathsf{F}}$ and $({\mathbf Z}^{n+1})^{I_\mathsf{F}}$ respectively, both satisfying the dominancy condition and the purity condition for every $\tau\in I_{\mathsf{F}}$.

\begin{enumerate}[label=(\arabic*)]
 \item We say that $\boldsymbol{\mu}^{(n)}$ and $\boldsymbol{\mu}^{(n+1)}$ satisfy the {\em interlace condition} if the following condition is fulfilled:
\begin{itemize}
 \item[--] the inequalities   $\mu^{(n+1)}_{\tau, n+1} 
               \leq \mu^{(n)}_{\tau, n} \leq \mu^{(n+1)}_{\tau, n}
                \leq \cdots \leq \mu^{(n+1)}_{\tau, 2} \leq \mu^{(n)}_{\tau, 1} \leq \mu^{(n+1)}_{\tau, 1}$ hold for every $\tau \in I_\mathsf{F}$.
\end{itemize}
If the above condition is satisfied, we write  $\boldsymbol{\mu}^{(n+1)} \succcurlyeq \boldsymbol{\mu}^{(n)}$. 
\item We say that $\boldsymbol{\mu}^{(n)}$ and $\boldsymbol{\mu}^{(n+1)}$ satisfy the {\em strong interlace condition} if the following condition is fulfilled:
\begin{itemize}
 \item[--] for each $\tau \in I_\mathsf{F}$, the inequalities   $\mu^{(n+1)}_{\tilde{\tau}, n+1} 
               \leq \mu^{(n)}_{\tilde{\tau}', n} \leq \mu^{(n+1)}_{\tilde{\tau}, n}
                \leq \cdots \leq \mu^{(n+1)}_{\tilde{\tau}, 2} \leq \mu^{(n)}_{\tilde{\tau}', 1} \leq \mu^{(n+1)}_{\tilde{\tau}, 1}$ hold with arbitrary $\tilde{\tau}, \tilde{\tau}'\in \{\tau,\rho\tau\}$.
\end{itemize}
If the above condition is satisfied, we write  $\boldsymbol{\mu}^{(n+1)} \succapprox \boldsymbol{\mu}^{(n)}$. 
\end{enumerate}
\end{dfn}

For each $\alpha\in {\rm Aut}({\mathbf C})$ and $\boldsymbol{\mu}\in ({\mathbf Z}^n)^{I_\mathsf{F}}$, 
    we define ${}^{\alpha}\boldsymbol{\mu}  \in ({\mathbf Z}^n)^{I_\mathsf{F}}$ so that ${}^{\alpha}\boldsymbol{\mu}_\tau = \boldsymbol{\mu}_{\alpha^{-1}\tau}$ 
holds for each $\tau \in I_\mathsf{F}$.
Let ${\mathbf Q}(\boldsymbol{\mu})$ be the maximal subfield  of ${\mathbf C}$ fixed  by elements of $\{ \alpha \in {\rm Aut}({\mathbf C}) \mid  {}^{\alpha}\boldsymbol{\mu} =  \boldsymbol{\mu}  \}$, 
 and $E$ a finite extension of ${\mathbf Q}(\boldsymbol{\mu})$.    
Then the complex representation ${\mathcal V}(\boldsymbol{\mu})$ is defined over $E$, that is,  there exists an $E$-vector space ${\mathcal V}(\boldsymbol{\mu})_E$ satisfying
${\mathcal V}(\boldsymbol{\mu})_E \otimes_{E} {\mathbf C} \cong {\mathcal V}(\boldsymbol{\mu})$
on which ${\rm GL}_n(\mathsf{F}_{\mathbf{A},\infty})$ acts via  $\varrho_{\boldsymbol{\mu}}$ (see \cite[page 122]{clo90} and  \cite[Section 2.1.5]{rag16}).

\subsection{$({\mathfrak g}, K)$-cohomology groups}

Let $\pi^{(n)} = \bigotimes_{v\in \Sigma_{\mathsf{F}}}^\prime \pi^{(n)}_v$ be a cohomological irreducible 
cuspidal automorphic representation of ${\rm GL}_n(\mathsf{F}_{\mathbf A})$ satisfying (\ref{eq:auto_assump}). 
For each $v\in \Sigma_{\mathsf{F}, \infty}$,   
write the Langlands parameter $\phi(\pi^{(n)}_v)$ of $\pi^{(n)}_v$ as in Theorem \ref{thm:LLC},  and write its restriction $\phi(\pi^{(n)}_v)|_{ {\mathbf C}^\times }$ to ${\mathbf C}^\times$  as in (\ref{eq:LangPara}).   
Fix $\Sigma_{\mathsf{F}, \infty} \subset I_\mathsf{F}$ and we write the embedding corresponding to an archimedean place $v\in \Sigma_{\mathsf{F},\infty}$ of $\mathsf{F}$ as $\tau=\tau_v \in I_\mathsf{F}$.  
Recall the definidion of $p^{(n)}_{\tau,i}$ and $q^{(n)}_{\tau,i}$ for each $\tau \in I_\mathsf{F}$ introduced in Section~\ref{sec:LangPara}, and set  
\begin{align*}
  \boldsymbol{p}^{(n)}_\tau 
      &= \left( p^{(n)}_{\tau, 1}, p^{(n)}_{\tau, 2},\dotsc, p^{(n)}_{\tau, n} \right),  &
  \boldsymbol{q}^{(n)}_\tau 
      &= \left( q^{(n)}_{\tau, 1}, q^{(n)}_{\tau, 2}, \dotsc, q^{(n)}_{\tau, n} \right). 
\end{align*}

Define $\boldsymbol{\rho}^{(n)}$ and $\boldsymbol{\mu}_{\pi,\tau}^{(n)}:=(\mu^{(n)}_{\pi,\tau, 1}, \dotsc, \mu^{(n)}_{\pi,\tau, n}) \in {\mathbf Z}^n$ as follows:
 \begin{align*}
  \boldsymbol{\rho}^{(n)}
  &=  \left(  \frac{n-1}{2}, \frac{n-3}{2} , \dotsc,  \dfrac{n-2i+1}{2}, \dotsc, -\frac{n-1}{2}  \right), \\
   \boldsymbol{\mu}^{(n)}_{\pi,\tau}
  &=  \boldsymbol{q}^{(n)}_\tau + \frac{n-1}{2}- \boldsymbol{\rho}^{(n)}
  =  \left(  q^{(n)}_{\tau, 1} , q^{(n)}_{\tau,2}+1, \dotsc, q^{(n)}_{\tau, i}  + i-1, \dotsc, q^{(n)}_{\tau, n} + n-1   \right).  
 \end{align*}      
We call $\boldsymbol{\mu}_\pi^{(n)}:=(\boldsymbol{\mu}^{(n)}_{\pi,\tau})_{\tau \in I_\mathsf{F}}$ the highest weight attached to $\pi^{(n)}$.
We can easily verify that $\boldsymbol{\mu}^{(n)}_\pi$ satisfies, for every $\tau\in I_{\mathsf{F}}$, both of the dominancy condition $\mu^{(n)}_{\pi,\tau, 1} \geq  \mu^{(n)}_{\pi,\tau, 2} \geq \cdots \geq \mu^{(n)}_{\pi,\tau, n}$ and the purity condition 
\begin{align*}
\begin{cases}
     \boldsymbol{\mu}^{(n)}_{\pi,\tau} = w(\pi^{(n)})\mathbf{1} + \boldsymbol{\mu}^{(n),\vee}_{\pi,\tau}  & \text{if $\tau$ is real},   \\
     \boldsymbol{\mu}^{(n)}_{\pi, \rho\tau} = w(\pi^{(n)})\mathbf{1} + \boldsymbol{\mu}^{(n),\vee}_{\pi,\tau}       & \text{if $\tau$ is complex}   
\end{cases}
\end{align*}
 by construction (note that the purity weight $w(\boldsymbol{\mu}^{(n)}_\pi)$ of $\boldsymbol{\mu}^{(n)}_\pi$ coincides with $w(\pi^{(n)})$).   
Define the complex representation ${\mathcal V} ( \boldsymbol{\mu}_\pi^{(n)} ) = \bigotimes_{v \in \Sigma_{\mathsf{F}, \infty}} {\mathcal V} ( \boldsymbol{\mu}^{(n)}_{\pi,v} )$ as in Section \ref{sec:FinDim}. 

\begin{prop} \label{prop:gK}
For each $v\in \Sigma_{\mathsf{F},\infty}$, define ${\rm b}_{n,v}$ and ${\rm t}_{n,v}$ as 
\begin{align*}
  {\rm b}_{n, v} 
         &= \begin{cases}  \lfloor\frac{n^2}{4}\rfloor  & \text{if $v$ is real},    \\
                                     \frac{n(n-1)}{2} &   \text{if $v$ is complex},   \end{cases}  &
  {\rm t}_{n,v} 
         &= {\rm b}_{n, v}
             +  \begin{cases}   \lfloor\frac{n-1}{2}\rfloor  & \text{if $v$ is real},    \\
                                         n-1 &   \text{if $v$ is complex}.   \end{cases}
\end{align*}
Then, for each integer $i_v$ with ${\rm b}_{n, v} \leq i_v \leq {\rm t}_{n,v}$, we have     
\begin{align*}
    H^{i_v} ( \mathfrak{g}_{n, v}, K^{(n)}_{v} ; \pi^{(n)}_v \otimes_\mathbf{C} {\mathcal V} ( \boldsymbol{\mu}^{(n), \vee}_{\pi,v})  )  \neq 0.  
\end{align*}
In particular, we have
\begin{align*}
   H^{{\rm b}_{n,v}}( \mathfrak{g}_{n, v}, K^{(n)}_{v} ; \pi^{(n)}_v \otimes_\mathbf{C} {\mathcal V} ( \boldsymbol{\mu}^{(n), \vee}_{\pi,v}) ) 
      \cong \begin{cases}  
        {\mathbf C}^{\oplus 2}  &  \text{if $v$ is real and $n$  is even},  \\
        {\mathbf C}  &  \text{if $v$ is real and $n$ is odd},   \\
        {\mathbf C}  &  \text{if $v$ is complex}.
         \end{cases}
\end{align*}
\end{prop}

\begin{proof}
The statement is firstly proved in \cite[Lemma 3.14]{clo90}.  
We here adopt formalisms in \cite[Section 3.1.2]{mah05} and \cite[Section 5.5]{gr14}.
\end{proof}

We define
${\rm b}_n = \sum_{ v \in \Sigma_{\mathsf{F}, \infty}} {\rm b}_{n, v}$ and call ${\rm b}_n$ the {\em bottom degree} of $\pi^{(n)}$.

\subsection{Motives attached to automorphic representations}\label{sec:CloConj}

In this subsection, we briefly recall the conjectural pure motives attached to automorphic representations of $\mathrm{GL}_n$, after Clozel \cite[Conjecture 4.5]{clo90}. Let $\mathsf{F}$ be a number field and $\pi^{(n)}$ a cohomological irreducible cuspidal automorphic representation of $\mathrm{GL}_n(\mathsf{F}_{\mathbf{A}})$. We first recall the notion of the field of rationality $\mathbf{Q}(\pi^{(n)})$ of $\pi^{(n)}$, according to \cite[Section~3]{clo90}; see also \cite[Section~3.1]{rs08}. Let $V^{(n)}_{\mathrm{fin}}$ denote the representation space of $\pi^{(n)}_{\rm fin}=\bigotimes^\prime_{v\in \Sigma_{\mathsf{F},\mathrm{fin}}}\pi^{(n)}_v$, and for each $\alpha \in {\rm Aut}({\mathbf C})$, consider the $\alpha$-twist ${}^\alpha V^{(n)}_{\mathrm{fin}}:=V^{(n)}_{\mathrm{fin}} \otimes_{{\mathbf C}, \alpha^{-1}} {\mathbf C}$ of $V^{(n)}_{\mathrm{fin}}$.     
Then there exists a cohomological cuspidal automorphic representation ${}^\alpha \pi^{(n)}$ of ${\rm GL}_n({\mathsf F}_{\mathbf A})$ such that ${}^\alpha V^{(n)}_{\mathrm{fin}}$ is the representation space of the finite part ${}^{\alpha}\pi^{(n)}_{\mathrm{fin}}$ of ${}^\alpha \pi^{(n)}$. Note that, if $\boldsymbol{\mu}^{(n)}_\pi=(\boldsymbol{\mu}^{(n)}_{\pi,\tau})_{\tau\in I_\mathsf{F}}$ denotes the highest weight attached to $\pi^{(n)}$, the highest weight attached to ${}^\alpha \pi^{(n)}$ is described as  $\boldsymbol{\mu}^{(n)}_{{}^\alpha \pi}=(\boldsymbol{\mu}^{(n)}_{\pi,\alpha^{-1}\tau})_{\tau\in I_{\mathsf{F}}}$ (see \cite[Th\'eor\`eme~3.13]{clo90}). 
Now define $\mathbf{Q}(\pi^{(n)})$ as the maximal subfield of $\mathbf{C}$ fixed by $\{  \sigma \in {\rm Aut}({\mathbf C})  \mid  {}^\sigma \pi^{(n)} = \pi^{(n)}  \}$, which we call the {\em field of rationality} of $\pi^{(n)}$. The field of rationality $\mathbf{Q}(\pi^{(n)})$ is known to be a number field.

To each cohomological irreducible cuspidal automorphic representation $\pi^{(n)}$ of $\mathrm{GL}_n(\mathsf{F}_{\mathbf{A}})$, let us take a sufficiently large  finite extension $E$ of $\mathbf{Q}(\pi^{(n)})$ contained in $\mathbf{C}$. Under these settings, Clozel conjectures that there exists a pure motive ${\mathcal M}[\pi^{(n)}]$ defined over $\mathsf{F}$ with coefficients in $E$, which is of weight $w-n+1$. 
The conjectural motive ${\mathcal M}[\pi^{(n)}]$ has the following properties:
\begin{itemize}
\item[--]  For each $\tau \in  I_\mathsf{F}$,  we have the following Hodge decomposition: 
          \begin{align*}
             H_{\rm B} ({\mathcal M}[\pi]_\tau) \otimes_{E,\iota_0} {\mathbf C}   =   \bigoplus^n_{i=1} H^{p^{(n)}_{\tau, i}, q^{(n)}_{\tau, i}  } ({\mathcal M}[\pi]_\tau), \quad \dim_{\mathbf{C}} H^{p^{(n)}_{\tau, i}, q^{(n)}_{\tau, i}  } ({\mathcal M}[\pi]_\tau)=1.
          \end{align*} 
\item[--]  We have the following equality on $L$-functions:
         \begin{align*}
            L\left(s - \frac{n-1}{2}, \pi^{(n)}\right)  = L^{\iota_0}(s, {\mathcal M}[\pi^{(n)}]).   
         \end{align*}
	   Here the left-hand side is the automorphic $L$-function of $\pi^{(n)}$ introduced by Godement and Jacquet, and the right-hand side is the Hasse--Weil $L$-function of the pure motive $\mathcal{M}[\pi^{(n)}]$ with respect to the canonical embedding $\iota_0\colon E\hookrightarrow \mathbf{C}$.
\end{itemize}

\section{Whittaker periods and Rankin--Selberg $L$-function}\label{sec:RSMot}

Let $\pi^{(n)}$ and $\pi^{(n+1)}$ be  cohomological irreducible cuspidal automorphic representations 
   of ${\rm GL}_n(\mathsf{F}_{\mathbf A})$ and ${\rm GL}_{n+1}(\mathsf{F}_{\mathbf A})$, respectively, which satisfy (\ref{eq:auto_assump}). 
In this section, we verify algebraicity of the critical values of the Rankin--Selberg $L$-function 
$L(s, \pi^{(n+1)} \times \pi^{(n)} )$.    
After introducing basic setups in Section \ref{sec:intlPi}, 
we summarize the construction of the cup product pairing on the symmetric space
$Y^{(n)}_{\mathcal K} := {\rm GL}_n(\mathsf{F}) \backslash  {\rm GL}_n(  \mathsf{F}_{\mathbf A} )  / K^{(n)}_\infty {\mathcal K}$
in Section \ref{sec:Cup}.
Here we emphasize precise computations of signatures, 
   since they are  important ingredients of the proof of our main theorem (Theorem \ref{thm:main}).   
We then give a cohomological interpretation of
the critical values of the Rankin--Selberg $L$-function in terms of the cup product of certain cohomology classes in Section \ref{sec:RSLfn}. 
In the final subsection (Section \ref{sec:Whitt}), we recall the definition of Whittaker periods according to \cite{rs08}, and discuss algebraicity of the critical values of the Rankin--Selberg $L$-functions following \cite[Theorem 1.1]{rag16} (see Theorem \ref{thm:AlgRS}).

\subsection{The interlace condition for automorphic representations}\label{sec:intlPi}

We briefly recall basic facts on  critical values of the Rankin--Selberg $L$-function $L(s,\pi^{(n+1)}\times \pi^{(n)})$.  
In the definition below, we set $\mathbf{1}=(1)_{\tau\in I_\mathsf{F}}$ and use the notation introduced in Definition~\ref{dfn:intl_weights}. 

\begin{dfn}\label{dfn:intlauto}
Let $\pi^{(n+1)}$ and $\pi^{(n)}$ be as above, and let $\boldsymbol{\mu}_\pi^{(n+1)} = ( \boldsymbol{\mu}^{(n+1)}_{\pi,\tau})_{\tau\in I_\mathsf{F}}$ and $\boldsymbol{\mu}^{(n)}_\pi = ( \boldsymbol{\mu}^{(n)}_{\pi,\tau})_{\tau\in I_\mathsf{F}}$ 
be the highest weights attached to $\pi^{(n+1)}$ and $\pi^{(n)}$, respectively.
\begin{enumerate}[label=(\arabic*)]
 \item We say that $\pi^{(n+1)}$ and $\pi^{(n)}$ satisfy the {\em interlace condition}
if there exists an integer $m_0\in \mathbf{Z}$ satisfying  $\boldsymbol{\mu}^{(n+1),\vee}_\pi  \succcurlyeq \boldsymbol{\mu}^{(n)}_\pi+m_0\mathbf{1}$.

 \item We say that $\pi^{(n+1)}$ and $\pi^{(n)}$ satisfy the {\em strong interlace condition}
if there exists an integer $m_0\in \mathbf{Z}$ satisfying  $\boldsymbol{\mu}^{(n+1),\vee}_\pi  \succapprox \boldsymbol{\mu}^{(n)}_\pi+m_0\mathbf{1}$.
\end{enumerate}
\end{dfn}

\begin{rem}
\begin{enumerate}
\item  The interlace condition for $\pi^{(n+1)}$ and $\pi^{(n)}$ defined in Definition \ref{dfn:intlauto} is equivalent to one of the following conditions (see Section~\ref{sec:LangPara} for the definition of the symbols appearing here):
         \begin{itemize}
             \item[--]  there exists an integer $A$ such that the inequalities  
\begin{align*}
\hspace*{6em} -a^{(n+1)}_{v, 1} > a^{(n)}_{v, n}+A >\cdots >-a^{(n+1)}_{v,i}>a^{(n)}_{v,n+1-i}+A>-a^{(n+1)}_{v,i+1} >  \cdots > a^{(n)}_{v, 1}+A > -a^{(n+1)}_{v, n+1}
\end{align*}
hold for every $\Sigma_{\mathsf{F},\infty}$; 
             \item[--] there exists an integer $B$ such that the inequalities  
\begin{align*}
\hspace*{6em} -b^{(n+1)}_{v, n+1} > b^{(n)}_{v, 1}+B >\cdots >-b^{(n+1)}_{v,n+2-i}>b^{(n)}_{v,i}+B>-b^{(n+1)}_{v,n+1-i} >  \cdots > b^{(n)}_{v, n}+B > -b^{(n+1)}_{v, 1}
\end{align*}
hold for every $\Sigma_{\mathsf{F},\infty}$;       
             \item[--] the condition $\boldsymbol{p}^{(n+1)}_\tau \succ \boldsymbol{p}^{(n)}_\tau$ holds for every $\tau\in I_{\mathsf{F}}$, and the integer $Q=n+m_0$ in (\ref{eq:intlpq}) is independent of $\tau\in I_{\mathsf{F}}$ and $\sigma\in I_E$;
             \item[--] the condition $\boldsymbol{q}^{(n+1)}_\tau \succ \boldsymbol{q}^{(n)}_\tau$ holds for every $\tau\in I_{\mathsf{F}}$, and the integer $Q=n+m_0$ in (\ref{eq:intlpq}) is independent of $\tau\in I_{\mathsf{F}}$ and $\sigma\in I_E$.
         \end{itemize} 
\item Suppose that $v\in \Sigma_{\mathsf{F}, \infty}$ is real.   
         Then the interlace condition is equivalent to the validity of the inequalities
         \begin{align*}
              l^{(n+1)}_{v, 1} > l^{(n)}_{v, 1} > l^{(n+1)}_{v, 2} > l^{(n)}_{v, 2} >  \cdots > l^{(n+1)}_{v, n} > l^{(n)}_{v, n} > l^{(n+1)}_{v, n+1},       
         \end{align*}
where $l^{(n+1)}_{v,i}$ and $l^{(n)}_{v,i}$ are integers appearing in the Langlands parameter of $\pi^{(n+1)}_v$ and $\pi^{(n)}_v$, respectively; see Theorem~\ref{thm:LLC}.
\end{enumerate}
\end{rem}


Let $L(s,\pi^{(n+1)}\times \pi^{(n)})$ denote the Rankin--Selberg $L$-function of $\pi^{(n+1)}$ and $\pi^{(n)}$ introduced by Jacquet, Piatetski-Shapiro and Shalika. We say $m\in \mathbf{Z}$ to be a {\em critical point} of $L(s,\pi^{(n+1)}\times \pi^{(n)})$ if neither $L(s,\pi^{(n+1)}_v \times \pi^{(n)}_v)$ nor $L(1-s,(\pi^{(n+1)}_v \times \pi^{(n)}_v)^\vee)$ has a pole at $s=m+\dfrac{1}{2}$ for each $v\in \Sigma_{\mathsf{F},\infty}$.  
Let $\mathrm{Crit}(\pi^{(n+1)}, \pi^{(n)})$ denote the set of all the critical points of $L(s,\pi^{(n+1)}\times \pi^{(n)})$;
\begin{align*}
 {\rm Crit}( \pi^{(n+1)},  \pi^{(n)}  )    
     := \left\{  m\in {\mathbf Z}  \, \left| \, \text{$m$ is a critical point of $L(s,\pi^{(n+1)}\times \pi^{(n)})$} \right. \right\}.
\end{align*}

\begin{lem} \label{lem:crit}
Let $\pi^{(n+1)}$ and $\pi^{(n)}$ be as above and $m$ an integer.  Then the following statements are equivalent$:$
\begin{enumerate}[label={\rm (\roman*)}]
\item the integer $m$ is a critical point of $L(s,\pi^{(n+1)}\times \pi^{(n)});$ that is, $m\in \mathrm{Crit}(\pi^{(n+1)},\pi^{(n)})$ holds$;$
\item the highest weights $\boldsymbol{\mu}^{(n+1)}_\pi$ and $\boldsymbol{\mu}^{(n)}_\pi$ attached to $\pi^{(n+1)}$ and $\pi^{(n)}$, respectively, satisfy the condition $\boldsymbol{\mu}^{(n+1),\vee}_\pi\succcurlyeq \boldsymbol{\mu}^{(n)}_\pi+m\mathbf{1}$$;$ 
\item there exists a $\mathrm{GL}_n(\mathsf{F}_{\mathbf{A}, \infty})$-equivariant morphism $\nabla^m \colon \mathcal{V}(\boldsymbol{\mu}^{(n+1),\vee}_\pi)\rightarrow \mathcal{V}(\boldsymbol{\mu}^{(n)}_\pi)\otimes \det^m$ derived from the branching rule for the pair $(\mathrm{GL}_{n+1}, \mathrm{GL}_n)$.
\end{enumerate}
\end{lem}

\begin{proof}
See \cite[Section 2.4.5]{rag16}.
\end{proof}

\subsection{The cup product pairings}\label{sec:Cup}

In this subsection, we briefly summarize the construction of the cup product pairing, which gives a cohomological interpretation of the critical values of $L(s,\pi^{(n+1)}\times \pi^{(n)})$. 
Although such a description can be found in \cite[Section 2.5]{rag16}, 
we include explanation of details of the construction since precise calculation of signatures plays a important role in the proof of our main theorem (Theorem~\ref{thm:main}). Refer also to \cite[Section~6.1]{hn}.

For each compact open subgroup ${\mathcal K}_n \subset {\rm GL}_n( \mathsf{F}_{{\mathbf A}, {\rm fin}} )$, 
set 
\begin{align*}
   Y^{(n)}_{\mathcal K_n} = {\rm GL}_n(\mathsf{F}) \backslash  {\rm GL}_n(  \mathsf{F}_{\mathbf A} )  / K^{(n)}_\infty {\mathcal K_n}, 
   \quad     {\mathcal Y}^{(n)}_{\mathcal K_n} = {\rm GL}_n(\mathsf{F}) \backslash  {\rm GL}_n(  \mathsf{F}_{\mathbf A} )  / C^{(n)}_\infty {\mathcal K_n}.   
\end{align*}
Refer to Section~\ref{sec:grinfty} for the definition of $K^{(n)}_\infty$ and $C^{(n)}_\infty$.
Each connected component of $\mathcal{Y}^{(n)}_{\mathcal{K}_n}$ is labeled by an element of ${\rm Cl}^+_\mathsf{F}({\mathcal K_n} ) := \mathsf{F}^\times \backslash \mathsf{F}^\times_{\mathbf A}  /  \mathsf{F}^\times_{\infty, +} \det {\mathcal K_n}$. The component ${\mathcal Y}^{(n)}_{{\mathcal K}_n, x}$ labeled by $x\in \mathrm{Cl}^+_\mathsf{F}(\mathcal{K}_n)$ is then described as $\Gamma_x \backslash  {\rm GL}_n(\mathsf{F}_{\mathbf{A},\infty})^\circ / C^{(n)}_\infty$ for a certain arithmetic subgroup $\Gamma_x$ of ${\rm GL}_n({\mathsf F})$.  
Therefore, once we fix an orientation on ${\rm GL}_n(\mathsf{F}_{\mathbf{A},\infty})^\circ / C^{(n)}_\infty$, 
it also induces an orientation on each connected component of ${\mathcal Y}^{(n)}_{ {\mathcal K}_n }$, 
and the fundamental class 
$[ {\mathcal Y}^{(n)}_{ {\mathcal K}_n, x } ]$ of the component $\overline{\mathcal Y}^{(n)}_{{\mathcal K}_n, x}$ is defined as an element of the relative homology group $H_{{\rm d}_n}(  \overline{\mathcal Y}^{(n)}_{{\mathcal K}_n, x}, \partial \overline{\mathcal Y}^{(n)}_{{\mathcal K}_n, x} ; {\mathbf Z} )$; here  ${\rm d}_n$ is the dimension of ${\mathcal Y}^{(n)}_{{\mathcal K}_n}$  
and $ \overline{\mathcal Y}^{(n)}_{{\mathcal K}_n, x}$ denotes  the Borel--Serre compactification of ${\mathcal Y}^{(n)}_{{\mathcal K}_n, x}$.     
Set 
\begin{align*}
  [ {\mathcal Y}^{(n)}_{ {\mathcal K}_n } ] 
  = \sum_{x \in \mathrm{Cl}^+_{\mathsf{F}}(\mathcal{K}_n)}   
     [ {\mathcal Y}^{(n)}_{ {\mathcal K}_n, x } ]       
     \in H_{{\rm d}_n}(  \overline{\mathcal Y}^{(n)}_{{\mathcal K}_n}, \partial \overline{\mathcal Y}^{(n)}_{{\mathcal K}_n} ; {\mathbf Z} ).  
\end{align*}
Hereafter we always fix an orientation on ${\mathcal Y}^{(n)}_{{\mathcal K}_n}$
 so that $[\mathcal{Y}^{(n)}_{\mathcal{K}_n,x}]$ corresponds to the volume form $\omega_{n}$ on ${\mathcal Y}^{(n)}_{{\mathcal K}_n,x}$ introduced in Section \ref{sec:lie} via the Poincar\'e duality.

Since, for each $v\in \Sigma_{\mathsf{F},\mathbf{R}}$, the element $\epsilon_{n, v} := {\rm diag}(-1, 1, \ldots, 1)\in {\rm GL}_n( \mathsf{F}_v )$ normalises $C^{(n)}_v={\rm SO}(n)$, the right translation $R_{\epsilon_{n, v}}$ by $\epsilon_{n, v}$ (regarded as an element of $\mathrm{GL}_n(\mathsf{F}_\mathbf{A})$) induces a continuous involution on ${\mathcal Y}^{(n)}_{{\mathcal K}_n}$. For $v\in \Sigma_{\mathsf{F},\mathbf{R}}$, let  $(-1)_v$ denote the element of $\mathsf{F}^\times_{\mathbf{A}}$ which is trivial at all places other than $v$, and equals $-1$ at $v$.

\begin{lem}\label{lem:signfund}
For a real place $v\in \Sigma_{\mathsf{F},\mathbf{R}}$, let $\epsilon_{n, v, \ast}$ denote the involution on $H_{{\rm d}_n}(  \overline{\mathcal Y}^{(n)}_{{\mathcal K}_n}, \partial \overline{\mathcal Y}^{(n)}_{{\mathcal K}_n} ; {\mathbf Z} )$ induced by $R_{\epsilon_n, v}$. 
 Then we have 
\begin{align*}
   \epsilon_{n, v, \ast} [ {\mathcal Y}^{(n)}_{ {\mathcal K}_n, x } ]  
   = (-1)^{n+1}  [ {\mathcal Y}^{(n)}_{ {\mathcal K}_n, (-1)_v x }].  
\end{align*}
\end{lem}

\begin{proof}
See \cite[Lemma 2.43]{rag16} (also refer to \cite[Section 5.1.3, Lemma]{mah05}). 
\end{proof}

Let ${\rm p}_n \colon {\mathcal Y}^{(n)}_{{\mathcal K}_n} \to Y^{(n)}_{{\mathcal K}_n}$ be the natural projection 
and $\iota : {\mathcal Y}^{(n)}_{{\mathcal K}_n}  \to {\mathcal Y}^{(n+1)}_{{\mathcal K}_{n+1}}$ a natural map induced by 
the inclusion $\iota \colon {\rm GL}_n \to {\rm GL}_{n+1} ; g \mapsto \begin{pmatrix} g & \\ & 1  \end{pmatrix}$. We hereafter assume that ${\mathcal K}_n = \iota^{-1}( {\mathcal K}_{n+1} )$ holds. Then the composition $j={\rm p}_{n+1}\circ \iota \colon \mathcal{Y}^{(n)}_{\mathcal{K}_n}\rightarrow Y^{(n+1)}_{\mathcal{K}_{n+1}}$ is a proper map due to the lemma of Borel and Prasad (see \cite[Lemma 2.7]{ash80} for the proof). 

Now let $\boldsymbol{\mu}^{(n+1)} \in ( {\mathbf Z}^{n+1} )^{I_\mathsf{F}}$ and $\boldsymbol{\mu}^{(n)} \in ( {\mathbf Z}^{n} )^{I_\mathsf{F}}$ be as in Section \ref{sec:intlPi}. Suppose that $\boldsymbol{\mu}^{ (n+1) }$ and $\boldsymbol{\mu}^{ (n) }$ satisfy the interlace condition (Definition~\ref{dfn:intlauto}). By Lemma~\ref{lem:crit} (1), the Rankin--Selberg $L$-function $L(s,\pi^{(n+1)}\times \pi^{(n)})$ then admits a critical point.  
Take $m\in \mathrm{Crit}(\pi^{(n+1)},\pi^{(n)})$ and let $\nabla^m \colon {\mathcal V}(  \boldsymbol{\mu}^{ (n+1), \vee } ) \to {\mathcal V}(  \boldsymbol{\mu}^{ (n) } ) \otimes {\rm det}^{m}$
 denote  the ${\rm GL}_n({\sf F}_{\mathbf A, \infty})$-equivariant morphism derived form the branching rule for the pair $({\rm GL}_{n+1},  {\rm GL}_n )$; its existence is justified due to Lemma~\ref{lem:crit}. 
Take a number field $E$ containing both $\mathbf{Q}(\boldsymbol{\mu}^{(n+1)})$ and ${\mathbf Q}(\boldsymbol{\mu}^{(n)})$, 
and let ${\mathcal V}(  \boldsymbol{\mu}^{ (n+1),\vee } )_E$ and ${\mathcal V}(  \boldsymbol{\mu}^{ (n) } )_E$ be $E$-rational models of ${\mathcal V}(  \boldsymbol{\mu}^{ (n+1),\vee } )$ and ${\mathcal V}(  \boldsymbol{\mu}^{ (n) } )$, respectively. We use the symbol $\widetilde{\mathcal V}(  \boldsymbol{\mu}^{ (n+1),\vee } )_E$ (resp.\ $\widetilde{\mathcal V}(  \boldsymbol{\mu}^{ (n),\vee } )_E$) for the local system on $Y^{(n+1)}_{\mathcal{K}_{n+1}}$ (resp.\ on $Y^{(n)}_{\mathcal{K}_n}$) defined by ${\mathcal V}(  \boldsymbol{\mu}^{ (n+1),\vee } )_E$ (resp.\ ${\mathcal V}(  \boldsymbol{\mu}^{ (n),\vee} )_E$).
Then, fixing a $\mathrm{GL}_n$-equivariant pairing $\langle \cdot, \cdot \rangle \colon  {\mathcal V}(  \boldsymbol{\mu}^{ (n) } )_E \times {\mathcal V}(  \boldsymbol{\mu}^{ (n) ,\vee } )_E\to E$, 
we can construct the cup product pairing by virtue of the numerical coincidence ${\rm b}_{n+1} + {\rm b}_n = {\rm d}_n$ (see \cite[Proposition~2.17]{rag16}): 
\begin{align*}
   H^{{\rm b}_{n+1}}_{\rm c}& ( Y^{(n+1)}_{{\mathcal K}_{n+1}, x} ,   \widetilde{\mathcal V} (\boldsymbol{\mu}^{(n+1), \vee}  )_E  )  
       \times  H^{{\rm b}_{n}} ( Y^{(n)}_{{\mathcal K}_{n}, x} ,   \widetilde{\mathcal V} (\boldsymbol{\mu}^{(n), \vee}  )_E  )                    \\
& \xrightarrow{ (j^\ast,  {\rm p}^\ast_n ) } 
        H^{{\rm b}_{n+1}}_{\rm c} ( {\mathcal Y}^{(n)}_{{\mathcal K}_{n}, x} ,  j^\ast \widetilde{\mathcal V} (\boldsymbol{\mu}^{(n+1), \vee}  )_E  )  
       \times  H^{{\rm b}_{n}} ( {\mathcal Y}^{(n)}_{{\mathcal K}_{n}, x} ,  {\rm p}^\ast_{n} \widetilde{\mathcal V} (\boldsymbol{\mu}^{(n), \vee}  )_E  )                      \\
& \xrightarrow{  (\nabla^m, {\rm id})  } 
        H^{{\rm b}_{n+1}}_{\rm c} ( {\mathcal Y}^{(n)}_{{\mathcal K}_{n}, x} , {\rm p}^\ast_{n} (\widetilde{\mathcal V} (\boldsymbol{\mu}^{(n)}) \otimes {\det}^m )_E  )  
       \times  H^{{\rm b}_{n}} ( {\mathcal Y}^{(n)}_{{\mathcal K}_{n}, x} ,  {\rm p}^\ast_{n} \widetilde{\mathcal V} (\boldsymbol{\mu}^{(n),\vee})_E )                      \xrightarrow{   \langle\cdot, \cdot \rangle \circ \cup } 
          H^{ {\rm d}_{n}} ( {\mathcal Y}^{(n)}_{{\mathcal K}_{n}, x} ,  \widetilde{E}(m)  ), 
\end{align*}
where $\widetilde{E}(m)$ is the local system on ${\mathcal Y}^{(n)}_{{\mathcal K}_{n}, x}$ 
     which is determined by 1-dimensional representation $\det^{m}$ over $E$.
Write $\widetilde{E} = \widetilde{E}(0)$ and define a twisting morphism ${\rm Tw}_{m} : \widetilde{E}(m) \to \widetilde{E}$ by 
\begin{align*}
  \widetilde{E}(m) \longrightarrow \widetilde{E} ;  (g, a_g) \longmapsto  (g,  \frac{ |\det g|^{m}_{\mathbf A} }{(\det g_\infty)^{m}} a_g).   
\end{align*}
Combining these maps with the Poincar\'e duality, we finally obtain the following pairing:
\begin{align*}
   H^{{\rm b}_{n+1}}_{\rm c} &( Y^{(n+1)}_{{\mathcal K}_{n+1}, x} ,   \widetilde{\mathcal V} (\boldsymbol{\mu}^{(n+1), \vee}  )_E  )  
       \times  H^{{\rm b}_{n}} ( Y^{(n)}_{{\mathcal K}_{n}, x} ,   \widetilde{\mathcal V} (\boldsymbol{\mu}^{(n), \vee}  )_E  )                    
\longrightarrow H^{ {\rm d}_{n}} ( {\mathcal Y}^{(n)}_{{\mathcal K}_{n}, x} ,  \widetilde{E}(m)  )   \\
& \xrightarrow{   {\rm Tw}_{m} }   H^{ {\rm d}_{n}} ( {\mathcal Y}^{(n)}_{{\mathcal K}_{n}, x} ,  \widetilde{E}  )   
 \xrightarrow{   - \cap [ {\mathcal Y}^{(n)}_{  {\mathcal K}_n, x }  ] }   E ; 
  (\xi, \eta) \longmapsto \int_{[{\mathcal Y}^{(n)}_{{\mathcal K}_n, x}   ]} {\rm Tw}_{m} \langle \nabla^{m} j^\ast \xi \cup {\rm p}^\ast_n \eta  \rangle. 
\end{align*}

\begin{lem}\label{lem:epadjoint}
For each $v \in \Sigma_{\mathsf{F},\mathbf{R}}$, let $\epsilon^\ast_{n, v}$ denote the involution on  
$H^{ {\rm d}_n }(  {\mathcal Y}^{(n)}_{{\mathcal K}_n, x}  , E   )$ induced by $R_{\epsilon_{n, v}}$. 
Then we have 
\begin{align*}
     \int_{  \epsilon_{n, v, \ast} [{\mathcal Y}^{(n)}_{{\mathcal K}_n, x}   ]} {\rm Tw}_{m} \langle \nabla^{m} j^\ast  \xi \cup {\rm p}^\ast_n \eta  \rangle
     =   \int_{[{\mathcal Y}^{(n)}_{{\mathcal K}_n, x}   ]}   \epsilon^\ast_{n, v}  {\rm Tw}_{m} \langle \nabla^{m}j^\ast \xi \cup {\rm p}^\ast_n \eta  \rangle.
\end{align*}
\end{lem}

\begin{proof}
This is directly derived from the functoriality of the cap product (see \cite[page 241]{hat02}). To be precise, let $X$ be a topological space with subspaces $A,B$ and  $X'$ another topological space with subspaces $A', B'$. Consider a continuous map $f\colon X\to X^\prime$  satisfying $f(A) \subset A^\prime$ and $f(B) \subset B^\prime$. Then the equality $f_\ast(\alpha) \cap \varphi = f_\ast  (\alpha \cap f^\ast(\varphi))$ holds for $\alpha\in H_k(X,A\cup B;E)$ and $\varphi\in H^l(X',A';E)$; 
or in other words, the following diagram commutes:   
\begin{align*}
 \begin{CD}
   H_k(X, A\cup B; E) @.    \times @.     H^l(X, A ; E)   @>\cap>>   H_{k-l}(X, B ; E)       \\
   @Vf_{\ast}VV  @. @AAf^{\ast}A   @VVf_{\ast}V    \\
   H_k(X^\prime, A^\prime\cup B^\prime; E)  @.  \times @. H^l(X^\prime, A^\prime ; E)   @>\cap>>   H_{k-l}(X^\prime, B^\prime ; E).      
 \end{CD} 
\end{align*}
Note that $\epsilon_{n, v, \ast} \colon H_0( {\mathcal Y}^{(n)}_{  {\mathcal K}_{n}, x  }, E  ) \to H_0( {\mathcal Y}^{(n)}_{  {\mathcal K}_{n}, (-1)_vx  }, E  )$ can be identified with the identity map $E \xrightarrow{=} E$.
\end{proof}

Let $N$ denote either $n+1$ or $n$. 
For each $\varepsilon^{(N)} = ( \varepsilon^{(N)}_v )_{v \in \Sigma_{\mathsf{F}, {\mathbf R}}  } \in \{ \pm 1 \}^{\Sigma_{\mathsf{F}, {\mathbf R}}}$,   
define
    $H^{{\rm b}_N} ( Y^{(N)}_{{\mathcal K}_N},  \widetilde{\mathcal V}(\boldsymbol{\mu}^{(N), \vee})     )[\varepsilon^{(N)}]$
  to be the subspace of $H^{{\rm b}_N} ( Y^{(N)}_{{\mathcal K}_N},  \widetilde{\mathcal V}(\boldsymbol{\mu}^{(N), \vee})     )$ 
 on  which $\epsilon_{N, v}^\ast$ acts as $\varepsilon^{(N)}_v$ for each $v\in \Sigma_{\mathsf{F}, {\mathbf R}}$.

\begin{prop}\label{prop:sgn}
Let 
$\xi^{(N)}$ be an element of $H^{{\rm b}_N} ( Y^{(N)}_{{\mathcal K}_N},  \widetilde{\mathcal V}(\boldsymbol{\mu}^{(N), \vee})     )[ \varepsilon^{(N)}]$ for $N=n+1$ or $n$,
and let $\varphi \colon\mathrm{Cl}^+_{\mathsf{F}}(\mathcal{K}_n) \to {\mathbf C}^\times$  be a  Hecke character of finite order. 

Then the equality $\varphi_v(-1) \varepsilon^{(n+1)}_v\varepsilon^{(n)}_v  = (-1)^{m+n+1}$ for each $v\in \Sigma_{\mathsf{F}, {\mathbf R}}$ is necessary for the integral
\begin{align*}
  {\mathcal I}( \xi^{(n+1)}, \xi^{(n)}, \varphi ) 
  :=  \sum_{x \in {\rm Cl}^+_\mathsf{F}({\mathcal K}_n)}    
           \varphi(x)
             \int_{[{\mathcal Y}^{(n)}_{{\mathcal K}_n, x}   ]} {\rm Tw}_{m} \langle \nabla^{m} j^\ast \xi^{(n+1)} \cup {\rm p}^\ast_n \xi^{(n)}  \rangle   
\end{align*}
not to vanish.
\end{prop}

\begin{proof}
Let $v \in \Sigma_{\mathsf{F},\mathbf{R}}$ be an arbitrary real place. 
Then, due to Lemma \ref{lem:signfund}, we have
\begin{align*}
    \sum_{x \in {\rm Cl}^+_\mathsf{F}({\mathcal K}_n)} &   
           \varphi(x)
             \int_{  \epsilon_{n, v, \ast} [{\mathcal Y}^{(n)}_{{\mathcal K}_n, x}   ]} {\rm Tw}_{m} \langle \nabla^{m} j^\ast \xi^{(n+1)} \cup {\rm p}^\ast_n \xi^{(n)}  \rangle       \\
 &=   \sum_{x   \in {\rm Cl}^+_\mathsf{F}({\mathcal K}_n)}    
           \varphi(x)  (-1)^{n+1}
             \int_{[{\mathcal Y}^{(n)}_{{\mathcal K}_n, (-1)_vx}   ]} {\rm Tw}_{m} \langle \nabla^{m} j^\ast \xi^{(n+1)} \cup {\rm p}^\ast_n \xi^{(n)}  \rangle       \\
 &=   (-1)^{n+1} \varphi((-1)_v)   
         \sum_{x  \in {\rm Cl}^+_\mathsf{F}({\mathcal K}_n)}    
           \varphi((-1)_vx ) 
             \int_{[{\mathcal Y}^{(n)}_{{\mathcal K}_n, (-1)_vx}   ]} {\rm Tw}_{m} \langle \nabla^{m} j^\ast \xi^{(n+1)} \cup {\rm p}^\ast_n \xi^{(n)}  \rangle       \\
 &=   (-1)^{n+1} \varphi_v(-1)     {\mathcal I}( \xi^{(n+1)}, \xi^{(n)}, \varphi ). 
\end{align*}
Meanwhile we also have 
\begin{align*}
   \sum_{x \in {\rm Cl}^+_\mathsf{F}({\mathcal K}_n)}    &
           \varphi(x)
             \int_{    [{\mathcal Y}^{(n)}_{{\mathcal K}_n, x}   ]}   
             \epsilon^\ast_{n, v}  {\rm Tw}_{m} \langle \nabla^{m} j^\ast \xi^{(n+1)} \cup {\rm p}^\ast_n \xi^{(n)}  \rangle       \\
&=    \sum_{x \in {\rm Cl}^+_\mathsf{F}({\mathcal K}_n)}    
           \varphi(x)
             \int_{    [{\mathcal Y}^{(n)}_{{\mathcal K}_n, x}   ]}   
              (\det \epsilon_{n, v})^{m}    {\rm Tw}_{m} \langle \nabla^{m}    j^\ast\epsilon^\ast_{n+1, v}   \xi^{(n+1)} \cup   {\rm p}^\ast_n \epsilon^\ast_{n, v}  \xi^{(n)}  \rangle       \\
&=     \sum_{x \in {\rm Cl}^+_\mathsf{F}({\mathcal K}_n)}    
          \varphi(x)
             \int_{    [{\mathcal Y}^{(n)}_{{\mathcal K}_n, x}   ]}   
              (-1)^m {\rm Tw}_{m} \langle  \varepsilon^{(n+1)}_v  \nabla^{m} j^\ast \xi^{(n+1)} \cup  \varepsilon^{(n)}_v  {\rm p}^\ast_n \xi^{(n)}  \rangle       \\
& =   (-1)^{m}      \varepsilon^{(n+1)}_v   \varepsilon^{(n)}_v
    {\mathcal I}( \xi^{(n+1)}, \xi^{(n)}, \varphi ). 
\end{align*}
Lemma~\ref{lem:epadjoint} shows that the two summations above coincide, and thus $(-1)^{n+1}\varphi_v(-1)=(-1)^m\varepsilon^{(n+1)}_v\varepsilon^{(n)}_v$ follows unless the integral $\mathcal{I}(\xi^{(n+1)},\xi^{(n)},\varphi)$ vanishes.
\end{proof}

In the following, we often consider that, at a complex place $v\in \Sigma_{\mathsf{F},\mathbf{C}}$, the signature $\varepsilon^{(N)}_v$ is defined as $\varepsilon^{(N)}_v=1$; indeed, since $\epsilon_{N,v}$ belongs to $C_v^{(N)}={\rm U}(N)$, the induced actions $\epsilon_{N,v,*}$ and $\epsilon_{N,v}^*$ must be trivial.

\subsection{The Rankin--Selberg $L$-functions}\label{sec:RSLfn}

For $N = n, n+1$, 
let $\pi^{(N)}$ be a cohomological irreducible cuspidal automorphic representation of ${\rm GL}_N(\mathsf{F}_{\mathbf A})$ of highest weight $\boldsymbol{\mu}^{(N)}$, satisfying (\ref{eq:auto_assump}).    
Assume that $\pi^{(N)}_{\mathrm{fin}}$ has a nontrivial vector fixed by a compact open subgroup ${\mathcal K}_N \subset {\rm GL}_N( \mathsf{F}_{{\mathbf A}, {\rm fin}})$. 
Let $\mathfrak{g}_{N,\infty}$ 
denotes the complexification of the Lie algebra of $\mathrm{GL}_N(\mathsf{F}_{\mathbf{A},\infty})$. 
Then, due to Proposition~\ref{prop:gK} and K\"unneth formula for relative Lie algebra cohomologies, we obtain an isomorphism
\begin{align} \label{eq:Kunneth}
 H^{{\rm b}_N}(\mathfrak{g}_{N,\infty},K_\infty^{(N)};\pi^{(N)}_\infty\otimes_\mathbf{C}\mathcal{V}(\boldsymbol{\mu}_\pi^{(N),\vee}))\cong \bigotimes_{v\in \Sigma_{\mathsf{F},\infty}} H^{{\rm b}_{N,v}}(\mathfrak{g}_{N,v},K^{(N)}_v; \pi^{(N)}_v\otimes_{\mathbf{C}}\mathcal{V}(\boldsymbol{\mu}^{(N),\vee}_{\pi,v}))
\end{align}
at the bottom degree. 
Now let us consider the canonical map
\begin{align}\label{eq:genES}
    H^{{\rm b}_N}(\mathfrak{g}_{N,\infty},K_\infty^{(N)};\pi^{(N)}_\infty\otimes_\mathbf{C}\mathcal{V}(\boldsymbol{\mu}_\pi^{(N),\vee}))
           \otimes \pi^{(N)}_{\rm fin}
           \longrightarrow 
           H^{{\rm b}_N}_{\rm cusp}(   Y^{(N)}_{ {\mathcal K}_{N}  },    \widetilde{\mathcal V}(\boldsymbol{\mu}^{(N), \vee}_\pi)     ).
\end{align}
The cohomology groups $H^{{\rm b}_N}(\mathfrak{g}_{N,\infty},K_\infty^{(N)};\pi^{(N)}_\infty\otimes_\mathbf{C}\mathcal{V}(\boldsymbol{\mu}_\pi^{(N),\vee}))$ and $H^{{\rm b}_N}_{\rm cusp}(Y^{(N)}_{ {\mathcal K}_{N}  },    \widetilde{\mathcal V}(\boldsymbol{\mu}^{(N), \vee}_\pi))$ admit the action of  $\prod_{v\in \Sigma_{\mathsf{F},\mathbf{R}}} {\rm O}(N)/{\rm SO}(N)$, called the {\em component action}, which is equivariant under the decomposition (\ref{eq:Kunneth}) and the map (\ref{eq:genES}). 
We define $H^{{\rm b}_N}_{\rm cusp} ( Y^{(N)}_{{\mathcal K}_N},  \widetilde{\mathcal V}(\boldsymbol{\mu}^{(N), \vee}_\pi)     )[\pi^{(N)}, \varepsilon^{(N)}]$ for any $\varepsilon^{(N)}=(\varepsilon^{(N)}_v)_{v\in \Sigma_{\mathsf{F},\mathbf{R}}} \in \{\pm 1\}^{\Sigma_{\mathsf{F},\mathbf{R}}}$ as the image of the $\varepsilon^{(N)}$-eigenspace of  $H^{\mathrm{b}_N}(\mathfrak{g}_{N,\infty},K^{(N)}_\infty;\pi^{(N)}_\infty\otimes_{\mathbf{C}}\mathcal{V}(\pi^{(N),\vee}_{\pi}))\otimes \pi^{(N)}_{\mathrm{fin}}$ under (\ref{eq:genES}). 

Now, 
for each $\epsilon\in \{\pm 1\}$,   
let $\psi^\epsilon_{\mathbf Q} \colon  {\mathbf Q} \backslash {\mathbf Q}_{\mathbf A}  \rightarrow  {\mathbf C}^\times$ denote the additive character of  ${\mathbf Q} \backslash {\mathbf Q}_{\mathbf A} $ which is characterized by  $\psi^\epsilon_{\mathbf{Q},\infty}( x ) = \exp(2\pi\sqrt{-1} \epsilon x)$ for $x\in {\mathbf R}$. We extend it to     
 $\psi^\epsilon \colon \mathsf{F} \backslash \mathsf{F}_\mathbf{A}\rightarrow \mathbf{C}^\times$ by setting $\psi^\epsilon = \psi^\epsilon_{\mathbf{Q}}\circ \mathrm{Tr}_{\mathsf{F}/\mathbf{Q}}$. 
If $\epsilon = +1$, we write $\psi^\epsilon$ as $\psi$ for the simplicity.
Let ${\mathcal W}(\pi^{(N)}, \psi^\epsilon) \cong \bigotimes^\prime_{v\in \Sigma_{\mathsf{F}}} {\mathcal W}(\pi^{(N)}_v, \psi_v^\epsilon)$ be the $\psi^\epsilon$-Whittaker model of $\pi^{(N)}$ for each $\epsilon\in \{\pm 1\}$.
Note that, for each $v\in \Sigma_{\mathsf{F}, \infty}$, the ${\mathbf C}$-vector space 
\begin{align}\label{eq:gkwhittv}
H^{{\rm b}_{N,v} } ( \mathfrak{g}_{ N, v}, K^{(N)}_{v} ; \pi^{(N)}_v     \otimes_\mathbf{C} {\mathcal V}(\boldsymbol{\mu}^{(N), \vee}_{\pi,v})      )  [ \varepsilon^{(N)}_v  ]
     \cong  
     \left(      {\mathcal W}(\pi^{(N)}_v, \psi_v^\epsilon)   
              \otimes \left(\bigwedge^{{\rm b}_{N,v} }   (\mathfrak{g}_{N, v}  / \mathfrak{k}_{N, v} )   \right)   
              \otimes V(  \boldsymbol{\mu}^{(N),\vee}_{\pi,v} )        \right)^{K^{(N)}_{v}}  [  \varepsilon^{(N)}_v ]   
\end{align}
is of dimension at most $1$, where ${\mathfrak k}_{N, v}$ denotes the complexification of the Lie algebra of $K^{(N)}_{v}$ (here we set $\varepsilon^{(N)}_v=1$ for $v\in \Sigma_{\mathsf{F},\mathbf{C}}$ as convention). 
Take and fix a nontrivial element $[\pi^{(N)}_v]^{\varepsilon^{(N)}_v}$ in the right-hand side 
of (\ref{eq:gkwhittv}), 
and write it down as
\begin{align} \label{eq:cohclass}
\begin{aligned}
     [\pi^{(n+1)}_v]^{\varepsilon^{(n+1)}_v}
        &= \sum_{w_v \in {\mathcal W}(\pi^{(n+1)}_v, \psi_v) } 
           \; \sum_{    X_v \in   \bigwedge^{{\rm b}_{n+1,v} }   (\mathfrak{g}_{n+1, v}  / \mathfrak{k}_{n+1, v} )       }    \; 
            \sum_{  \boldsymbol{v}_v \in V(  \boldsymbol{\mu}^{(n+1),\vee}_{\pi,v} )  } w_v \otimes X_v \otimes  \boldsymbol{v}_v, \\
     [\pi^{(n)}_v]^{\varepsilon^{(n)}_v}
        &= \sum_{w'_v \in {\mathcal W}(\pi^{(n)}_v, \psi_v^{-1}) } 
           \; \sum_{    X'_v \in   \bigwedge^{{\rm b}_{n,v} }   (\mathfrak{g}_{n, v}  / \mathfrak{k}_{n, v} )       }    \; 
            \sum_{  \boldsymbol{v}'_v \in V(  \boldsymbol{\mu}^{(n),\vee}_{\pi,v} )  } w'_v \otimes X'_v \otimes  \boldsymbol{v}'_v.
\end{aligned}
 \end{align}
Each $X_v \in \bigwedge^{{\rm b}_{N,v} }   (\mathfrak{g}_{N, v}  / \mathfrak{k}_{N, v} )$ in the  identity above is described as
  a $\mathbf{C}$-linear combination of the wedge products of $\{ E_{11, v}-E_{ii,v} \}_{2\leq i  \leq n}\cup \{E_{ij,v}+E_{ji,v}\}_{1\leq i<j\leq n}$ (resp.\ $\{ E^\pm_{11, v}-E^\pm_{ii,v} \}_{2\leq i  \leq n}\cup \{E^\pm_{ij,v}+E^\pm_{ji,v} \}_{1\leq i<j\leq n}$) if $v$ is real (resp.\ complex).
Set $[\pi^{(N)}_\infty]^{\varepsilon^{(N)}} = \bigotimes_{v\in \Sigma_{\mathsf{F}, \infty}}  [\pi^{(N)}_v]^{\varepsilon^{(N)}_v}$. 
Then, evaluating at $[\pi^{(N)}_\infty]^{\varepsilon^{(N)}}$ in  (\ref{eq:genES}) via (\ref{eq:Kunneth}), we obtain a map 
\begin{align}\label{eq:Fscr}
  \mathscr{F}^{\varepsilon^{(N)}}_{\pi^{(N)}} \colon {\mathcal W}(\pi^{(N)}_{\rm fin}, \psi^\epsilon_{{\rm fin} })  
      =  {\bigotimes_{v\in \Sigma_{\mathsf{F}, {\rm fin}}}\!\!\!}^\prime \;{\mathcal W}(\pi^{(N)}_v, \psi^\epsilon_{v})
    \xrightarrow{\,[\pi^{(N)}_\infty]^{\varepsilon^{(N)}}\otimes -\,}   
         H^{{\rm b}_N}_{\rm cusp}   (   Y^{(N)}_{ {\mathcal K}_N}  ,    \widetilde{\mathcal V}(\boldsymbol{\mu}^{(N), \vee}_\pi)     )[\varepsilon^{(N)}].
\end{align}

Let $\varphi \colon \mathsf{F}^\times \backslash \mathsf{F}^\times_{\mathbf A} \to {\mathbf C}^\times$ be a Hecke character of finite order.    
For each $v\in \Sigma_{\mathsf{F}}$, take Whittaker functions $w_v \in {\mathcal W}(\pi^{(n+1)}_v, \psi_v)$ and $w^\prime_v \in {\mathcal W}(\pi^{(n)}_v, \psi_v^{-1})$, and define the {\em local zeta integral} $\mathcal{I}_v(s,w_v,w_v',\varphi_v)$ as 
\begin{align*}
     {\mathcal I}_v(s, w_v, w^\prime_v, \varphi_v) 
     =  \int_{  {\rm N}_n(\mathsf{F}_v) \backslash {\rm GL}_n(\mathsf{F}_v) }   
           w_v (  {\rm diag}(g_v, 1)  )
           w^\prime_v (g_v)     
           \varphi_v(\det g_v)
           \lvert \det g_v\rvert^{s-\frac{1}{2}}_v
             {\rm d} g_v.          
 \end{align*}
 Here we normalize a Haar measure ${\rm d} g_v$ on $\mathrm{GL}_n(\mathsf{F}_v)$ so that 
 ${\rm vol}(\mathrm{GL}_n(\mathcal{O}_{\mathsf{F},v}),  {\rm d} g_v) = 1$ holds if $v$ is finite,
 and that ${\rm d} g_v$  
 corresponds to a fixed volume form $\omega_{n,v}$ on $Y^{(n)}_{{\mathcal K}_n}$ if $v$ is infinite (see Section \ref{sec:lie}).

 As in Section~\ref{sec:Cup}, fix a ${\rm GL}_n(\mathsf{F}_{\mathbf A, \infty})$-equivariant pairing
 \begin{align*}
     \langle \cdot, \cdot  \rangle  \colon
           \bigl({\mathcal V}(\boldsymbol{\mu}^{(n)} ) \otimes {\rm det}^{m}\bigr) 
          \times  {\mathcal V}(\boldsymbol{\mu}^{(n), \vee} )
           \to \mathbf{C}({\rm det}^{m}).
 \end{align*}
Write the image of a pair
$(  [\pi^{(n+1)}_v]^{\varepsilon^{(n+1)}_v},  [\pi^{(n)}_v]^{\varepsilon^{(n)}_v}   )$
   under the map $(j^* \cdot\wedge \cdot )\otimes \langle \nabla^m \cdot , \cdot \rangle$ as 
\begin{align*}
    \sum_{w_v \in {\mathcal W}(\pi^{(n+1)}_v, \psi_v) }
    \sum_{w'_v \in {\mathcal W}(\pi^{(n)}_v, \psi_v^{-1}) }
    c^{(m)} (w_v, w'_v)
    w_v w'_v \omega_{n,v}
\end{align*}
for $c^{(m)} (w_v, w'_v) \in \mathbf{C}$.
Then define the archimedean local zeta integral $\widetilde{\mathcal I}_v(s, [\pi^{(n+1)}_v]^{\varepsilon^{(n+1)}_v},  [\pi^{(n)}_v]^{\varepsilon^{(n)}_v}  , \varphi_v)$ for $[\pi^{(n+1)}_v]^{\varepsilon^{(n+1)}_v}$ and $[\pi^{(n)}_v]^{\varepsilon^{(n)}_v}$ as
\begin{align*}
   \widetilde{\mathcal I}_v&(s, [\pi^{(n+1)}_v]^{\varepsilon^{(n+1)}_v},  [\pi^{(n)}_v]^{\varepsilon^{(n)}_v}  , \varphi_v) \\
   &=     \sum_{w_v \in {\mathcal W}(\pi^{(n+1)}_v, \psi_v) } \;
    \sum_{w'_v \in {\mathcal W}(\pi^{(n)}_v, \psi_v^{-1}) }
    c^{(m)} (w_v, w'_v) 
    {\mathcal I}_v(s, w_v, w'_v, \varphi_v).
\end{align*}

The following statement immediately follows from the construction above combined with unfolding calculations of integrals:

\begin{lem}\label{lem:zeta_int}
For each $\bigotimes_{v\in \Sigma_{\mathsf{F}, {\rm fin}}} w_v \in {\mathcal W}(\pi^{(n+1)}_{\rm fin}, \psi_{{\rm fin} })$
              and $\bigotimes_{v\in \Sigma_{\mathsf{F}, {\rm fin}}} w^\prime_v  \in {\mathcal W}(\pi^{(n)}_{\rm fin}, \psi^{-1}_{{\rm fin} })$,    
we have 
\begin{align*}
  & {\mathcal I}\left(  \mathscr{F}^{\varepsilon^{(n+1)}}_{\pi^{(n+1)}} \left({\bigotimes}_{v\in \Sigma_{\mathsf{F}, {\rm fin}}} w_v\right),     
                      \mathscr{F}^{\varepsilon^{(n)}}_{\pi^{(n)}}   \left({\bigotimes}_{v\in \Sigma_{\mathsf{F}, {\rm fin}}} w^\prime_v\right),
                       \varphi \right)   \\ 
   &\qquad = [{\rm GL}_n(\widehat{\mathcal O}_\mathsf{F})  :   {\mathcal K}_n  ]
    \left. \left(\prod_{v\in\Sigma_{\mathsf{F}, {\rm fin}}} {\mathcal I}_v(s, w_v, w^\prime_v, \varphi_v) 
     \prod_{v\in\Sigma_{\mathsf{F}, \infty}}    
                 \widetilde{\mathcal I}_v(s, [\pi^{(n+1)}_v]^{\varepsilon^{(n+1)}_v},  [\pi^{(n)}_v]^{\varepsilon^{(n)}_v}  , \varphi_v)\right) \right|_{s=\frac{1}{2} + m}.
\end{align*}
\end{lem}

\subsection{Whittaker periods}\label{sec:Whitt}

Philosophically the zeta integrals appearing in Lemma~\ref{lem:zeta_int} are expected to give critical values of the Rankin--Selberg $L$-function for appropriate choice of the (local) Whittaker functions. Indeed, for a finite place $v\in \Sigma_{\mathsf{F},\mathrm{fin}}$ at which both $\pi^{(n+1)}$ and $\pi^{(n)}$ are unramified, 
the local zeta integral $\mathcal{I}_v(s,w_v,w_v^\prime,\varphi_v)$ gives the local $L$-factor $L_v(s,\pi^{(n+1)}\times \pi^{(n)} \times \varphi_v)$  
 if we choose $w_v$ and $w_v^\prime$ as Shintani's class-1 Whittaker functions \cite{shin76}.  
Concerning study of algebraicity of critical values of the  Rankin--Selberg $L$-functions, the most difficult issue is analysis of the archimedean local zeta integrals 
$\widetilde{\mathcal I}_v(s, [\pi^{(n+1)}_v]^{\varepsilon^{(n+1)}_v},  [\pi^{(n)}_v]^{\varepsilon^{(n)}_v}  , \varphi_v)\vert_{s=\frac{1}{2}+m}$.
We propose the following conjecture on the archimedean local  zeta integrals, which actually holds in several cases.

\begin{conj}\label{conj:mellin}
Let $\pi^{(n+1)}$ and $\pi^{(n)}$ be cohomological irreducible cuspidal automorphic representations of $\mathrm{GL}_{n+1}(\mathsf{F}_{\mathbf{A}})$ and $\mathrm{GL}_n(\mathsf{F}_\mathbf{A})$, respectively, which satisfy (\ref{eq:auto_assump}). Assume that they also satisfy the interlace condition $($see Definition~$\ref{dfn:intlauto})$. 
Let $N$ denote either $n+1$ or $n$. Then, for each $v\in \Sigma_{\mathsf{F},\infty}$ and each signature $\varepsilon^{(N)}_v \in \{\pm 1\}$ with respect to which the $(\mathfrak{g}_{N,v},K^{(N)}_v)$-cohomology group $(\ref{eq:gkwhittv})$ does not vanish, there exist a nontrivial $(\mathfrak{g}_{N,v},K^{(N)}_v)$-cohomology class $[\pi^{(N)}_v]^{\varepsilon^{(N)}_v}$ as in $(\ref{eq:cohclass})$ and a number field $E(\varepsilon^{(n+1)}_v,\varepsilon^{(n)}_v)$ satisfying 
\begin{align*}
   \widetilde{\mathcal I}_v(s, [\pi^{(n+1)}_v]^{\varepsilon^{(n+1)}_v},  [\pi^{(n)}_v]^{\varepsilon^{(n)}_v}  , \varphi_v) |_{s=\frac{1}{2} + m}
   \sim_{E(\varepsilon^{(n+1)}_v,\varepsilon^{(n)}_v)}
    L_v\left(\frac{1}{2}+m, \pi^{(n+1)} \times \pi^{(n)} \times \varphi \right)
\end{align*}
for every $m\in \mathrm{Crit}(\pi^{(n+1)},\pi^{(n)})$ if, at a real place $v$, the signatures $\varepsilon^{(n+1)}_v$ and $\varepsilon^{(n)}_v$ are chosen so that the equality$\varphi_v(-1)\varepsilon^{(n+1)}_v\varepsilon^{(n)}_v=(-1)^{m+n+1}$ holds.
\end{conj}

\begin{rem}\label{rem:conj}
\begin{enumerate}[label={\rm (\roman*)}]
\item \label{rem:conj1}
      When $n$ equals $2$, we have identified in  \cite[Theorem 7.1]{hn} the cohomology classes $[\pi^{(3)}_v]^{\varepsilon^{(3)}_v}$ and $[\pi^{(2)}_v]^{\varepsilon^{(2)}_v}$ appearing in Conjecture~\ref{conj:mellin} at a real place, by constructing the explicit Eichler--Shimura map and adopting normalization of the archimedean Whittaker functions discussed in \cite{him}.   
          Furthermore we have explicitly calculated the archimedean local zeta integrals $\widetilde{\mathcal I}_v(s, [\pi^{(3)}_v]^{\varepsilon^{(3)}_v},  [\pi^{(2)}_v]^{\varepsilon^{(2)}_v}  , \varphi_v) |_{s=\frac{1}{2} + m}$ in the case.   
\item If the base field is totally imaginary,  
           Conjecture \ref{conj:mellin} holds for an arbitrary $n$ due to \cite[Corollary 2.11]{im}. 
           Combining \ref{rem:conj1}, 
           we found that Conjecture \ref{conj:mellin} holds when $n$ equals 2 and the base field is, for example, the rational number field $\mathbf{Q}$ or an arbitrary quadratic field. 
\item The naive version of Conjecture \ref{conj:mellin} is found in \cite[Section 0, Question]{kms00}, where they ask whether
         the arthicedean local zeta integral $\widetilde{\mathcal I}_v(s, [\pi^{(n+1)}_v]^{\varepsilon^{(n+1)}_v},  [\pi^{(n)}_v]^{\varepsilon^{(n)}_v}  , \varphi_v) |_{s=\frac{1}{2} + m}$ is nontrivial. 
        The answer is affirmative due to Sun's recent result \cite{sun17}.  
        Conjecture \ref{conj:mellin} is a stronger and more precise version of \cite[Section 0, Question]{kms00}, which is necessary for further study of critical values of Rankin--Selberg $L$-functions; 
        namely it is fundamental for precise study of $p$-adic Rankin--Selberg $L$-functions discussed by Januszewski \cite{jan}.  
\end{enumerate}
\end{rem}

From now on we admit Conjecture~\ref{conj:mellin} 
and take cohomology classes $[\pi^{(n+1)}_v]^{\varepsilon^{(n+1)}_v}$ and $[\pi^{(n)}_v]^{\varepsilon^{(n)}_v}$ appearing there. 
For $N=n+1$ or $n$,   
this choice of $[\pi^{(N)}_v]^{\varepsilon^{(N)}_v}$ induces a map 
\begin{align*}
  \mathscr{F}^{\varepsilon^{(N)}}_{\pi^{(N)}} \colon {\mathcal W}(\pi^{(N)}_{\rm fin}, \psi^\epsilon_{{\rm fin} })  
    \longrightarrow 
         H^{{\rm b}_N}_{\rm cusp}   (   Y^{(N)}_{ {\mathcal K}_N}  ,    \widetilde{\mathcal V}(\boldsymbol{\mu}^{(N), \vee}_\pi)     )[\varepsilon^{(N)}]
\end{align*}
as in (\ref{eq:Fscr}). 
Then, for each $\alpha \in \mathrm{Aut}(\mathbf{C})$, there exists a complex number $p({}^{\alpha} \pi^{(N)}, \varepsilon^{(N)}) \in \mathbf{C}^\times$   
such that the normalized maps 
$\mathscr{F}^{\varepsilon^{(N)},\circ}_{{}^\alpha\pi^{(N)}}
:= p({}^\alpha\pi^{(N)}, \varepsilon^{(N)})^{-1} \mathscr{F}^{\varepsilon^{(N)}}_{{}^\alpha \pi^{(N)}}$ fit into  
 the following commutative diagram (recall that ${}^{\alpha}\boldsymbol{\mu}^{(N)}_\pi$ and ${}^{\alpha}\pi^{(N)}$ are defined in Sections~\ref{sec:FinDim} and \ref{sec:CloConj} respectively):   
\begin{align*}
\xymatrix{  {\mathcal W}(\pi^{(N)}_{\rm fin}, \psi^\epsilon_{{\rm fin} })    
                            \ar[rr]^-{\mathscr{F}^{\varepsilon^{(N)},\circ}_{\pi^{(N)}}} \ar[d]_{\alpha } \ar@{}[rrd]|{\text{\large $\circlearrowright$}} &&
                                      H^{{\rm b}_N}_{\rm cusp} ( Y^{(N)}_{ {\mathcal K}_N},    \widetilde{\mathcal V}(\boldsymbol{\mu}^{(N), \vee}_\pi)   )[\varepsilon^{(N)}]  \ar[d]^{\alpha} \\
    {\mathcal W}({}^{\alpha}\pi^{(N)}_{\rm fin}, \psi^\epsilon_{{\rm fin} })   
                    \ar[rr]_-{\mathscr{F}^{\varepsilon^{(N)},\circ}_{{}^\alpha\pi^{(N)}}}  &&
                       H^{{\rm b}_N}_{\rm cusp} ( Y^{(N)}_{ {\mathcal K}_N},    \widetilde{\mathcal V}({}^{\alpha}\boldsymbol{\mu}^{(N), \vee}_\pi)   )[\varepsilon^{(N)}].}      
\end{align*}
Furthermore, the complex number $p(\pi^{(N)},\varepsilon^{(N)})$, which we call the {\em Whittaker period} of $\pi^{(N)}$, is well defined unique up to multiplication by elements of $\mathbf{Q}(\pi^{(N)})^\times$, and the element of $(\mathbf{Q}(\pi^{(N)}) \otimes_{\mathbf{Q}} \mathbf{C})^\times$ defined by the collection of complex numbers $\{p({}^{\alpha} \pi^{(N)}, \varepsilon^{(N)}) \mid \alpha \in \mathrm{Aut}(\mathbf{C})\}$ is also unique up to multiplication by elements of $\mathbf{Q}(\pi^{(N)})^\times$.  For details, see  \cite[Definition/Proposition~3.3]{rs08} and \cite[Section~2.5.2]{rag16}.

Hereafter we fix a sufficiently large coefficient field $E$ so that it contains all of $\mathbf{Q}(\pi^{(n+1)})$, $\mathbf{Q}(\pi^{(n)})$, $\mathbf{Q}(\boldsymbol{\mu}^{(n+1)}_\pi)$ and $\mathbf{Q}(\boldsymbol{\mu}^{(n)}_\pi)$. Recall from Lemma~\ref{lem:crit} that, for each critical point $m\in {\rm Crit}(\pi^{(n+1)} , \pi^{(n)}  )$, we have a ${\rm GL}_n(\mathsf{F}_{\mathbf A, \infty})$-equivariant projection 
$\nabla^m \colon  V(\boldsymbol{\mu}^{(n+1), \vee}_\pi ) \to  V(\boldsymbol{\mu}^{(n)}_\pi ) \otimes \det^{m}$.  To discuss algebraicity of the critical values of the Rankin--Selberg $L$-function, we here verify rationality of the projection $\nabla^m$. 

\begin{lem}
For each critical point $m\in {\rm Crit}(\pi^{(n+1)} , \pi^{(n)}  )$, 
the ${\rm GL}_n(\mathsf{F}_{\mathbf A, \infty})$-equivariant projection 
$\nabla^m$ appearing in Lemma~$\ref{lem:crit}$ is defined over $E;$ that is, 
there exists a ${\rm GL}_n(\mathsf{F}_{\mathbf{A},\infty})$-equivariant projection of $E$-vector spaces 
$\nabla^m_E \colon  V(\boldsymbol{\mu}^{(n+1), \vee}_\pi )_E \to  V(\boldsymbol{\mu}^{(n)}_\pi )_E \otimes \det^{m}$
satisfying $\nabla^m = \nabla^m_E \otimes_E {\mathbf C}$. 
\end{lem}

\begin{proof}
Since both $\nabla^m( V(\boldsymbol{\mu}^{(n+1), \vee}_\pi )_E )$ and $V(\boldsymbol{\mu}^{(n)}_\pi )_E \otimes \det^{m}$ 
   are $E$-vector spaces of highest weight $\boldsymbol{\mu}^{(n)}_\pi$,      
Schur's lemma yields a complex constant $c\in \mathbf{C}$
such that $c \nabla^m( V(\boldsymbol{\mu}^{(n+1), \vee}_\pi )_E ) = V(\boldsymbol{\mu}^{(n)}_\pi )_E \otimes \det^{m}$ holds.
Hence $\nabla^m_E := c \nabla^m$ satisfies the statement.
\end{proof}

We now present the main statement of this section. 
For each $\sigma \in I_E$, take an automorphism $\alpha_\sigma \in \mathrm{Aut}(\mathbf{C})$ of $\mathbf{C}$ satisfying $\sigma=\alpha_\sigma \iota_0$. Then set
\begin{align*}
 \boldsymbol{p}(\pi^{(N)},\varepsilon^{(N)})=\left( p({}^{\alpha_\sigma}\pi^{(N)}, \varepsilon^{(N)})\right)_{\sigma\in I_E} \qquad \in (E\otimes_{\mathbf{Q}}\mathbf{C})^\times.
\end{align*} 
We also set $\mathcal{L}(s,\pi^{(n+1)}\times \pi^{(n)}\times \varphi)=\left(L(s,{}^{\alpha_\sigma}(\pi^{(n+1)}\times \pi^{(n)}\times \varphi))\right)_{\sigma\in I_E} \in E\otimes_\mathbf{Q}\mathbf{C}$.

\begin{thm}\label{thm:AlgRS}
Let $m\in \mathrm{Crit}(\pi^{(n+1)},\pi^{(n)})$ be a critical point of $L(s,\pi^{(n+1)}\times \pi^{(n)})$ and $\varphi\colon \mathsf{F}^\times \backslash \mathsf{F}_\mathbf{A}^\times\rightarrow \mathbf{C}^\times$ a Hecke character of finite order. Suppose that all of the following conditions are fulfilled$:$
\begin{itemize}
\item[--] the equality $\varphi_v(-1) \varepsilon^{(n+1)}_v\varepsilon^{(n)}_v  = (-1)^{m+n+1}$ holds for each $v\in \Sigma_{\mathsf{F}, {\mathbf R}}$.
\item[--] Conjecture $\ref{conj:mellin}$ holds.
\end{itemize}
Then the ratio
$\displaystyle \frac{\mathcal{L}(\frac{1}{2}+m,  \pi^{(n+1)} \times \pi^{(n)}  \times \varphi )}{  \boldsymbol{p}(\pi^{(n+1)}, \varepsilon^{(n+1)})  \boldsymbol{p}(\pi^{(n)}, \varepsilon^{(n)})    }$
is an element of $E^\times$.
\end{thm}

\begin{proof}
This is an immediate corollary of Conjecture \ref{conj:mellin} and \cite[Theorem 1.1]{rag16}.
\end{proof}

\section{Motivic interpretation of Whittaker periods}\label{sec:MotInt}

In the final section, we give a motivic interpretation of the Whittaker period $\boldsymbol{p}(\pi^{(n+1)},\varepsilon^{(n)})$ (Theorem \ref{thm:main}), using the explicit factorization formula (Proposition~\ref{prop:perdecomp}). The problem is that the motive $\mathcal{M}^{[\tau]}\otimes_{\mathsf{F}^{(\tau,\rho)}}\mathcal{N}^{[\rho\tau]}$ appearing in Proposition~\ref{prop:perdecomp} (2) a priori varies depending on the complex embedding $\tau\in I_\mathsf{F}$, and thus it is difficult to relate the product of $(\tau,\sigma)$-periods $\prod_{\tau\in I_{\mathsf{F}}} c^{\pm}_{(\iota_0, \sigma)}(\mathcal{M}^{[\tau]}\otimes_{\mathsf{F}^{(\tau,\rho)}} \mathcal{N}^{[\rho\tau]})$ to a special value of an appropriate $L$-function. We can overcome this difficulty if we assume the following:
\begin{align} \label{eq:basefield}
 \text{the base field $\mathsf{F}=F$ is a totally real field or a CM field.}
\end{align}
Indeed, in the former case there are no complex places and thus the problem does not arise, whereas in the latter case $F$ admits a unique complex conjugation $\rho$, and thus $\mathcal{M}^{[\tau]}\otimes_{F^{(\tau,\rho)}} \mathcal{N}^{[\rho\tau]}$ can be regarded as the extension of scalars of the motive $\mathcal{M}\otimes_{F}\mathcal{N}^\rho$, which does not depend on $\tau\in I_F$. Hence only in this section we always assume (\ref{eq:basefield}). Section~\ref{sec:motives_cm} is devoted to several preliminary facts on motives over CM fields. In Section~\ref{sec:GamPer} we introduce the $\Gamma$-factors of the Rankin--Selberg $L$-functions, and then rewrite Proposition~\ref{prop:perdecomp} into a manageable form (Proposition~\ref{prop:perprod}). Finally we verify our main theorem of this article in Section~\ref{sec:main}.

\subsection{Motives defined over CM fields}\label{sec:motives_cm}

We use notation on motives and periods from Sections~\ref{sec:mot} and \ref{sec:RSPer}. In this subsection we assume that the base field $F$ of motives is a CM field; that is, $F$ is a totally imaginary quadratic extension of $F^+$. The Galois group $\mathrm{Gal}(F/F^+)$ is generated by a unique nontrivial $F^+$-involution, which is also denoted as $\rho$. This notation is justified by the fact that $\rho\tau=\tau\rho$ holds for {\em every} $\tau\in I_F$. Let $\mathcal{M}$ and $\mathcal{N}$ be motives defined over $F$ with coefficients in $E$, and let $\mathcal{N}^\rho$ denote the extension of scalars of $\mathcal{N}$ from $F$ to itself according to $\rho\colon F\xrightarrow{\sim} F$ (this is also a motive defined over $F$). Obviously $I^0_{\iota_0, \sigma}(\mathcal{M}^{[\tau]}\otimes_{F^\tau} \mathcal{N}^{[\rho\tau]})$ coincides\footnote{For a CM field $F$, we have $F^{(\tau,\rho)}=F^\tau$ since $\rho\tau(F)$ coincides with $\tau(F)$.} with $I_{\tau,\sigma}^0(\mathcal{M}\otimes_F\mathcal{N}^\rho)$, and thus we have
$c^{\pm}_{(\iota_0, \sigma)}(\mathcal{M}^{[\tau]}\otimes_{F^\tau}\mathcal{M}^{[\rho\tau]})=c^{\pm}_{(\tau,\sigma)}(\mathcal{M}\otimes_F \mathcal{N}^\rho)$.

\begin{lem} \label{lem:res_ext}
Let $k'/k$ be a finite extension of fields. Then the Frobenius reciprocity
\begin{align*}
 \mathrm{Res}_{k'/k}(\mathcal{M}'\otimes_{k'}\mathcal{N}_{k'})\cong \mathrm{Res}_{k'/k}(\mathcal{M}')\otimes_k \mathcal{N}
\end{align*}
holds for pure motives $\mathcal{M}'$ and $\mathcal{N}$ respectively defined over $k'$ and $k$, where $\mathcal{N}_{k'}$ denotes the extension of scalars of $\mathcal{N}$ from $k$ to $k'$.
\end{lem}

\begin{proof}
 It suffices to show the statement for smooth projective varieties. Following \cite[Exemple 0.1.1.]{del79}, let $\coprod_{k'/k}$ denote the Grothendieck restriction of scalars functor defined as
\begin{align*}
 \coprod_{k'/k}(V\rightarrow \mathrm{Spec}\, k')=(V\rightarrow \mathrm{Spec}\, k' \rightarrow \mathrm{Spec}\, k).
\end{align*}
Then, due to the successive Cartesian diagram
\begin{align*}
 \xymatrix{
Z \ar[r] \ar[d] \ar@{}[rd]|{\text{\large $\lrcorner$}} & W_{k'} \ar[r] \ar[d] \ar@{}[rd]|{\text{\large $\lrcorner$}} & W \ar[d] \\
V' \ar[r] & \mathrm{Spec}\, k' \ar[r] & \mathrm{Spec}\, k,}
\end{align*}
we have $Z\cong V'\times_{\mathrm{Spec}\, k'} W_{k'}\cong \coprod_{k'/k}(V')\times_{\mathrm{Spec}\, k} W$ for smooth projective varieties $V'$ and $W$ respectively defined over $k'$ and $k$.
\end{proof}

\begin{lem} \label{lem:mot_CM}
Let $\mathcal{M}$ and $\mathcal{N}$ be pure motives defined over a CM field $F$ with coefficients in $E$.
\begin{enumerate}[label={\rm (\arabic*)}]
 \item For each $\tau \in I_F$, we have 
\begin{align*}
 \delta_{\tau\vert_{F^+}}(\mathrm{Res}_{F/F^+}(\mathcal{M}))&\sim_{E\otimes_{\mathbf{Q}} F} \delta_\tau(\mathcal{M}), & c^{\pm}_{\tau\vert_{F^+}}(\mathrm{Res}_{F/F^+}(\mathcal{M})) &\sim_{E\otimes_{\mathbf{Q}}F} c^{\pm}_\tau(\mathcal{M}).
\end{align*}
 \item We have 
    $ c^{\pm}_{(\tau,\sigma)}(\mathcal{M},\mathcal{N}) 
     \sim_{E^\sigma F^\tau} c^{\pm}_{(\tau,\sigma\vert_{F^+})}(\mathrm{Res}_{F/F^+}(\mathcal{M})\otimes_{F^+} \mathrm{Res}_{F/F^+}(\mathcal{N}))$.
\end{enumerate}
\end{lem}

\begin{proof}
\begin{enumerate}[label=(\arabic*)]
 \item For $\tau \in I_F$, we have 
\begin{align*}
 H_{\rm B}(\mathrm{Res}_{F/F^+}(\mathcal{M})_{\tau\vert_{F^+}})\otimes_{\mathbf{Q}} \mathbf{C} &= \left( H_{\rm B}(\mathcal{M}_\tau)\oplus H_{\rm B}(\mathcal{M}_{\rho\tau})\right)\otimes_{\mathbf{Q}}\mathbf{C}, \\
H_{\rm dR}(\mathrm{Res}_{F/F^+}(\mathcal{M}))\otimes_{F^+,\tau\vert_{F^+}} \mathbf{C} &=H_{\rm dR}(\mathcal{M}) \otimes_F \left(F\otimes_{F^+,\tau\vert_{F^+}} \mathbf{C}\right) =H_{\rm dR}(\mathcal{M}) \otimes_{F,\tau\oplus \rho \tau} (\mathbf{C}\oplus \mathbf{C})
\end{align*}
by definition. Therefore it is easy to see that the comparison isomorphism $I_{\tau\vert_{F^+}}(\mathrm{Res}_{F/F^+}(\mathcal{M}))$ coincides with $I_{\tau}(\mathcal{M})\oplus I_{\rho\tau}(\mathcal{M})$, from which the claim readily follows.
\item Lemma~\ref{lem:res_ext} implies
\begin{align*}
 \mathrm{Res}_{F/F^+}(\mathcal{M})\otimes_{F^+} \mathrm{Res}_{F/F^+}(\mathcal{N}) \cong \mathrm{Res}_{F/F^+}\left( \mathcal{M}\otimes_F \mathrm{Res}_{F/F^+}(\mathcal{N})_F\right), 
\end{align*}
where $\mathrm{Res}_{F/F^+}(\mathcal{N})_F$ denotes the extension of scalars of  $\mathrm{Res}_{F/F^+}(\mathcal{N})$ from $F^+$ to $F$. Meanwhile we readily observe that  $\mathrm{Res}_{F/F^+}(\mathcal{N})_F=\mathcal{N} \oplus \mathcal{N}^\rho$ holds. Combining these facts with the statement (1), we may easily deduce the claim. \qedhere
\end{enumerate}
\end{proof}

\begin{rem}\label{rem:n_rho}
As a final remark, note that the $(\tau,\sigma)$-periods of $\mathcal{N}^\rho$ can be identified with those of $\mathcal{N}$; this is just because $I^0_{\tau,\sigma}(\mathcal{N}^\rho)\colon \mathcal{B}_{\sigma}(\mathcal{N}^\rho_\tau)_\mathbf{C} \xrightarrow{\, \sim\,}\mathcal{T}_{\tau,\sigma}(\mathcal{N}^{\rho})_\mathbf{C}$ is the same as  $I^0_{\tau,\sigma}(\mathcal{N})\colon \mathcal{B}_{\sigma}(\mathcal{N}_\tau)_\mathbf{C}  \xrightarrow{\, \sim\,}\mathcal{T}_{\tau,\sigma}(\mathcal{N})_\mathbf{C}$ by construction, and thus we can take the same period matrix $\mathcal{Y}^{\mathcal{N}}_{\tau,\sigma}$ for $\mathcal{N}$ as that for $\mathcal{N}^{\rho}$. 
\end{rem}

\subsection{$\Gamma$-factors and periods}\label{sec:GamPer}

Hereafter let $F$ be a totally real field or a CM field and $F^+$ the maximal totally real subfield of $F$. Consider two pure motives $\mathcal{M}$ and $\mathcal{N}$ defined over $F$ with coefficients in $E$.  Let $\bullet$ denote either $\mathcal{M}$ or $\mathcal{N}$. If the rank $d(\bullet)$ of $\bullet$ is odd and the weight $\mathsf{w}(\bullet)$ is even, 
   we let $\boldsymbol{\varepsilon}^{(\bullet_\tau)} \in \{\pm 1\}$
   the signature of the action of the involution $F_{\infty,\tau}$ on the diagonal component $H^{\mathsf{w}(\bullet)/2,\mathsf{w}(\bullet)/2}(\bullet_\tau)$ if $\tau$ is real, and set $\boldsymbol{\varepsilon}^{(\bullet_\tau)}=1$ if $\tau$ is complex. Then the assumption on the involution introduced in Section~\ref{sec:t-per} is equivalent to the following:
\begin{itemize}
  \item[--] for each real $\tau$ and $\tau^\prime \in I_F$, we have  $\boldsymbol{\varepsilon}^{(\bullet_\tau)} = \boldsymbol{\varepsilon}^{(\bullet_{\tau^\prime})}$.  
\end{itemize}
Hereafter we always assume this condition, and write $\boldsymbol{\varepsilon}^{(\bullet)} = \boldsymbol{\varepsilon}^{(\bullet_\tau)} $ if $F$ is totally real and $\boldsymbol{\varepsilon}^{(\bullet)}=1$ otherwise.
Then we find the following factorization formula: 

\begin{prop}\label{prop:perprod}
Let $\mathcal{M}$ and $\mathcal{N}$ be as above. 
Suppose the following conditions are fulfilled$:$ 
\begin{itemize}
\item[--]  the motives $\mathcal{M}$ and $\mathcal{N}$ satisfies the strong interlace condition $(\ref{eq:strintlpq})$. 
\item[--] the motives $\mathcal{M}$ and $\mathcal{N}$ are regular $($see Definition~$\ref{dfn:regular})$.
\end{itemize} 
In particular, we have $d(\mathcal{N})=t^{\mathcal{N}_{\tau,\sigma}}$ and $d(\mathcal{M})=t^{\mathcal{M}_{\tau,\sigma}}=d(\mathcal{N})+1$.
For each $\bullet =\mathcal{M}$ and $\mathcal{N}$, set  
\begin{align*}
   \lambda( \bullet) &=  \sum_{\tau \in I_F } \sum^{ d(\bullet)}_{i=1} p^{\bullet_{\tau,\sigma}}_i ( d(\bullet) -i), & 
   \widetilde{c}( \bullet ) &=\left(\tilde{c}_\sigma(\bullet)\right)_{\sigma\in I_E}= \left( \prod_{\tau\in I_F}\prod^{ \lceil \frac{{d(\bullet)}}{2}-1 \rceil }_{\beta=1} c_{\beta,(\tau,\sigma)}(  \bullet  )\right)_{\sigma\in I_E} \in E\otimes_{\mathbf{Q}}\mathbf{C}.
\end{align*}
Then the following statements hold\,$:$
\begin{enumerate}[label={\rm (\arabic*)}]
 \item The integer $\lambda(\bullet)$ is independent of the embedding $\sigma \in I_E$.
 \item For each $m\in {\mathbf Z}$, we find that 
\begin{align*}
  L_\infty&(m, \mathrm{Res}_{F^+/\mathbf{Q}}(  \mathrm{Res}_{F/F^+}(\mathcal{M}) \otimes_{F^+} \mathrm{Res}_{F/F^+}(\mathcal{N}) )  )       
          c^+ \left( \mathrm{Res}_{F^+/\mathbf{Q}}({\rm Res}_{F/ F^+} (\mathcal{M}) \otimes_{F^+} \mathrm{Res}_{F/F^+}(\mathcal{N})) (m)  \right)    \\
  & \sim \left(  (2\pi\sqrt{-1})^{\lambda(  \mathcal{M}  )  }   \widetilde{c}( \mathcal{M} )    
               \cdot   (2\pi\sqrt{-1})^{ \lambda( \mathcal{N} )    }   \widetilde{c}( \mathcal{N} )\right)^{[F:F^+]} \\        
  &  \qquad  \cdot  \begin{cases}   c^{ (-1)^m   \boldsymbol{\varepsilon}^{(\mathcal{N})}     }(  {\rm Res}_{F/{\mathbf Q}}(\mathcal{M})  )  &  \text{if $F$ is totally real and $d(\mathcal{N})$ is odd},      \\
                                          c^{  (-1)^m  \boldsymbol{\varepsilon}^{(\mathcal{M})}   }(   {\rm Res}_{F/{\mathbf Q}}(\mathcal{N})   )  &  \text{if $F$ is totally real and $d(\mathcal{N})$ is even}, \\
c^{(-1)^m}(\mathrm{Res}_{F/\mathbf{Q}}(\mathcal{M}))^2   & \text{if $F$ is CM and $d(\mathcal{N})$ is odd}, \\
c^{(-1)^m}(\mathrm{Res}_{F/\mathbf{Q}}(\mathcal{N}))^2  & \text{if $F$ is CM and $d(\mathcal{N})$ is even}.
 \end{cases}%
\end{align*}
\end{enumerate}

%
\end{prop}

\begin{proof}
\begin{enumerate}[label=(\arabic*)]
 \item This is a direct consequence of Proposition~\ref{prop:auto_C}; indeed, for arbitrary $\sigma$ and $\sigma'\in I_E$, we have
\begin{align*}
 \sum_{\tau\in I_F}\sum_{i=1}^{d(\bullet)}p^{\bullet_{\tau,\sigma'}}_i(d(\bullet)-i)=\sum_{\tau \in I_F} \sum_{i=1}^{d(\bullet)}p^{\bullet_{\alpha^{-1}\tau,\sigma}}_i (d(\bullet)-i)\stackrel{\tilde{\tau}=\alpha^{-1}\tau}{=}\sum_{\tilde{\tau} \in I_F} \sum_{i=1}^{d(\bullet)}p^{\bullet_{\tilde{\tau},\sigma}}_i (d(\bullet)-i) 
\end{align*}
if we take $\alpha\in \mathrm{Aut}(\mathbf{C})$ satisfying $\sigma'=\alpha\sigma$.
 \item We only prove the statement when $F$ is a CM field; the proof is the same (or easier) when $F$ is totally real. Let us first compute $L_\infty\left(m,   \mathrm{Res}_{F^+/\mathbf{Q}}(\mathrm{Res}_{F/F^+}(\mathcal{M}) \otimes_{F^+} \mathrm{Res}_{F/F^+}(\mathcal{N})  )\right)   $. It suffices to show the statement for each $\sigma$-component ($\sigma\in I_E$).  
 Lemma~\ref{lem:qmp} 
 and the regularity of $\mathcal{M}$ and $\mathcal{N}$ yield that 
\begin{align*}
   L^\sigma_\infty&\left(m,   \mathrm{Res}_{F^+/\mathbf{Q}}(\mathrm{Res}_{F/F^+}(\mathcal{M}) \otimes_{F^+} \mathrm{Res}_{F/F^+}(\mathcal{N}))  \right) \\
&=L^\sigma_\infty(m,\mathrm{Res}_{F/\mathbf{Q}}(\mathcal{M}\otimes_F\mathcal{N}))L^\sigma_\infty(\mathrm{Res}_{F/\mathbf{Q}}(\mathcal{M}\otimes_F \mathcal{N}^\rho)) \\
    &= \prod_{\tau \in I_F} \prod_{2\leq i+j \leq d(\mathcal{N})+1} \Gamma_{\mathbf C} \left(m   -   p^{\mathcal{M}_{\tau,\sigma}}_i  -   p^{\mathcal{N}_{\tau,\sigma}}_j   \right)   \Gamma_{\mathbf C} \left(m   -   p^{\mathcal{M}_{\tau,\sigma}}_i  -   p^{\mathcal{N}_{\rho\tau,\sigma}}_j   \right)    \\
    &\sim  (2\pi \sqrt{-1})^{-  m  d(\mathcal{M}) d(\mathcal{N})  [F:{\mathbf Q}]  }
       \prod_{\tau \in I_F} \prod_{2\leq i+j \leq d(\mathcal{N})+1} (2\pi \sqrt{-1})^{  2 p^{\mathcal{M}_{\tau,\sigma}}_i  + p^{\mathcal{N}_{\tau,\sigma}}_j  +p^{\mathcal{N}_{\rho\tau,\sigma}}_j      }        \\
    &\sim  (2\pi \sqrt{-1})^{  - m d(\mathcal{M})   d(\mathcal{N})   [F:{\mathbf Q}]   } 
       \cdot (2\pi \sqrt{-1})^{  \sum^{d(\mathcal{N})}_{i=1} (d(\mathcal{N})+1-i) \sum_{\tau \in I_F}  (2p^{\mathcal{M}_{\tau,\sigma}}_i+p^{\mathcal{N}_{\tau,\sigma}}_i+p^{\mathcal{N}_{\rho\tau,\sigma}}_i)   } \\
    &\sim  (2\pi \sqrt{-1})^{  - m d(\mathcal{M})d(\mathcal{N})  [F:{\mathbf Q}]   } 
       \cdot (2\pi \sqrt{-1})^{ 2\sum_{\tau\in I_F} \sum_{i=1}^{d(\mathcal{N})} p^{\mathcal{M}_{\tau,\sigma}}_i (d(\mathcal{M})-i)  }  \\
& \hspace*{15em} \cdot (2\pi\sqrt{-1})^{2\sum_{\tau\in I_F} \sum_{i=1}^{d(\mathcal{N})} p^{\mathcal{N}_{\tau,\sigma}}_i(d(\mathcal{N})-i)+p^{\mathcal{N}_{\tau,\sigma}}_i} \\ 
    &\sim  (2\pi \sqrt{-1})^{  - m d(\mathcal{M})d(\mathcal{N})  [F:{\mathbf Q}]   } 
       \cdot (2\pi \sqrt{-1})^{2 \lambda(  \mathcal{M}  )  } 
       \cdot (2\pi \sqrt{-1})^{ 2 \lambda(  \mathcal{N}  )  +   2 \sum_{\tau \in I_F}   \sum^{n}_{i=1} p^{\mathcal{N}_{\tau,\sigma}}_i  }.
\end{align*}
Since    $ (2\pi \sqrt{-1})^{ \sum_{\tau \in I_F}  \sum^{n}_{i=1} p^{\mathcal{N}_{\tau,\sigma}}_i  }   = \delta_\sigma( {\rm Res}_{F/{\mathbf Q}} (\mathcal{N})  )^{-1}$, we have
\begin{align}\label{eq:gamma}
\begin{aligned}
    L^\sigma_\infty& \left(m, \mathrm{Res}_{F^+/\mathbf{Q}}(\mathrm{Res}_{F/F^+}(\mathcal{M}) \otimes_{F^+} \mathrm{Res}_{F/F^+}(\mathcal{N}) ) \right)   \\
   & \sim  (2\pi \sqrt{-1})^{  - md(\mathcal{M})d(\mathcal{N}) [F:{\mathbf Q}] + \lambda( \mathcal{M} )  +  \lambda( \mathcal{N} )  } \cdot \delta_\sigma(  {\rm Res}_{F/{\mathbf Q}} (\mathcal{N})  )^{-1}. 
\end{aligned}
 \end{align}

Next we compute $c_\sigma^\pm\left(\mathrm{Res}_{F^+/\mathbf{Q}}( \mathrm{Res}_{F/F^+}(\mathcal{M}) \otimes_{F^+} \mathrm{Res}_{F/F^+}(\mathcal{N}) )(m)  \right)$. Due to \cite[(5.1.8)]{del79}, we have
\begin{align}\label{eq:cTate}
\begin{aligned}
 c_\sigma^\pm&\left(\mathrm{Res}_{F^+/\mathbf{Q}}( \mathrm{Res}_{F/F^+}(\mathcal{M}) \otimes_{F^+} \mathrm{Res}_{F/F^+}(\mathcal{N}) )(m) \right) \\
  &\sim  (2\pi \sqrt{-1})^{ 2 md(\mathcal{M})d(\mathcal{N})  [F^+:{\mathbf Q}] }   c_\sigma^{\pm (-1)^m }\left(\mathrm{Res}_{F^+/\mathbf{Q}}( \mathrm{Res}_{F/F^+}(\mathcal{M}) \otimes_{F^+} \mathrm{Res}_{F/F^+}(\mathcal{N}) )\right). 
\end{aligned}
\end{align}
Then Lemma~\ref{lem:perfac}, Proposition \ref{prop:perdecomp} and Lemma~\ref{lem:mot_CM} imply
\begin{align}\label{eq:yoshida}
 c_\sigma^\pm &(  {\rm Res}_{F^+/{\mathbf Q}} (\mathrm{Res}_{F/F^+}(\mathcal{M})\otimes_{F^+}  \mathrm{Res}_{F/F^+}(\mathcal{N}))      )  \\
   &\sim \widetilde{c}_\sigma(  \mathcal{M}  )  
       \widetilde{c}_\sigma(  \mathcal{N}  )   
       \delta_\sigma( {\rm Res}_{F/{\mathbf Q}} (N) )  
       \cdot \begin{cases}  c_\sigma^+(\mathrm{Res}_{F/\mathbf{Q}}(\mathcal{M}))c_\sigma^+(\mathrm{Res}_{F/\mathbf{Q}}(\mathcal{M}))  &  \text{if $d(\mathcal{N})$ is odd},      \\
                      c_\sigma^+(\mathrm{Res}_{F/\mathbf{Q}}(\mathcal{N}))c_\sigma^+(\mathrm{Res}_{F/\mathbf{Q}}(\mathcal{N}))  &  \text{if $d(\mathcal{N})$ is even}.    \end{cases}
\end{align}

Combining (\ref{eq:gamma}), (\ref{eq:cTate}), (\ref{eq:yoshida}) and Lemma~\ref{lem:cpm},  we obtain the statement. \qedhere
\end{enumerate}
\end{proof}

\subsection{Motivic interpretation of the Whitterker periods} \label{sec:main}

As before, let $F$ be a totally real field or a CM field and $E$ a sufficiently large coefficient field. For a positive integer $N$, let ${\mathcal M}[ \pi^{(N)} ]$ be the conjectural pure motives attached to a cohomological irreducible cuspidal automorphic representation $\pi^{(N)}$ of $\mathrm{GL}_N(F_\mathbf{A})$ satisfying (\ref{eq:auto_assump}), and  
let ${\mathcal M}(\pi^{(n+1)} \times \pi^{(n)})$ denote the tensor product ${\mathcal M}[\pi^{(n+1)}] \otimes_F {\mathcal M}[\pi^{(n)}]$.   
Then Clozel's conjecture \cite[Conjecture 4.5]{clo90} implies 
the identity of $L$-functions
\begin{align*}
  \mathcal{L}\left (s+\frac{1}{2}, \pi^{(n+1)} \times \pi^{(n)} \right) 
   =   L(s+n,  {\mathcal M}(\pi^{(n+1)} \times \pi^{(n)}  )   )   \quad \in E\otimes_{\mathbf{Q}}\mathbf{C}
\end{align*}
(see also Section \ref{sec:CloConj}). 

\begin{dfn} \label{dfn:good}
 Let $\boldsymbol{\mu}^{(n)}=(\mu^{(n)}_\tau)_{\tau\in I_F}\in (\mathbf{Z}^n)^{I_F}$ be a pure dominant weight of purity weight $w^{(n)}=w(\boldsymbol{\mu}^{(n)})$ (see Section~\ref{sec:FinDim}). 
 We say $\boldsymbol{\mu}^{(n)}$ to be {\em in a good position} if the inequality 
\begin{align} \label{eq:good}
\min\{\mu_{\tau,i}^{(n)}, w^{(n)}-\mu_{\tau,n+1-i}^{(n)}\} \geq \max \{\mu_{\tau,i+1}^{(n)}, w^{(n)}-\mu_{\tau,n-i}^{(n)}\} 
\end{align}
holds for every $\tau \in I_F$ and $i=1,2,\dotsc,n$. 
\end{dfn}

Note that, for a pure dominant weight $\boldsymbol{\mu}^{(n)}$, the condition (\ref{eq:good}) is always fulfilled at real $\tau\in I_F$. 

\begin{lem} \label{lem:good}
 Let $n$ be a positive integer and $\boldsymbol{\mu}^{(n+1)}\in (\mathbf{Z}^{n+1})^{I_F}$ a pure dominant weight in a good position. When $F$ is totally real, assume that the purity weight $w^{(n+1)}$ of $\boldsymbol{\mu}^{(n+1)}$ is even. Then there exists a pure dominant weight $\boldsymbol{\mu}^{(n)}\in (\mathbf{Z}^n)^{I_F}$ in a good position, whose purity weight $w^{(n)}$ has the same parity with $w^{(n+1)}$, such that $\boldsymbol{\mu}^{(n+1)}$ and $\boldsymbol{\mu}^{(n)}$ satisfy the strong interlace condition $($see Definition~$\ref{dfn:intlauto})$.
\end{lem}

\begin{proof}
First consider the CM field case. Let $\tau\in I_F$ be a complex embedding. Then, by (\ref{eq:good}), there exists an integer $\mu_{\tau,i}^{(n)}$ for each $i=1,2,\dotsc,n$ satisfying 
\begin{align} \label{eq:goood}
 -\max\{\mu_{\tau,n+2-i}^{(n+1)}, w^{(n+1)}-\mu_{\tau,i}^{(n+1)}\} \geq \mu^{(n)}_{\tau,i} \geq  -\min\{\mu_{\tau,n+1-i}^{(n+1)}, w^{(n+1)}-\mu_{\tau,i+1}^{(n+1)}\}.
\end{align}
Set $w^{(n)}=-w^{(n+1)}$ and $\mu_{\rho\tau,i}^{(n)}=w^{(n)}-\mu_{\tau,n+1-i}^{(n)}$ for $i=1,2,\dotsc,n$. Then the inequalities (\ref{eq:goood}) imply
\begin{align*}
 -w^{(n+1)}+\min\{\mu^{(n+1)}_{\tau,i},w^{(n+1)}-\mu^{(n+1)}_{\tau,n+2-i}\}\geq w^{(n)}-\mu_{\tau,n+1-i}^{(n)} \geq -w^{(n+1)}+\max\{\mu^{(n+1)}_{\tau,i+1},w^{(n+1)}-\mu^{(n+1)}_{\tau,n+1-i}\}
\end{align*}
and we see that $\mu_{\rho\tau,i}^{(n)}$ also satisfies the inequalities (\ref{eq:goood}) for $i=1,2,\dotsc,n$. 

Next we consider the totally real field case. Let $\tau\in I_F$ be a real embedding. Since purity implies $\mu_{\tau,i}^{(n+1)}=w^{(n+1)}-\mu^{(n+1)}_{\tau,n+2-i}$ for every $i=1,2,\dotsc,n+1$, there exists an integer $\mu_{\tau,i}^{(n)}$ for each $i=1,2,\dotsc,\lfloor \frac{n}{2}\rfloor$ satisfying 
\begin{align} \label{eq:gooood}
-\mu_{\tau,n+2-i}^{(n+1)}\geq \mu^{(n)}_i\geq -\mu^{(n+1)}_{\tau,n+1-i}  
\end{align}
 due to (\ref{eq:good}). Set $w^{(n)}=-w^{(n+1)}$ and, for $i=\lfloor \frac{n}{2}\rfloor+1,\dotsc,n$, define $\mu^{(n)}_{\tau,i}$ to be
\begin{align*}
 \mu_{\tau,i}^{(n)}=
\begin{cases}
 \frac{w^{(n)}}{2}=-\frac{w^{(n+1)}}{2} & \text{for $i=\lfloor \frac{n}{2}\rfloor+1$ if $n$ is odd}, \\
w^{(n)}-\mu_{\tau,n+1-i}^{(n)}=-w^{(n+1)}-\mu^{(n)}_{\tau,n+1-i} & \text{otherwise}.
\end{cases}
\end{align*} 
Then, using the purity condition and (\ref{eq:gooood}), we readily verify that the inequalities (\ref{eq:gooood}) hold even for $i= \lfloor \frac{n}{2}\rfloor+1,\dotsc,n$. 

In both cases, the weight $\boldsymbol{\mu}^{(n)}=(\mu^{(n)}_{\tau,i})_{\tau\in I_F, 1\leq i\leq n}$ is the desired one; 
obviously $w^{(n)}=-w^{(n+1)}$ has the same parity as $w^{(n)}$, and dominancy and purity of $\boldsymbol{\mu}^{(n+1)}$ easily follow by construction.
\end{proof}

\begin{rem} \label{rem:parity}
 When $F$ is totally real and $n=2m+1$ is odd, the purity condition for $\boldsymbol{\mu}^{(n)}$ implies that $\mu^{(n)}_{\tau,m+1}$ has to satisfy $\mu^{(n)}_{\tau,m+1}=w^{(n)}/2=-w^{(n+1)}/2$. This is the only reason why we assume that $w^{(n+1)}$ is even in Lemma~\ref{lem:good} when $F$ is totally real. We note that this constraint can be weakened under some good situations. For example, suppose that $F$ is totally real, $n=2m+1$ is odd and $\boldsymbol{\mu}^{(n+1)}$ is a pure dominant weight in a good position, which has an odd purity weight but is  {\em strictly dominant}, that is, the inequalities
\begin{align*}
 \mu^{(n+1)}_{\tau,n+1}>\mu^{(n+1)}_{\tau,n}>\cdots>\mu^{(n+1)}_{\tau,2}>\mu^{(n+1)}_{\tau,1}
\end{align*}
hold for every $\tau\in I_F$. Then if, for each $\tau\in I_F$, we choose $\mu^{(n)}_{\tau,i}$ satisfying $-\mu^{(n+1)}_{\tau,n+1-i}<\mu^{(n)}_{\tau,i}\leq -\mu^{(n+1)}_{\tau,n+2-i}$ for $i=1,2,\dotsc,m$ and set 
\begin{align*}
 \mu^{(n)}_{\tau,i}=
\begin{cases}
 \frac{1-w^{(n+1)}}{2} & \text{for }i=m+1, \\
 1-w^{(n+1)}-\mu^{(n)}_{\tau,n+1-i} &  \text{for } i=m+2,\dotsc,n,
\end{cases}
\end{align*}
the resulted dominant weight $\boldsymbol{\mu}^{(n)}=(\mu^{(n)}_{\tau,i})_{\tau\in I_F, 1\leq i\leq n}$ is pure of weight $w^{(n)}=1-w^{(n+1)}$ such that $\boldsymbol{\mu}^{(n+1)}$ and $\boldsymbol{\mu}^{(n)}$ satisfy the (strong) interlace condition. 
\end{rem}

Now let us introduce the main theorem in the present article. For a CM field $F$ and a cohomological irreducible cuspidal automorphic representation $\pi^{(n)}$ of $\mathrm{GL}_n(F_{\mathbf{A}})$, let $\pi^{(n),\rho}=(\varrho_{\pi^{(n),\rho}},V_{\pi^{(n),\rho}})$ denote the $\rho$-twist of $\pi$ defined as $\varrho_{\pi^{(n),\rho}}(g)=\varrho_{\pi^{(n)}}(g^\rho)$. The motive $\mathcal{M}[\pi^{(n),\rho}]$ attached to $\pi^{(n),\rho}$ then coincides with $\mathcal{M}[\pi^{(n)}]^{\rho}$. Note that the strong interlace condition (see Definition~\ref{dfn:intlauto} (2)) for  $m_0\in \mathbf{Z}$ implies that $m_0$ is a critical point not only for $L(s,\pi^{(n+1)}\times \pi^{(n)})$ but also for $L(s,\pi^{(n+1)}\times \pi^{(n),\rho})$ (or $L(s,\pi^{(n+1),\rho}\times \pi^{(n)})$).  Recall that the constant $  \lambda({\mathcal M}[\pi^{(n)}])$ in Proposition \ref{prop:perprod} is defined to be 
\begin{align*}
   \lambda( {\mathcal M}[\pi^{(n)}] ) 
   =  \sum_{\tau \in I_F } \sum^{ n  }_{i=1} p^{( n )}_{\tau, i} ( n -i),  
\end{align*}
where $p^{( n )}_{\tau, i}$ is given in Remark \ref{rem:Hodgepq}.


\begin{thm}\label{thm:main}
Let $F$ be a totally real number field or a CM field, $E$ a sufficiently large coefficient field, and  $\pi^{(n)}$ a cohomological irreducible cuspidal automorphic representation of $\mathrm{GL}_n(F_\mathbf{A})$ satisfying $(\ref{eq:auto_assump})$. Suppose that the highest weight $\boldsymbol{\mu}^{(n)}_\pi$ attached to $\pi^{(n)}$ is in a good position 
         and the purity weight $w^{(n)}=w(\pi^{(n)})$ is even. 
Further suppose that all the following conditions are fulfilled$:$ 
\begin{itemize}
\item[--]  there exist pure motives associated with cohomological irreducible cuspidal automorphic representations of ${\rm GL}_n(F_{\mathbf A})$ $($see \cite[Conjecture 4.5]{clo90}$);$  
\item[--] Deligne's conjecture on critical values \cite[Conjecture 1.8]{del79} holds for Rankin--Selberg $L$-functions $L(s, \pi^{(n+1)}\times \pi^{(n)});$ 
\item[--] Conjecture \ref{conj:mellin} holds.
\end{itemize}
Then we have 
\begin{align*}
     \boldsymbol{p}(\pi^{(n)}, \varepsilon^{(n)})
     \sim 
     (2\pi\sqrt{-1})^{ \lambda(  {\mathcal M}[\pi^{(n)}]  )    }   \widetilde{c}(  {\mathcal M}[\pi^{(n)}]  )    
              \times 
         \begin{cases}  1 & \text{if $n$ is odd},  \\
                                c^{-\varepsilon^{(n)}}(\mathrm{Res}_{F/\mathbf{Q}}(\mathcal{M}[\pi^{(n)}]))  & \text{if $n$ is even}.
          \end{cases}
\end{align*}
\end{thm}

\begin{proof}
We prove the statement by induction on $n$. Let $\pi^{(n+1)}$ be an irreducible cohomological cuspidal automorphic representation of ${\rm GL}_{n+1}(F_{\mathbf A})$ of highest weight $\boldsymbol{\mu}^{(n+1)}_\pi$. By Lemma~\ref{lem:good}, there exists a pure dominant weight $\boldsymbol{\mu}^{(n)}\in (\mathbf{Z}^n)^{I_F}$ in a good position, whose purity weight has the same parity with that of $\boldsymbol{\mu}^{(n+1)}_\pi$, such that $\boldsymbol{\mu}^{(n+1)}_\pi$ and $\boldsymbol{\mu}^{(n)}$ satisfy the strong interlace condition. 
Take a cohomological irreducible cuspidal automorphic representation $\pi^{(n)}$ of $\mathrm{GL}_n(F_{\mathbf{A}})$ of highest weight $\boldsymbol{\mu}^{(n)}$.   
Such a $\pi^{(n)}$ actually exists according to \cite[Theorem 2.10]{br17}. 
                   Note that, if $F$ is a CM field, the purity weight $w^{(n)}$ of $\boldsymbol{\mu}^{(n)}$ must be even to apply \cite[Theorem 2.10]{br17}.
 However, this is satisfied since  $w^{(n+1)}$ is even by assumption and the purity weight $w^{(n)}$ of $\boldsymbol{\mu}^{(n)}$ is chosen so that it has the same parity as $w^{(n+1)}$. 

We first consider the case where $F$ is a totally real field. 
Then Deligne's conjecture on critical values of Rankin--Selberg $L$-functions, combined with Theorem \ref{thm:AlgRS}, implies
\begin{align*}
\boldsymbol{p}&(\pi^{(n+1)}, \varepsilon^{(n+1)})  \boldsymbol{p}(\pi^{(n)}, \varepsilon^{(n)}) 
   \sim \mathcal{L}\left(m+\frac{1}{2},\pi^{(n+1)} \times \pi^{(n)} \right) \\
&\quad   \sim L_\infty(m+n,  {\rm Res}_{F/{\mathbf Q}}(\mathcal{M}(\pi^{(n+1)}\times \pi^{(n)}) )) c^+\left(   {\rm Res}_{F/{\mathbf Q}}( \mathcal{M}(\pi^{(n+1)}\times \pi^{(n)})  )   (m+n)  \right)
\end{align*}
for each $m\in {\rm Crit}(\pi^{(n+1)}, \pi^{(n)} )$ and signatures $\varepsilon^{(n+1)}$, $\varepsilon^{(n)}$ satisfying $\varepsilon^{(n+1)} \varepsilon^{(n)} = (-1)^{m+n+1}$.
Therefore Proposition \ref{prop:perprod} yields that 
\begin{align*}
               \boldsymbol{p}&(\pi^{(n+1)}, \varepsilon^{(n+1)})  \boldsymbol{p}(\pi^{(n)}, \varepsilon^{(n)})     \\
  & \quad \sim    (2\pi\sqrt{-1})^{\lambda(  {\mathcal M}[\pi^{(n+1)}]  )  }   \widetilde{c}(  {\mathcal M}[\pi^{(n+1)}]  ) 
               \cdot (2\pi\sqrt{-1})^{ \lambda(  {\mathcal M}[\pi^{(n)}]  )    }   \widetilde{c}(  {\mathcal M}[\pi^{(n)}]  )            \\
           & \hspace*{7em}  \cdot \begin{cases}   c^{ (-1)^{m+n}   {\varepsilon}^{(n)}    }(  {\rm Res}_{F/{\mathbf Q}}(   {\mathcal M}[\pi^{(n+1)}]   )  )  &  \text{if $n$ id odd},      \\
                                          c^{  (-1)^{m+n}  {\varepsilon}^{(  n+1  )}   }(   {\rm Res}_{F/{\mathbf Q}}(  {\mathcal M}[\pi^{(n)}]   )   )  &  \text{if $n$ is even}.    \end{cases}
\end{align*}
Then by the sign condition $\varepsilon^{(n+1)} \varepsilon^{(n)} = (-1)^{m+n+1}$, 
the above period relation is rewritten as
\begin{align}\label{eq:perprodpp}
\begin{aligned}
               \boldsymbol{p}&(\pi^{(n+1)}, \varepsilon^{(n+1)})  \boldsymbol{p}(\pi^{(n)}, \varepsilon^{(n)})      \\
  & \qquad \sim   (2\pi\sqrt{-1})^{\lambda(  {\mathcal M}[\pi^{(n+1)}]  )  }   \widetilde{c}(  {\mathcal M}[\pi^{(n+1)}]  )    
               \cdot (2\pi\sqrt{-1})^{ \lambda(  {\mathcal M}[\pi^{(n)}]  )    }   \widetilde{c}(  {\mathcal M}[\pi^{(n)}]  )         \\
           & \hspace*{7em}   \cdot \begin{cases}   c^{ -   {\varepsilon}^{(  n+1  )}     }(  {\rm Res}_{F/{\mathbf Q}}(   {\mathcal M}[\pi^{(n+1)}]   )  )  &  \text{if $n$ is odd},      \\
                                          c^{  -  {\varepsilon}^{(  n  )}   }(   {\rm Res}_{F/{\mathbf Q}}(  {\mathcal M}[\pi^{(n)}]   )   ),  &  \text{if $n$ is even}.    \end{cases}
\end{aligned}
\end{align}

By definition, we have
\begin{align*}
\boldsymbol{p}(\pi^{(1)}, \varepsilon^{(1)}  ) &\sim  1,  &
   \lambda(  {\mathcal M}[\pi^{(1)}]  )   &=  \lambda(  {\mathcal M}[\pi^{(2)}]  )  = 0,  &
   \widetilde{c}(  {\mathcal M}[\pi^{(1)}]  )
               \sim \widetilde{c}(  {\mathcal M}[\pi^{(2)}]  )
               \sim 1.
\end{align*}
Substituting these identities and $n=1$ into the equation (\ref{eq:perprodpp}), 
we obtain 
\begin{align}\label{eq:Per2}
   \boldsymbol{p}(\pi^{(2)}, \varepsilon^{(2)}  ) \sim c^{  -  {\varepsilon}^{(  2  )}   }(   {\rm Res}_{F/{\mathbf Q}}(  {\mathcal M}[\pi^{(2)}]   )   ).
\end{align}
Substituting $n=2$ and (\ref{eq:Per2}) into (\ref{eq:perprodpp}), 
we obtain 
\begin{align*}
   \boldsymbol{p}(\pi^{(3)}, \varepsilon^{(3)}  ) 
   \sim (2\pi \sqrt{-1})^{\lambda(\mathcal{M}[\pi^{(3)}])} \widetilde{c} ({\mathcal M}[\pi^{(3)}]).    
\end{align*}
Hence, by induction $n$, the statement follows when $F$ is totally real.

Next consider the case where $F$ is a CM field. Recall that, in this case, we set $\varepsilon^{(n+1)}=\varepsilon^{(n)}=1$ as convention. Again, by Deligne's conjecture for Rankin--Selberg $L$-functions and Theorem \ref{thm:AlgRS}, we have 
\begin{align*}
\boldsymbol{p}&(\pi^{(n+1)}, \varepsilon^{(n+1)})  \boldsymbol{p}(\pi^{(n)}, \varepsilon^{(n)})\boldsymbol{p}(\pi^{(n+1)}, \varepsilon^{(n+1)})  \boldsymbol{p}(\pi^{(n), \rho}, \varepsilon^{(n)})   \\
&\quad \sim \mathcal{L}\left(m+\frac{1}{2},\pi^{(n+1)}\times \pi^{(n)}\right) \mathcal{L}\left( m+\frac{1}{2},\pi^{(n+1)}\times \pi^{(n),\rho}\right) \\
&\quad   \sim L_\infty(m+n,  {\rm Res}_{F^+/{\mathbf Q}}(  \mathrm{Res}_{F/F^+}(  {\mathcal M}[\pi^{(n+1)}])  \otimes_{F^+} \mathrm{Res}_{F/F^+}(\mathcal{M}[\pi^{(n)}]  )  ) ) \\
&\hspace*{7em}
        \cdot c^+\left(   {\rm Res}_{F^+/{\mathbf Q}}(  \mathrm{Res}_{F/F^+}(  {\mathcal M}[\pi^{(n+1)}])  \otimes_{F^+} \mathrm{Res}_{F/F^+}(\mathcal{M}[\pi^{(n)}]  )  )   (m+n)  \right)
\end{align*}
for each $m\in {\rm Crit}(\pi^{(n+1)}, \pi^{(n)} )\cap \mathrm{Crit}(\pi^{(n+1)},\pi^{(n),\rho})$.
Therefore Proposition \ref{prop:perprod} implies
\begin{align*}
               \boldsymbol{p}&(\pi^{(n+1)}, \varepsilon^{(n+1)})  \boldsymbol{p}(\pi^{(n)}, \varepsilon^{(n)})       \boldsymbol{p}(\pi^{(n+1)}, \varepsilon^{(n+1)})  \boldsymbol{p}(\pi^{(n),\rho}, \varepsilon^{(n)})       \\
  & \quad \sim    (2\pi\sqrt{-1})^{2\lambda(  {\mathcal M}[\pi^{(n+1)}]  )  }   \widetilde{c}(  {\mathcal M}[\pi^{(n+1)}]  )^2    
               \cdot (2\pi\sqrt{-1})^{ 2\lambda(  {\mathcal M}[\pi^{(n)}]  )    }   \widetilde{c}(  {\mathcal M}[\pi^{(n)}]  )^2            \\
           & \hspace*{7em}  \cdot \begin{cases}   c^{ -   {\varepsilon}^{(n+1)}    }(  {\rm Res}_{F/{\mathbf Q}}(   {\mathcal M}[\pi^{(n+1)}]   )  )^2  &  \text{if $n$ is odd},      \\
                                          c^{  -{\varepsilon}^{(  n  )}   }(   {\rm Res}_{F/{\mathbf Q}}(  {\mathcal M}[\pi^{(n)}]   )   )^2  &  \text{if $n$ is even}    \end{cases}
\end{align*}(note that  $c^{+}(\mathrm{Res}_{F/\mathbf{Q}}(\mathcal{M}[\pi^{(N)}]))\sim c^{-}(\mathrm{Res}_{F/\mathbf{Q}}(\mathcal{M}[\pi^{(N)}]))$ holds by the assumption (\ref{eq:auto_assump}) and Lemma~\ref{lem:cpm}).
By definition, we have
\begin{align*}
p(\pi^{(1)}, \varepsilon^{(1)}  ) \sim  1, 
\quad    \lambda(  {\mathcal M}[\pi^{(1)}]  )   =  \lambda(  {\mathcal M}[\pi^{(2)}]  )  = 0,  
\quad     \widetilde{c}(  {\mathcal M}[\pi^{(1)}]  )
               \sim \widetilde{c}(  {\mathcal M}[\pi^{(2)}]  )
               \sim 1.
\end{align*}
Substituting these identities and $n=1$ into the equation (\ref{eq:perprodpp}), 
we obtain 
\begin{align}\label{eq:Per2c}
   \boldsymbol{p}(\pi^{(2)}, \varepsilon^{(2)}  )^2 &\sim c^{  -  {\varepsilon}^{(  2  )}   }(   {\rm Res}_{F/{\mathbf Q}}(  {\mathcal M}[\pi^{(2)}]   )^2 
   &  \therefore \;  \boldsymbol{p}(\pi^{(2)}, \varepsilon^{(2)}  ) &\sim c^{  -  {\varepsilon}^{(  2  )}   }(   {\rm Res}_{F/{\mathbf Q}}(  {\mathcal M}[\pi^{(2)}]   ).
\end{align}
Due to Remark~\ref{rem:n_rho}, we also have 
\begin{align*}
\boldsymbol{p}(\pi^{(2),\rho}, \varepsilon^{(2)})\sim  c^{  -  {\varepsilon}^{(  2  )}   }(   {\rm Res}_{F/{\mathbf Q}}(  {\mathcal M}[\pi^{(2)}]^\rho   ) \sim c^{  -  {\varepsilon}^{(  2  )}   }(   {\rm Res}_{F/{\mathbf Q}}(  {\mathcal M}[\pi^{(2)}]   ) \sim \boldsymbol{p}(\pi^{(2)}, \varepsilon^{(2)}).
\end{align*}
  Then, by substituting $n=2$, we obtain 
\begin{align*}
   \boldsymbol{p}(\pi^{(3)}, \varepsilon^{(3)}  )^2 
   &\sim (2\pi \sqrt{-1})^{2\lambda(\pi^{(3)})} \widetilde{c} ({\mathcal M}[\pi^{(3)}])^2 
   &  \therefore \;  \boldsymbol{p}(\pi^{(3)}, \varepsilon^{(3)}  ) 
   &\sim (2\pi \sqrt{-1})^{\lambda(\pi^{(3)})} \widetilde{c} ({\mathcal M}[\pi^{(3)}]).    
\end{align*}
Hence, again by induction on $n$, the statement follows when $F$ is a CM field.
\end{proof}

\begin{rem}
 When $F$ is a totally real field, we can weaken the assumption on the parity of the purity weight $w^{(n)}$ of $\pi^{(n)}$ under certain good situations, by utilizing variants of Lemma~\ref{lem:good} (see also Remark~\ref{rem:parity}). 
\end{rem}

\end{document}